\documentclass[11pt,reqno]{amsart}
\usepackage[centering,bottom=3.5cm,top=3.4cm]{geometry}

\usepackage[T1]{fontenc} 
\usepackage[utf8]{inputenc}  
\usepackage[english]{babel}

\usepackage{hyperref}
\usepackage[capitalize,nameinlink,sort]{cleveref}
\allowdisplaybreaks

\usepackage{graphicx}
\usepackage[dvipsnames]{xcolor}

\usepackage{todonotes}

\usepackage{tikz}
\usetikzlibrary{arrows, automata, shapes, positioning}
\usetikzlibrary{arrows.meta}
\usetikzlibrary{decorations.pathmorphing}
\usetikzlibrary{calc}

\usepackage{enumitem}

\theoremstyle{plain}
\newtheorem{theorem}{Theorem}
\newtheorem{lemma}[theorem]{Lemma}
\newtheorem{corollary}[theorem]{Corollary}
\newtheorem{proposition}[theorem]{Proposition}
\newtheorem{conjecture}[theorem]{Conjecture}

\theoremstyle{definition}
\newtheorem{definition}[theorem]{Definition}
\newtheorem{example}[theorem]{Example}
\newtheorem{remark}[theorem]{Remark} 
\numberwithin{equation}{section}

\newcommand{\C}{\mathbb C}
\newcommand{\Q}{\mathbb Q}
\newcommand{\Z}{\mathbb Z}
\newcommand{\N}{\mathbb N}

\newcommand{\lex}{\le_{\mathrm{lex}}}
\newcommand{\rep}{\mathrm{rep}}
\newcommand{\val}{\mathrm{val}}
\newcommand{\Max}{\mathrm{Max}}
\newcommand{\Pref}{\mathrm{Pref}}

\title{Numeration systems without a dominant root and regularity}
\author{\'{E}milie Charlier, Savinien Kreczman}
\date{\today}

\begin{document}

\begin{abstract}
Positional numeration systems are a large family of numeration systems used to represent natural numbers. Whether the set of all representations forms a regular language or not is one of the most important questions that can be asked of such a system.
This question was investigated in a 1998 article by Hollander. Central to his analysis is a property linking positional numeration systems and Rényi numeration systems, which use a real base to represent real numbers. However, this link only exists when the initial numeration system has a dominant root, which is not a necessary condition for regularity. 
In this article, we show a more general link between positional numeration systems and alternate base numeration systems, a family generalizing Rényi systems. We then take advantage of this link to provide a full characterization of those numeration systems that generate a regular language. We also discuss the effectiveness of our method, and comment Hollander's results and conjecture in the light of ours.
\end{abstract}

\keywords{Numeration systems, Regular languages, Alternate bases, Linear recurrence sequences}

\maketitle
\setcounter{tocdepth}{1}
\tableofcontents

\section{Introduction}

If we are to compute operations on numbers (which could be integers, real numbers, Gaussian integers, complex numbers, or even $p$-adic numbers), we first need to agree on a way to represent these numbers. The theory of numeration systems examines the different methods that can be used to represent numbers through sequences of symbols. Depending on the tasks we are interested in performing, we can then try to select the most suitable numeration system. A desirable property of the chosen system is that the operations of interest can be performed by some finite state machine.  

In this paper, we consider the representation of non-negative integers via positional numeration systems as initiated by Fraenkel \cite{Fraenkel:1985}. As explained in the nice chapter by Frougny and Sakarovitch \cite{Frougny&Sakarovitch:2010}, perhaps the most fundamental question to be asked is to understand which positional numeration systems support a regular set of representations. That is, do all valid expansions form a regular language in the sense of finite automata? Indeed, the regularity of the set of expansions is a necessary condition to the use of automata-theoretic techniques in relation to numeration systems. We can mention the ability to normalize representations with transducers \cite{Frougny:1992, Frougny&Solomyak:1996}, which enables the computation of addition in the numeration system using finite state machines. In turn, this allows us to use tools such as \texttt{Walnut}\cite{Mousavi:2021,Shallit:2022} to automatically prove results related to the numeration system, following a connection to first-order logic pioneered by B\"{u}chi \cite{Buchi:1990}.

Most of the commonly used numeration systems, even in the non-standard numeration systems community, have the property to generate a regular numeration language. This is the case of all Bertrand numeration systems associated with a Parry number. These define a very constrained family of positional numeration systems since there are only one or two Bertrand numeration systems associated with a given Parry number \cite{Bertrand-Mathis:1989,Charlier&Cisternino&Stipulanti:2022}. All integer base systems as well as the Zeckendorf numeration system \cite{Zeckendorf:1972} based on the Fibonacci sequence are Bertrand numeration systems. Another family with this property is that of Pisot numeration systems, which are the positional numeration systems defined by a linear recurrence relation whose characteristic polynomial is the minimal polynomial of a Pisot number \cite{Frougny&Solomyak:1996,Bruyere&Hansel:1997}. This time, many  Pisot numeration systems are associated with one given Pisot number, since the initial conditions can be more freely chosen. Nevertheless, this family is still very special, in particular because of the irreducibility constraint on the characteristic polynomial. These two families share the property of having a dominant root, meaning that the quotient of consecutive terms in the base sequence admits a limit, which in addition is greater than one. On the other hand, even some simple numeration systems, like the one based on the sequence of squares, do not support a regular set of representations \cite{Shallit:1994}.

As an attempt to unify the theory of numeration systems with a regular numeration language, Lecomte and Rigo defined the family of abstract numeration systems \cite{Lecomte&Rigo:2001,Lecomte&Rigo:2010}. Here, the point of view is reversed. One starts with any given regular language, orders the words in the language with respect to the radix order (i.e., by length first and then using the lexicographic order within each length) and declares that a non-negative integer $n$ is represented by the $n$-th word in the language. This framework encompasses the previously mentioned families of numeration systems as well as others, such as the Dumont-Thomas numeration systems built from a substitution \cite{Dumont&Thomas:1989}. The price of this very general point of view is that we lose the information of the algorithm to directly produce the representation of a number. This raises the question of determining which regular languages can be produced by a positional numeration system. This problem was studied in \cite{Kreczman&Labbe&Stipulanti:2025} for the family of Dumont-Thomas numeration systems.

The question of characterizing positional systems generating a regular numeration language was addressed by Hollander in the case where the numeration system satisfies the dominant root condition \cite{Hollander:1998}. In his study, Hollander provided a detailed, however not full, description of systems with a dominant root and giving rise to a regular set of representations. In particular, he does not entirely solves the issue of the dependence to the initial conditions. Central to Hollander's arguments is a property \cite[Section 4]{Hollander:1998} linking positional numeration systems to the representation of real numbers via a real base \cite{Renyi:1957,Parry:1960,Schmidt:1980,Bertrand-Mathis:1986, Dajani&DeVries&Komornik&Loreti:2011}. Whenever the dominant root condition is dropped, this link disappears and hence new tools are needed in order to attack the general problem. 

The aim of this paper is to provide a complete characterization of positional numeration systems that yield a regular numeration language. We solve this problem in its full generality, not adding any particular condition on the given positional numeration system. Our study will lead us to deal with systems that are not associated with a real base as was the case in Hollander’s work. However, we will see that systems with a regular numeration language are always associated with a very specific way to represent real numbers, namely by using an alternate base of real numbers as introduced in \cite{Charlier&Cisternino:2021}. This observation was in fact the original motivation for the study of these new representations of real numbers. In particular, the notion of Parry alternate base developed in \cite{Charlier&Cisternino&Masakova&Pelantova:2023} will be central in the present work since we will see that this condition is necessary in order to ensure the regularity of the numeration language of a positional numeration system.

Our main result will be separated into four parts, each handling a different behavior of the greedy algorithm for the representation of $1$ via alternate real bases. Together, these results yield a semi-decision procedure for deciding the regularity of the numeration language of a positional numeration system. Testing the ultimate periodicity of the greedy expansions of $1$ in an alternate base (in fact, even in a single real base) is a difficult task, that is not known to be decidable in general. But, provided that the associated greedy expansions of $1$ are known to be finite or ultimately periodic, which we can indeed check in many situations, our results actually provide us with a genuine decision procedure. 

The article will be structured as follows. In \cref{sec:positional-systems}, we recall the basics of positional numeration systems and introduce our notation. In \cref{sec:restriction-linear} we show that we can restrict our study to the case where the base sequence is a linear recurrence sequence. In \cref{sec:max-regular}, we explain how to restrict our study to the thin sub-language of the lexicographically maximal words of each length. We explain how alternate bases naturally arise and recall the related notions in \cref{sec:alternate-bases}, then show how they induce four different types of behavior within the numeration language in \cref{sec:graph}. We prove the link between lexicographically maximal words of each length and the associated alternate base in \cref{sec:Parry}, allowing a further restriction to Parry alternate bases only. With the preparatory work done, \cref{sec:strategy} details our strategy for proving the announced criteria. \cref{sec:i-infini,sec:i-vers-infini,sec:i-cycle,sec:i-vers-cycle} contain the bulk of the work, carefully analyzing each of the four cases outlined above and obtaining a criterion for each of them. In these four sections, we take care of positioning our results in comparison to those obtained by Hollander, when we restrict our hypotheses to the case of a dominant root: either we recover the previously known results, or we handle a new situation, not occurring in the case of a dominant root. In particular, our results allow the possibility of a base to be equal to $1$. In the dominant root case, this corresponds to the polynomial case, which was not treated in \cite{Hollander:1998}. \cref{sec:decision} summarizes our results and comments on their effectiveness. We provide two carefully designed examples in order to illustrate different scenarios. Finally, \cref{sec:MemePolGraphesDiffs,sec:comments-Hollander} answer Hollander's conjecture in the negative and explain why similar conjectures are unlikely to be true in the non-dominant root case, concluding the paper.

\section{Positional numeration systems}\label{sec:positional-systems}
A \emph{positional numeration system} is given by an increasing sequence $U=(U_n)_{n\ge 0}$ of integers 
such that $U_0=1$ and the quotients $\frac{U_{n+1}}{U_n}$ are uniformly bounded. The \emph{value} of a word $w_{\ell-1}\cdots w_0$ over $\N$ in the system $U$ is written
\[
    \val_U(w_{\ell-1}\cdots w_0)=\sum_{n=0}^{\ell-1} w_n U_n.
\] 
For a given $x\in\N$, any word $w$ with letters in $\N$ such that $\val_U(w)=x$ is said to be a \emph{$U$-representation} of $x$. Such a representation need not be unique. 

Among all possible $U$-representations of $x$, we will consider the one obtained using the greedy algorithm. First, we let $\ell$ be the least integer such that $x<U_\ell$ and we let $r_\ell=x$. Then for every $n=\ell-1,\ldots,0$, we set $a_n=\lfloor \frac{r_{n+1}}{U_n}\rfloor$ and $r_n=r_{n+1}-a_nU_n$. The produced $U$-representation is 
\[
    \rep_U(x):=a_{\ell-1}\cdots a_0,
\]
which is called the \emph{greedy} $U$-representation of $x$. 

The language of all greedy $U$-representations, possibly preceded by zeros, i.e., the language
\[
    L_U:=0^*\rep_U(\N),
\]
is called  the \emph{numeration language}. It is written over the alphabet 
\[
    A_U:=\{0,\ldots,\sup_{n\ge 0} \left\lceil \frac{U_{n+1}}{U_n}\right\rceil - 1\},
\]
called the \emph{numeration alphabet}.

Our aim is to characterize the positional numeration systems giving rise to a regular numeration language. Note that the languages $L_U$ and $\rep_U(\N)$ are simultaneously regular. However, allowing leading zeros will often be more convenient for our future developments. In particular, a word $w_{\ell-1}\cdots w_0\in A_U^*$ belongs to $L_U$ if and only if 
\begin{equation}
    \label{eq:val-greedy}
    \val_U(w_{n-1}\cdots w_0)< U_n 
\end{equation}
for all $n\in\{0,\ldots,\ell\}$. This also implies that the language $L_U$ is suffix-closed.

\section{First step: Restriction to linear numeration systems}\label{sec:restriction-linear}
The following well-known result -- see for instance \cite{Shallit:1994,Loraud:1995} -- asserts that we can restrict our study to \emph{linear numeration systems}, which are positional numeration systems with the additional condition that the base sequence $U=(U_n)_{n\ge 0}$ is \emph{linear over $\Z$}: there exists integers $c_0,\ldots,c_{m-1}$ such that
\begin{equation}
    \label{eq:lin-rec}
    U_{n+m}=c_{m-1}U_{n+m-1}+\ldots+c_0 U_n
\end{equation}
for all $n \geq 0$. The polynomial
\[
    X^m-c_{m-1}X^{m-1}-c_2X^{m-2}-\cdots - c_0
\]
is called the \emph{characteristic polynomial} of the linear recurrence relation \eqref{eq:lin-rec}. 

\begin{proposition}
\label{prop:reg->linear}
Any positional numeration system with a regular numeration language is linear.
\end{proposition}

Let us recall some background on linear sequences that will be useful in our future analysis. Given a sequence of rational numbers $U=(U_n)_{n\ge 0}$, the set of polynomials
\[
    \{c_mX^m+c_{m-1}X^{m-1}+\cdots+c_0\in\Q[X]: \forall n\ge 0,\ c_m U_{n+m}+c_{m-1}U_{n+m-1}+\cdots+c_0 U_n=0\}
\]
is an ideal $I_U$ of the ring of polynomials $\mathbb{Q}[X]$. The monic generator of this ideal is called the \emph{minimal polynomial} of $U$. Its roots are called the \emph{eigenvalues} of $U$. The \emph{multiplicity} of an eigenvalue of $U$ is its multiplicity as a root of the minimal polynomial of~$U$. Now, if $U$ is an integer sequence that is linear over $\Z$, i.e., if the ideal $I_U$ contains a monic polynomial with integer coefficients, then the minimal polynomial of $U$ also has integer coefficients. We refer to \cite{Berstel&Reutenauer:2011} for more details on linear recurrence sequences.

\section{Second step: Restriction to lexicographically maximal words of each length}
\label{sec:max-regular}

In this section, we show that we can focus on the regularity of the language of \emph{maximal words} of the numeration language. The language of maximal words of a language $L$ over a totally ordered alphabet is the language
\[
    \mathrm{Max}(L) := \{u\in L :  \text{for all }v\in L,\ |v|=|u|\implies v\lex u\}, 
\]
where $\lex$ is the lexicographic order. That is, $\mathrm{Max}(L)$ is the language obtained by extracting the lexicographically greatest word of each length present in $L$. In our case, we may present an alternative characterization of this language. 

First, let us recall the following well-known property of greedy $U$-representations; see for example~\cite[Chapter 7]{Lothaire:2002}. It uses the so-called \emph{radix order} $\le_{\mathrm{rad}}$ where words are ordered length by length first, and then lexicographically: $u<_{\mathrm{rad}} v$ if either $|u|<|v|$, or $|u|=|v|$ and $u<_{\mathrm{lex}} v$. 

\begin{lemma}
\label{lem:Rep-Val-RadixOrder}
Let $U$ be a positional numeration system, let $x\in\N$ and let $w\in A_U^*\setminus 0A_U^*$. If $\rep_U(x)<_{\mathrm{rad}} w$ then $x<\val_U(w)$. 
\end{lemma}

This allows us to obtain the following description of the maximal words in $L_U$.

\begin{lemma}
\label{lem:Max-Words}
Let $U$ be a positional numeration system. 
Then 
\[
    \mathrm{Max}(L_U)=\{\rep_U(U_n-1):n\in\N\}.
\]
\end{lemma}

\begin{proof}
    Clearly, the word $\rep_U(U_n-1)$ belongs to $L_U$ for all $n$. From the greedy algorithm, we know that $|\rep_U(U_n-1)|=n$ and that every word $w$ in $L_U$ of length $n$ is such that $\val_U(w)< U_n$. If $w$ starts with $0$, then $w<_{\mathrm{lex}} \rep_U(U_n-1)$. Otherwise, we get $w\lex \rep_U(U_n-1)$ by \cref{lem:Rep-Val-RadixOrder}.
\end{proof}

The language $\mathrm{Max}(L_U)$ is of great interest to us because it encompasses the regularity of the whole numeration language $L_U$ as shown by the following proposition.

\begin{proposition}
\label{prop:RegIffMaxReg}\ 
\smallskip

\begin{itemize}\setlength\itemsep{0.7em}
    \item If $L$ is a regular language, then the language $\mathrm{Max}(L)$ is regular.
    \item If $U$ is a positional numeration system such that $\mathrm{Max}(L_U)$ is regular, then the numeration language $L_U$ is also regular.
\end{itemize}
\end{proposition}

For the sake of completeness, we reproduce Shallit's proof of the first item from \cite{Shallit:1994}, and give a new construction of an automaton accepting a numeration language built from the knowledge of the lexicographically maximal words (provided that they form a regular language). Another construction of such an automaton was proposed in~\cite{Hollander:1998}, but the resulting automaton contains far more states as the one we describe here, since it is obtained as a product of several automata. 

Also, note that for an arbitrary language $L$, it is not the case that the regularity of the language $\mathrm{Max}(L)$ implies that of $L$. For example, the Dyck language of well-parenthesized binary words is well known to be non-regular whereas its language of maximal words with the order $0<1$ is $(01)^*$, hence is regular. Hence, the hypothesis that the language is derived from a numeration system is important, as such languages have some structure with respect to the lexicographic order. 

We first recall two useful lemmas. The first one, already in \cite{Hollander:1998}, gives a characterization of the words in $L_U$ in terms of a lexicographic condition of their suffixes.

\begin{lemma}
\label{lem:Suffixes}
Let $U$ be a positional numeration system. A word $w_{\ell-1}\cdots w_0\in A_U^*$ belongs to $L_U$ if and only if 
\[
    w_{n-1}\cdots w_0\lex \rep_U(U_n-1) 
\]
for all $n\in\{0,\ldots,\ell\}$.
\end{lemma}

\begin{proof}
The necessary condition follows from the fact that $L_U$ is suffix-closed and \cref{lem:Max-Words}. The sufficient condition is obtained by using~\eqref{eq:val-greedy} and an induction on the length of the words.
\end{proof}

The second lemma, proved in \cite{Shallit:1994}, provides a useful decomposition of \emph{slender} regular languages, i.e., containing a bounded number of words of each length. 

\begin{lemma}
\label{lem:xy*z}
A language is slender and regular if and only if it is a finite union of languages of the form $xy^*z$.
\end{lemma}

\begin{proof}[Proof of~\cref{prop:RegIffMaxReg}]
    We reproduce Shallit's simple argument from \cite{Shallit:1994} for proving the regularity of the language of lexicographically maximal words of each length of a regular language. Let $L$ be an arbitrary regular language written over a totally ordered alphabet $(A,<)$. In order to show that $\mathrm{Max}(L)$ is also regular, we can equivalently show that $L\setminus \mathrm{Max}(L)$ is regular. Let $\mathcal{A}=(Q,i,F,A,\delta)$ be a deterministic finite automaton accepting $L$. Consider now the following non-deterministic finite automaton: $\mathcal{B}=(Q\times Q\times \{e,\ell\},(i,i,e),F\times F\times \{\ell\},A,R)$, where the transition relation $R$ contains the following transitions:
    \begin{itemize}
        \item $(p,q,e)\overset{a}{\longrightarrow} (\delta(p,a),\delta(q,a),e)$ for all $p,q\in Q$ and $a\in A$;
        \item $(p,q,e)\overset{a}{\longrightarrow} (\delta(p,a),\delta(q,b),\ell)$ for all $p,q\in Q$ and $a,b\in A$ with $a<b$;
        \item $(p,q,\ell)\overset{a}{\longrightarrow} (\delta(p,a),\delta(q,b),\ell)$ for all $p,q\in Q$ and $a,b\in A$.
    \end{itemize}
    Transitions in the second component of $Q$ nondeterministically guess a word of $L$ that is greater than the word read by $\mathcal{B}$. As a result, a given word is accepted by $\mathcal{B}$ if and only if it belongs to $L$ and there exists a lexicographically greater word in $L$ of the same length: $\mathcal{B}$ accepts exactly the words in $L\setminus \mathrm{Max}(L)$.

    \medskip
    Let us now turn to the second item. The proof is constructive, and will be illustrated in \cref{Ex:Automaton} below.
    Let $U$ be a positional numeration system and suppose that the lexicographically maximal words of each length in $L_U$ form a regular language. \cref{lem:xy*z} combined with elementary arithmetical considerations implies that this language can be decomposed in the following way:
    \[
        \Max(L_U)=F\cup \left(\bigcup_{j=0}^{d-1}x_jy_j^*z_j\right)
    \]
    where the unions are disjoint, $F$ is a finite language, $d\ge 1$, and for all $j$, we have $x_j,y_j,z_j\in A_U^*$, $x_j\ne \varepsilon$, $|y_j|=d$, $|x_jz_j|\equiv j\pmod d$, $|x_{j+1}z_{j+1}|=|x_jz_j|+1$, and $y_j$ and $z_j$ do not share the same first letter. Let $n=|x_0z_0|$, so that we have $F=\{\rep_{U}(U_{\ell}-1) : \ell< n\}$ and $|x_jz_j|< n+d$ for all $j\in\{0,\ldots,d-1\}$. 
    
    \smallskip
    We now construct an NFA $\mathcal{A}$ as follows. 
    For each $j\in\{0,\ldots,d-1\}$, we consider two finite automata $P_j$ and $S_j$, each accepting finitely many words:
    \begin{itemize}
        \item $P_j$ accepts the words $w\in L_U$ such that $|w|< n$ and $|w|\equiv j\pmod d$.
        \item $S_j$ accepts the words $w\in L_U$ such that $|w|=|z_j|$ and $w\lex z_j$.
    \end{itemize}
    We add the states $(j,k)$ for all $j\in\{0,\ldots,d-1\}$ and $k\in\{0,\ldots,|x_jy_j|-1\}$, and the following transitions, with $x_jy_j=a_{j,1}\cdots a_{j,|x_jy_j|}$ where each $a_{j,k}$ is a letter:
    \begin{itemize}
        \item $(j,k)\xrightarrow{a_{j,k+1}} (j,k+1)$ for all $k\in \{0,\ldots,|x_jy_j|-2\}$;
        \item $(j,|x_jy_j|-1)\xrightarrow{a_{j,|x_jy_j|}} (j,|x_j|)$;
        \item $(j,k)\overset{a}{\longrightarrow} ((j-k-1)\bmod d,0)$ for all $k\in \{0,\ldots,|x_jy_j|-1\}$ and $a<a_{j,k+1}$.
    \end{itemize}
    Finally, for each $j\in\{0,\ldots,d-1\}$, we add 
    \begin{itemize}
        \item an $\varepsilon$-transition from $(j,0)$ to the initial state of the automaton $P_j$;
        \item an $\varepsilon$-transition from $(j,|x_j|)$ to the initial state of the automaton $S_j$.
    \end{itemize}
    The initial states are the states $(j,0)$ and the final states are the final states of $P_j$ and $S_j$ for all $j\in\{0,\ldots,d-1\}$.
    
    \smallskip
    We claim that
    \[
        L_U=L(\mathcal{A})\setminus\, (K_1\cup K_2)
    \]
    where $L(\mathcal{A})$ denotes the language accepted by the NFA $\mathcal{A}$, and $K_1$ and $K_2$ are the following regular languages:
    \begin{align*}
        K_1 &=\bigcup_{\ell=1}^{n+d-1} A_U^*\{s \in A_U^\ell: s>_{\mathrm{lex}} \rep_{U}(U_{\ell}-1)\} \\
        K_2 &=\bigcup_{j=0}^{d-1} A_U^*x_jy_j^*\{s \in A_U^{|z_j|}: s>_{\mathrm{lex}} z_j\}.
    \end{align*}
    By the stability properties of regular languages, this will imply that $L_U$ is a regular language.
    
    \smallskip
    We start with the inclusion $L_U\subseteq L(\mathcal{A})\setminus\, (K_1\cup K_2)$. Let $w\in L_U$.  We show by induction on its length that $w$ is accepted by $\mathcal{A}$ from the state $(j,0)$ where $j=|w|\bmod d$. Our base case is as follows: if $|w|< n$ then $w$ is accepted by $\mathcal{A}$ by first following the $\varepsilon$-transition from $(j,0)$ to the initial state of $P_j$, and then is accepted by $P_j$ by definition. Note that since $n\ge d$, the base case contains all possible lengths modulo $d$. Now, suppose that $|w|\ge n$ and that all shorter words $s$ of $L_U$ are accepted by $\mathcal{A}$ from the state $(k,0)$ where $k=|s|\bmod d$. We know that $\rep_{U}(U_{|w|}-1)\in x_jy_j^*z_j$. More precisely, we have $\rep_{U}(U_{|w|}-1)= x_jy_j^m z_j$ with $m=(|w|-|x_jz_j|)/d$. \cref{lem:Suffixes} yields that $w\lex \rep_{U}(U_{|w|}-1)$. If $x_jy_j^m$ is a prefix of $w$, i.e., $w=x_jy_j^ms$, then the suffix $s$ is such that $s\in L_U$, $|s|=|z_j|$ and $s\lex z_j$. In this case, the word $w$ is accepted by first following the path labeled by $x_jy_j^m$ from $(j,0)$ to $(j,|x_j|)$, then by taking the $\varepsilon$ transition from $(j,|x_j|)$ to the initial state of $S_j$, and then following the unique path labeled by $s$ in $S_j$. Since $S_j$ accepts $s$ by definition, this shows that $w$ is accepted by $\mathcal{A}$ from $(j,0)$. Now suppose that $w=pas$, $\rep_{U}(U_{|w|}-1)=pbs'$ with $|p|<|x_jy_j^m|$, $a,b\in A_U$, $a<b$. Then there is a path reading $pa$ from $(j,0)$ to $(k,0)$ where $k=(|w|-|pa|)\bmod d=|s|\bmod d$. By induction hypothesis, the suffix $s$ is accepted from $(k,0)$, hence $w$ is accepted by $\mathcal{A}$ from $(j,0)$. Moreover, \cref{lem:Suffixes} also implies that $w$ does not belong to $K_1$ nor $K_2$.
    
    \smallskip
    We now turn to the converse inclusion. Let $w\in L(\mathcal{A})\setminus\, (K_1\cup K_2)$. Then no suffix of $w$ belongs to $K_1\cup K_2$ either. Again, we proceed by induction on the length of the words.  
    By construction of the automaton $\mathcal{A}$, the word $w$ must be accepted from the state $(j,0)$ where $j=|w|\bmod d$. If $|w|< n$, then $w\in L_U$ since $w\notin K_1$. Now, suppose that $w$ is an accepted word of length $|w|\ge n$, and that all shorter words accepted by $\mathcal{A}$ but not belonging to $K_1$ nor $K_2$ belong to $L_U$. Since $|w|\ge n$, the lexicographically greatest word in $L_U$ of length $|w|$ is $\rep_{U}(U_{|w|}-1)=x_jy_j^mz_j$ with $j=|w|\bmod d$ and $m=(|w|-|x_jz_j|)/d$. Let us compare $w$ and $\rep_{U}(U_{|w|}-1)$. 
    
    If these two words first differ in the prefix of length $|x_jy_j^m|$, then by construction of $\mathcal{A}$, the first differing letter in $w$ must be less than the corresponding letter in $\rep_{U}(U_{|w|}-1)$, i.e., $w=pas$, $\rep_{U}(U_{|w|}-1)=pbs'$, $|p|<|x_jy_j^m|$, $a,b\in A_U$, $a<b$. Our accepting path must start in $(j,0)$, and after reading the prefix $pa$, ends in the state $(k,0)$ where $k=(|w|-|pa|)\bmod d=|s|\bmod d$. Therefore, the suffix $s$ is accepted by $\mathcal{A}$ from this state $(k,0)$. 
    Since $s\notin K_1\cup K_2$, we may apply the induction hypothesis. We get that $s$ belongs to $L_U$. 
    Now, consider a suffix $t$ of $w$ such that $|t|>|s|$. Then $p=p_1p_2$ and $t=p_2as<_{\mathrm{lex}} p_2bs'$. Since $p_2bs'$ is a suffix of $\rep_{U}(U_{|w|}-1)$, it satisfies $p_2bs'\lex \rep_{U}(U_{|t|}-1)$ by \cref{lem:Suffixes}. Hence $t\lex \rep_{U}(U_{|t|}-1)$. By \cref{lem:Suffixes} again, we obtain that $w\in L_U$. 
    
    We are left with the case where $w$ and $\rep_{U}(U_{|w|}-1)$ are either equal or differ in the last $|z_j|$ digits. This means that $w=x_jy_j^ms$ with $|s|=|z_j|$. Since $w\notin K_1\cup K_2$ and $|s|<n+d$, there are strong restrictions on the suffix $s$: it must satisfy $s\in L_U$ and $s\lex z_j$. Therefore, as in the previous paragraph, all suffixes $t$ of $w$ satisfy $t\lex \rep_{U}(U_{|t|}-1)$, which implies that $w\in L_U$ by \cref{lem:Suffixes}.
    \end{proof}

    In order to provide an example illustrating the proof of \cref{prop:RegIffMaxReg}, we show how to build an ad hoc positional numeration system from a list of candidate maximal words.

\begin{lemma}
\label{lem:Candidate-Max-Words}
    Let $M$ be a language over a (finite) alphabet included in $\N$. There exists a positional numeration system $U$ such that $M=\Max(L_U)$ if and only if the language $M$ satisfies the following properties:

    \smallskip
    \begin{itemize}\setlength\itemsep{0.7em}
        \item $M$ has exactly one word of each length.
        \item No word in $M$ starts with the digit $0$.
        \item For all words $u_{n-1}\cdots u_0$ and $v_{\ell-1}\cdots v_0$ in $ M$ with $n< \ell$, we have $v_{n-1}\cdots v_0\lex u_{n-1}\cdots u_0$.
    \end{itemize}

    \smallskip
    Furthermore, if such a positional numeration system $U$ exists, then it is unique. 
\end{lemma}

\begin{proof}
    The necessary condition and the uniqueness of the system are straightforward. Let us prove the sufficient condition. So, we suppose that the language $M$ satisfies the three properties given in the statement and we show how to build a positional numeration system $U$ such that $M=\Max(L_U)$.

    For each $n\ge 0$, let $w_n$ denote the unique word of length $n$ in $M$ and let $U_n=\val_U(w_n)+1$. In particular, $w_0$ is the empty word and $U_0=1$. Since the computation of $\val_U(w_n)$ only requires the knowledge of $U_0,\ldots,U_{n-1}$, the terms of the sequence $U=(U_n)_{n\ge 0}$ are obtained recursively. As we have assumed that no word in $M$ starts with $0$, this sequence is increasing. We have to show that for all $n\ge 0$, the word $w_n$ belongs to $L_U$, which is equivalent to show that $w_n=\rep_U(U_n-1)$. In particular, we will also obtain that the quotients $\frac{U_{n+1}}{U_n}$ are uniformly bounded by $C+1$ where $C$ is the maximal element of the alphabet of $M$.
    
    We proceed by induction on $n$. For $n=0$, this is clear. Suppose that $n\ge 1$ and that we have $w_\ell=\rep_U(U_\ell-1)$ for all $\ell<n$. Write $w_n=w_{n,n-1}\cdots w_{n,0}$ where the $w_{n,\ell}$'s are letters. By \cref{lem:Suffixes}, we have to show that $w_{n,\ell-1}\cdots w_{n,0}\lex \rep_U(U_\ell-1)$ for all $\ell\le n$. For $\ell<n$, by using the properties of the language $M$ and the induction hypothesis, we get that $w_{n,\ell-1}\cdots w_{n,0}\lex w_\ell=\rep_U(U_\ell-1)$. For $\ell=n$, since $\val_U(w_n)=U_n-1$ by definition of $U$ and $w_n$ does not start with $0$, we get that $w_n\lex \rep_U(U_n-1)$ from \cref{lem:Rep-Val-RadixOrder}.     
\end{proof}

\begin{example}
    \label{Ex:Automaton}
    Let $M$ be the following regular language, which is intended to provide the candidate maximal words of some positional numeration system according to \cref{lem:Candidate-Max-Words}:
    \[
        M = \ 21(11)^*00 
        \ \cup \ 2101(01)^*1
        \ \cup \ \{\varepsilon,1,11,101\}.
    \]
    With the notation of the proof of \cref{prop:RegIffMaxReg}, we have     
    \[
        d=2,\ n=4,\ 
        x_0=21,\ y_0=11,\ z_0=00,\ 
        x_1=2101,\ y_1=01,\ z_1=1.
    \]
    By \cref{lem:Candidate-Max-Words}, there exists a unique positional system $U=(U_n)_{n\ge 0}$ such that $M=\Max(L_U)$. One can check that it is given by the recurrence relation $U_{n+9}=8U_{n+7}-10U_{n+5}+2U_{n+3}$ for $n\ge 0$ and the initial values $(U_0,\ldots,U_8)=(1,2,4, 6, 17, 44, 116, 286, 760)$. In \cref{fig:automaton}, we have drawn the automaton described in the proof of \cref{prop:RegIffMaxReg}. The finitely many words accepted by the automata $P_0,P_1,S_0,S_1$ are listed explicitly. The initial states are marked with an incoming arrow. The accepting path all end in one of the automata $P_0,P_1,S_0,S_1$.
    \begin{figure}[htb]
\centering
\begin{tikzpicture}
\tikzstyle{every node}=[shape=rectangle, fill=none, draw=black,minimum size=20pt, inner sep=2pt]
\node(00) at (0,3) {$0,0$};
\node(01) at (2,3) {$0,1$} ;
\node(02) at (4,3) {$0,2$} ;
\node(03) at (6,3) {$0,3$} ;
\node(10) at (0,0) {$1,0$} ;
\node(11) at (2,0) {$1,1$};
\node(12) at (4,0) {$1,2$} ;
\node(13) at (6,0) {$1,3$} ;
\node(14) at (8,0) {$1,4$} ;
\node(15) at (10,0) {$1,5$} ;

\node(P0) at (-2,5) {$P_0:\varepsilon,00,01,10,11$};
\node(S0) at (7,5) {$S_0: 00$};
\node(P1) at (-2,-2.5) [text width=2.2cm] {$P_1:0,1,000$, \\
$001,010,011$, \\
$100,101$};
\node(S1) at (9,-2.5) {$S_1: 0,1$};

\tikzstyle{every path}=[color=black, line width=0.5 pt]
\tikzstyle{every node}=[shape=circle]
% fleche des états initiaux
\draw [-Latex] (-1,3) to node {} (00);
\draw [-Latex] (-1,0) to node {} (10);
% autres flèches
\draw [-Latex] (00) to node [above] {$2$} (01);
\draw [-Latex] (01) to node [above] {$1$} (02);
\draw [-Latex] (02) to node [above] {$1$} (03);
\draw [-Latex] (03) to [bend left=15] node [below] {$1$} (02) ;
\draw [-Latex] (10) to node [above] {$2$} (11);
\draw [-Latex] (11) to node [above] {$1$} (12);
\draw [-Latex] (12) to node [above] {$0$} (13);
\draw [-Latex] (13) to node [above] {$1$} (14);
\draw [-Latex] (14) to node [above] {$0$} (15);
\draw [-Latex] (15) to [bend left=15] node [below] {$1$} (14) ;

\draw [-Latex] (00) to [bend right=15] node [left] {$0,1$} (10) ;
\draw [-Latex] (10) to [bend right=15] node [right] {$0,1$} (00) ;
\draw [-Latex] (01) to [bend left=15] node [below] {$0$} (00) ;
\draw [-Latex] (11) to [bend left=15] node [below] {$0$} (10) ;
\draw [-Latex] (02) to node [below] {$0$} (10) ;
\draw [-Latex] (03) to [bend left=40] node [below] {$0$} (00) ;
\draw [-Latex] (13) to [bend left=40] node [below] {$0$} (10) ;
\draw [-Latex] (15) to [bend left=52] node [below] {$0$} (10) ;

% flèches vers P S
\draw [-Latex] (00) to node [right] {$\varepsilon$} (P0);
\draw [-Latex] (02) to node [near end,left] {$\varepsilon$} (S0);
\draw [-Latex] (10) to node [right] {$\varepsilon$} (P1);
\draw [-Latex] (14) to node [left] {$\varepsilon$} (S1);

\end{tikzpicture}
\caption{The non-deterministic automaton  $\mathcal{A}$ built from the set of maximal words $M=\ 21(11)^*00 
        \ \cup \ 2101(01)^*1
        \ \cup \ \{\varepsilon,1,11,101\}$.}
\label{fig:automaton}
\end{figure}
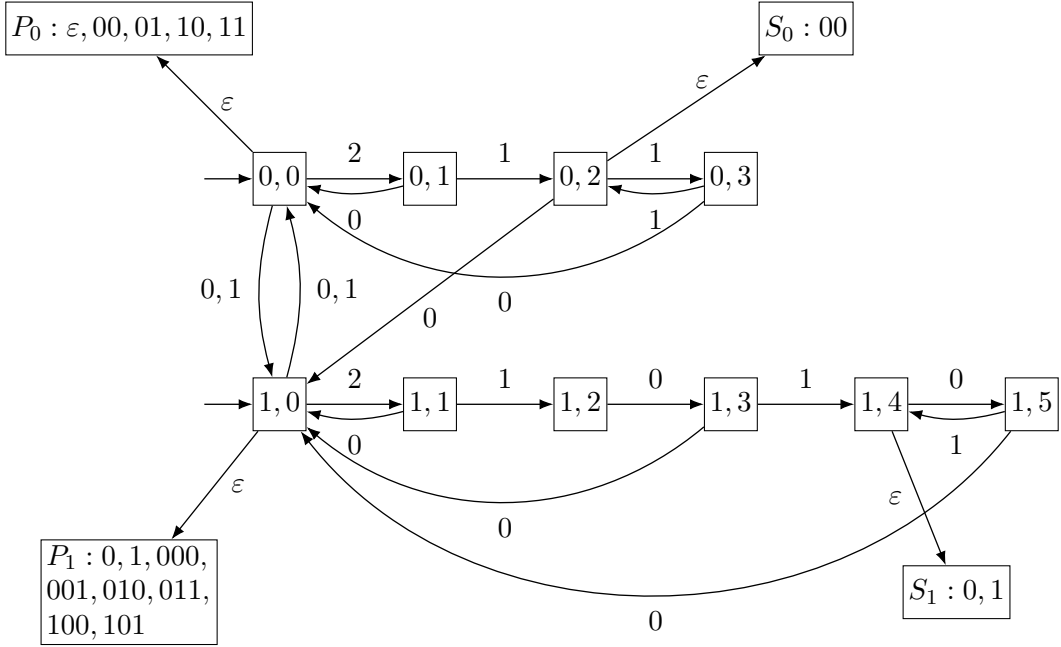

This example shows that words in $K_1\cup K_2$ may be accepted by the automaton $\mathcal{A}$ despite the fact that they do not belong to $L_U$. For example, the words in $21(11)^*01$ are accepted by $\mathcal{A}$ by following a path starting in the state $(0,0)$ and ending in $P_1$. However they do not belong to $L_U$ as they are lexicographically greater than the maximal corresponding words, which belong to $21(11)^*00$. Since the suffix $01$ is lexicographically greater than $z_0=00$, one has $21(11)^*01\subset K_2$ . Now, consider the word $111$. It is accepted by $\mathcal{A}$ by following the path from $(1,0)$ to $(0,0)$ labeled by $1$, and then taking the $\varepsilon$-transition going to $P_0$, which accepts $11$. This word $111$ belongs to $K_1$ as it is of length $3<n+d=6$ and is lexicographically greater than $101$, the maximal word of length $3$.
    \end{example}

Studying the regularity of the language $\mathrm{Max}(L_U)$ is easier than that of $L_U$ itself since $\mathrm{Max} (L_U)$ is a \emph{thin} language, that is, it contains at most one word of each length \cite{Paun&Salomaa:1995}. This allows us to use \cref{lem:xy*z}.

We introduce one last straightforward lemma, which will allow us to separate the study of the regularity of $\mathrm{Max} (L_U)$ according to the residue class modulo $p$ for some suitable $p$. 

\begin{lemma}
\label{lem:SplitModp}
Let $L$ be a language over a finite alphabet $A$ and let $p,N\in\N$ with $p\ge 1$. The language $L$ is regular if and only if the $p$ languages 
\[
    \{w\in L : |w|\equiv i \pmod p,\ |w|\ge N\}
\]
are all regular.
\end{lemma}

\section{Third step: Emergence of alternate bases associated with regular numeration systems}
\label{sec:alternate-bases}

Alternate bases are particular Cantor real bases that were independently introduced in~\cite{Caalim&Demegillo:2020,Charlier&Cisternino:2021}. As it turns out, the first appearance of these alternating systems of real bases which gave rise to the study~\cite{Charlier&Cisternino:2021} is due to the observations reported in this section. See \cite{Charlier:2023} for a survey on alternate bases.  

\begin{proposition}
\label{prop:U->beta}
Let $U=(U_n)_{n\ge0}$ be a positional numeration system such that $L_U$ is regular. There exists a positive integer $p$ such that the $p$ limits, for $i\in\{0,\ldots,p-1\}$,
\[
    \lim_{n\to+\infty}\frac{U_{np-i}}{U_{np-i-1}}
\]
exist and can be effectively computed. In particular, the limit 
\[
    \lim_{n\to+\infty}\frac{U_{n}}{U_{n-p}}
\]
exist and is equal to the product of the $p$ above limits.
\end{proposition}

\begin{proof}
We know from \cref{prop:reg->linear} that $U$ is linear.
Let $\alpha_1,\ldots,\alpha_d$ be the eigenvalues of $U$ with multiplicities $m_1,\ldots,m_d$ respectively. It is classical (see for instance~\cite[Chapter 6]{Berstel&Reutenauer:2011}) that
\[
	U_n=\sum_{j=1}^d P_j(n) \alpha_j^n
\]
for all sufficiently large $n$, where the $P_j$'s are polynomials (with coefficients in $\mathbb{C}$) of degree equal to $m_j-1$.

Since the language $L_U$ is regular, the formal series $\sum_{n\ge0}U_n X^n$ is $\N$-rational (see for instance \cite[Proposition 7.3.7]{Lothaire:2002}). The sequence $U$ is not ultimately zero in our framework, hence the series $\sum_{n\ge0}U_n X^n$ is not a polynomial and we may apply \cite[Theorem II.10.1]{Salomaa&Soittola:1978}; also see \cite{Berstel:1971,Berstel&Reutenauer:2011}.
This result tells us that the eigenvalues of $U$ of maximum modulus are of the form $\rho \xi$ where $\rho>0$ and $\xi$ is a root of unity. Moreover, $\rho$ itself is among these eigenvalues and the multiplicity of any such eigenvalue $\rho\xi$ is at most that of $\rho$.

Let $p$ be the least positive integer such that $\xi^p=1$ for all $\xi$ as in the previous paragraph. For each $i\in\{0,\ldots,p-1\}$ and for all sufficiently large $n$, we obtain
\[
	U_{np-i}
	= Q_i(n) \rho^{np} + \sum_{\substack{1\le j\le d\\ |\alpha_j|<\rho}}\alpha_j^{-i} P_j(np-i) \alpha_j^{np}
\]
where 
\[
	Q_i(n)=\sum_{\substack{1\le j\le d\\ |\alpha_j|=\rho}} \alpha_j^{-i} P_j(np-i).
\]
Since the sequence $U$ is increasing, all the polynomials $Q_i$ share the same degree, for $i\in\{0,\ldots,p-1\}$. This implies that the announced $p$ limits exist: we have
\[
     \lim_{n\to+\infty}\frac{U_{np-i}}{U_{np-i-1}}=\frac{q_i}{q_{i+1}}, \text{ for }i\in\{0,\ldots,p-2\},
     \quad \text{ and }\quad
     \lim_{n\to+\infty}\frac{U_{np-{p-1}}}{U_{np-p}}=\frac{q_{p-1}}{q_0}\rho^p.
\]
where $q_i$ is the leading coefficient of the polynomial $Q_i$, for every $i$. Finally, we get that
\[
    \lim_{n\to+\infty}\frac{U_{np-i}}{U_{np-i-p}}=\rho^p.
\]
Since the latter limit does not depend on $i$, the particular case is also proven.
\end{proof}

\begin{remark}
Without the regularity hypothesis, that is, for an arbitrary linear numeration system $U=(U_n)_{n\ge0}$, we would only get that the formal series $\sum_{n\ge0}U_n X^n$ is $\Z$-rational, which is a strictly weaker property than being $\N$-rational. In this case, we would not be able to use the result of Berstel \cite{Berstel:1971} as reported in \cite{Salomaa&Soittola:1978,Berstel&Reutenauer:2011}.
\end{remark}

\begin{remark}
In general, for an arbitrary linear recurrence sequence $U=(U_n)_{n\ge0}$ and a given $p\ge 2$, the existence of the limit $\lim_{n\to+\infty}\frac{U_{n}}{U_{n-p}}$ does not imply the existence of the $p$ limits $\lim_{n\to+\infty}\frac{U_{np-i}}{U_{np-i-1}}$, for $i\in\{0,\ldots,p-1\}$. For $p=1$, the limit $\lim_{n\to+\infty}\frac{U_n}{U_{n-1}}$ is sometimes referred to as the Kepler limit of $U$ \cite{Berend&Kumar:2022,Berend&Kumar:2025}. Call the \emph{dominating eigenvalues} of $U$ the eigenvalues of $U$ with maximal multiplicity among the eigenvalues with maximal modulus. Then the Kepler limit exists if and only if there is only one dominating eigenvalue, in which case, the limit is precisely given by the dominating eigenvalue \cite{Fiorenza&Vincenzi:2011}. This terminology is thus coherent with that of Hollander \cite{Hollander:1998}. For $p\ge 2$, this result generalizes as follows. Let $\alpha_1,\ldots,\alpha_k$ be the dominating eigenvalues of $U$. If the limit of the quotients $\frac{U_{n}}{U_{n-p}}$ exists then $\alpha_1^p=\alpha_2^p=\cdots=\alpha_k^p$ and the limit is precisely given by this value $\alpha_1^p$. The converse does not hold in general. However, the converse indeed holds if we add the condition that the sequence of moduli $(|U_n|)_{n\ge 0}$ is eventually increasing. Of course, this additional condition is fulfilled whenever $U$ is the base sequence of a linear numeration system as is our case in the present work.
\end{remark}

In view of \cref{prop:U->beta}, with a positional numeration system with a regular language can be associated a $p$-tuple of real numbers greater than or equal to $1$, which correspond to the $p$ values of the limits.
We will see that these $p$ values play a crucial role in our study. This will be done through the introduction of a new way of representing real numbers by using these $p$ values in an alternating manner. Note that non-minimal choices of $p$ only lead us to consider multiples of the minimal possible one, which amounts to repeating the associated tuple of real numbers. Even though this is not forbidden as it won't be a problem of any kind for our future developments, it does not bring any interesting feature either. So, in what follows, we will always assume that $p$ is minimal.

\begin{definition}
An \emph{alternate real base} is given by a tuple $B=(\beta_0,\ldots,\beta_{p-1})$ of real numbers greater than or equal to $1$. In this system, the \emph{B-value} of an infinite word $\mathbf{a}=a_0a_1\cdots\in\N^{\N}$ is the real number 
\[ 
    \val_B(\mathbf{a}):=\sum_{n=0}^{+\infty} \frac{a_n}{\beta_0\beta_1\cdots\beta_n}
\]
provided that the series converges, where we set $\beta_n:=\beta_{n\bmod p}$ for every $n$, extending $B$ to a periodic sequence $(\beta_n)_{n\ge 0}$.

Every such sequence $\mathbf{a}$ with $\val_B(\mathbf{a})=x$ is said to be a \emph{$B$-representation} of the real number $x$. Unless we are in the degenerate case where $p=1$ and $\beta_0=1$, we can use a greedy algorithm to find a specific $B$-representation of a real number $x\in[0,1]$, which will be called the \emph{$B$-expansion} of $x$. We set $r_0:=x$. Then, for $n\in\N$, we compute $a_n:=\lfloor \beta_n r_n\rfloor$ and $r_{n+1}:=\beta_n r_n-\lfloor \beta_n r_n\rfloor$. The sequence $a_0a_1a_2\cdots$ of greedy digits is denoted by $d_B(1)$. It is easily checked that this is indeed a $B$-representation of $x$. Whenever $B=(\beta_0)=(1)$, the greedy algorithm works only for $x\in\{0,1\}$, in which case we set $d_B(1):=10^\omega$.
\end{definition}

Among these expansions, we will be mostly interested in the expansion of $1$, not only in the base $B$ but also in all the bases obtained as circular permutations of $B$. For $B=(\beta_0,\ldots,\beta_{p-1})$ and $i\in\{0,\ldots,p-1\}$, we let 
\[
    \mathbf{d}_i:=d_{B^{(i)}}(1)
\]
where $B^{(i)}:=(\beta_i,\ldots,\beta_{i+p-1})$ denotes the $i$-th circular shift of the base $B$. Again, it will be convenient to extend this notation to $\mathbf{d}_n$ for all $n$, where the index $n$ is seen modulo $p$. In particular, we have $\mathbf{d}_{np}=\mathbf{d}_0$ for all $n$. 

We say that a representation is  \emph{finite} if it ends in a tail of zeros, and \emph{infinite} otherwise. In our study, the case where a $B^{(i)}$-expansion of $1$ is finite will require additional care, as will be apparent over time. We describe how to obtain from such a representation a new infinite one which is maximal among all infinite $B^{(i)}$-representations of $1$. The obtained infinite sequence is called the \emph{quasi-greedy $B^{(i)}$-expansion} of $1$, is denoted by $\mathbf{d}_i^*$, and is computed thanks to a recursive process as follows: 
\[
    \mathbf{d}_i^*:=
    \begin{cases}
        \mathbf{d}_i,                     & \text{if } \mathbf{d}_i \text{ is infinite }; \\
        t_{i,0}t_{i,1}\cdots t_{i,\ell_i-2}(t_{i,\ell_i-1}-1)\mathbf{d}_{i+\ell_i}^*,    & \text{if } \mathbf{d}_i=t_{i,0}t_{i,1}\cdots t_{i,\ell_i-1}0^\omega \text{ with } t_{i,\ell_i-1}> 0.
    \end{cases}
\]
If at least one $\beta_j$ is greater than $1$ then $\mathbf{d}_i^*$ is indeed a $B^{(i)}$-representation of $1$. In the degenerate case where $B=\beta_0=1$, we obtain $\mathbf{d}_0^*=0^\omega$, which is clearly not a $B$-representation of $1$. Note that, in this degenerate case, no infinite representation of $1$ exists.

\begin{definition}
    We say that a positional numeration system $U=(U_n)_{n\ge 0}$ has an \emph{associated alternate real base} $B=(\beta_0,\ldots,\beta_{p-1})$ if 
    \[
        \beta_i=\lim_{n\to+\infty}\frac{U_{np-i}}{U_{np-i-1}}
    \]
    for each $i\in\{0,\ldots,p-1\}$. By considering the minimal possible $p$ for which these limits exist, we may talk about \emph{the} alternate real base associated with $U$.
\end{definition}

Reformulating \cref{prop:U->beta}, every positional numeration system $U=(U_n)_{n\ge 0}$ with a regular numeration language has an associated alternate real base. Let us present a few examples in order to illustrate our considerations.

\begin{example} \
\label{ex:associated-alternate-base}
\begin{enumerate}[label=(\arabic*)]
    \item The base associated with a positional numeration system having a dominant root is simply given by this dominant root. In this case, $p=1$ and the notion of alternate base reduces to the Rényi real base case. For example, this is the case of the integer base numeration systems given by $U_n=b^n$ for all $n$, where $b\ge 2$ is an integer. The associated real base is the integer $b$. This is also the case of the Zeckendorf numeration system based on the Fibonacci sequence starting with $(1,2)$. This system has a dominant root, which is the golden ratio $\varphi=\frac{1+\sqrt{5}}{2}$. So the associated real base is $\varphi$.
    
    \item \label{item:ex-favori}
    Consider the linear numeration system built on the sequence $U=(U_n)_{n\ge0}$ defined by $U_{n+4}=3U_{n+2}+U_n$ for $n\ge 0$ and $(U_0,U_1,U_2,U_3)=(1,2,5,7)$. Then we have 
    \[
	   \lim_{n\to+\infty}\frac{U_{n+2}}{U_n}=\beta=\frac{3+\sqrt{13}}{2},
    \] 
    which is the unique root of maximum modulus of the polynomial $X^2-3X-1$. Further, we have 
    \begin{align*}
        \lim_{n\to+\infty}\frac{U_{2n}}{U_{2n-1}}=\beta_0=\beta-1=\frac{1+\sqrt{13}}{2} \\
        \lim_{n\to+\infty}\frac{U_{2n-1}}{U_{2n-2}}=\beta_1=\frac{\beta}{\beta-1}=\frac{5+\sqrt{13}}{6}.
    \end{align*}
    Therefore, this system has no dominant root. The alternate base associated with the numeration system $U$ is given by $B=( \frac{1+\sqrt{13}}{2},\frac{5+\sqrt{13}}{6})$. The greedy expansions of $1$ with respect to the alternate bases $B^{(0)}=B$ and $B^{(1)}=(\frac{5+\sqrt{13}}{6},\frac{1+\sqrt{13}}{2})$ are given by
    \[
        \mathbf{d}_0=2010^\omega
        \qquad\text{and}\qquad 
        \mathbf{d}_1=110^\omega.
    \]
    The quasi-greedy expansions of $1$ are then given by
    \[
        \mathbf{d}_0^*=20(01)^\omega
        \qquad\text{and}\qquad 
        \mathbf{d}_1^*=(10)^\omega.
    \]
    
    \item Let now $U=(U_n)_{n\ge0}$ be defined by $U_{n+6}=2U_{n+4}-U_n$ for $n\ge 0$ and the initial conditions $(U_0,\ldots,U_5)=(1,3,4,6,9,11)$. The alternate base associated with this system is $B=(\varphi,1)$. The corresponding greedy expansions of $1$ are 
    \[
        \mathbf{d}_0=1010^\omega
        \qquad\text{and}\qquad 
        \mathbf{d}_1=10^\omega.
    \]
    The quasi-greedy expansions of $1$ are then given by
    \[
        \mathbf{d}_0^*=(1000)^\omega
        \qquad\text{and}\qquad 
        \mathbf{d}_1^*=0(1000)^\omega.
    \]
    This illustrates that in the case where $B=(\beta_0,\ldots,\beta_{p-1})$ with $\beta_i=1$ for some $i$, the quasi-greedy $B^{(i)}$-expansion of $1$ starts with the digit $0$. Such a situation will indeed be allowed and taken into account in our future developments.
\end{enumerate}
\end{example}

\section{The quasi-greedy algorithm seen in a graph}
\label{sec:graph}

The quasi-greedy algorithm will be central in our study. We dedicate some notation to it. From the recursive definition of $\mathbf{d}_i^*$, we will be lead to combine the expansions with respect to the $p$ circular shifts of the alternate base $B$ in order to be able to compute a quasi-greedy expansion. In order to emphasize these links between the different shifted bases, we define a directed graph $G$ in the following way. The set of vertices of $G$ is $\{0,\ldots,p-1\}$. There is an edge going from $i$ to $(i+\ell_i) \bmod p$ if $\mathbf{d}_i$ is finite of length $\ell_i$. This defines a graph where each vertex has outdegree at most $1$. Whenever it exists, we let $\sigma(i)$ denote the successor of $i$, i.e., 
\[
    \sigma(i):=(i+\ell_i)\bmod p.
\]
This allows us to separate the vertices of $G$ into four categories:
\begin{enumerate}
    \item Vertices $i$ with no successor.
    \item Vertices $i$ leading to a vertex with no successor, i.e., such that there exists $r\ge 1$ such that $\sigma^r(i)$ has no successor.
    \item Vertices $i$ within a cycle, i.e., such that there exists $r\ge 1$ with $\sigma^r(i)=i$.
    \item Vertices $i$ leading to a cycle, i.e., such that there exists $r\ge 1$ such that $\sigma^r(i)$ belongs to a cycle.
\end{enumerate}
Each of these categories will be treated separately in the proof of our main result. 

The digits of the greedy expansions will be denoted by $t_{i,n}$ for $n\ge 0$, i.e., 
\[
     \mathbf{d}_i:=t_{i,0}t_{i,1}t_{i,2}\cdots 
\]
Moreover, we introduce the following notation to designate the prefix of $\mathbf{d}_i^*$ before encountering the next quasi-greedy expansion in the quasi-greedy process:
\[
    d'_i=t_{i,0}t_{i,1}\cdots t_{i,\ell_i-2}(t_{i,\ell_i-1}-1)
\]
whenever $\mathbf{d}_i=t_{i,0}t_{i,1}\cdots t_{i,\ell_i-1}0^\omega$ with $t_{i,\ell_i-1}>1$. In this case, we will always keep the notation  $\ell_i$ to refer to the \emph{length} of $\mathbf{d}_i$, i.e., the length of its prefix up to the last non-zero digit, which we denote by
\[
    d_i=t_{i,0}t_{i,1}\cdots t_{i,\ell_i-1}.
\]
Thus, in this situation, we have 
\[
    \mathbf{d}^*_i=d'_i\mathbf{d}^*_{i+\ell_i}.
\]
Note that we use bold letters to designate infinite words 
while we use a normal font for finite words. In particular, to avoid any confusion, we always write tails of zeros explicitly in infinite words. 

In order to iterate the quasi-greedy process, we define 
\[
    k_{i,j}:=\ell_i+\ell_{\sigma(i)}+\ell_{\sigma^2(i)}+\ldots+\ell_{\sigma^{j-1}(i)}
\]
and
\begin{equation}
\label{eq:w_ij}
    \mathbf{w}_{i,j}:= d'_id'_{i+k_{i,1}}\cdots d'_{i+k_{i,j-1}}\mathbf{d}_{i+k_{i,j}}
\end{equation}  
for $j\ge 0$, provided that all the involved lengths are finite. In particular, note that 
\begin{itemize}
    \item $k_{i,0}=0$ and $k_{i,1}=\ell_i$
    \item $\mathbf{d}_i=\mathbf{w}_{i,0}$ 
    \item if $i$ leads to a vertex with no successor or itself has no successor in $G$, then $\mathbf{d}_i^*=\mathbf{w}_{i,s}$ where $s\ge 0$ is such that $\sigma^s(i)$ has no successor in $G$.
    \item if $i$ leads to a cycle or itself belongs to a cycle in $G$, then $\mathbf{w}_{i,j}$ is defined for all $j\ge 0$ and $\mathbf{d}_i^*=\lim_{j\to\infty}\mathbf{w}_{i,j}=d'_id'_{i+k_{i,1}}\cdots $.   
\end{itemize}
All the other words $\mathbf{w}_{i,j}$ are called \emph{intermediate $B^{(i)}$-expansions} of $1$ as they evaluate to $1$ in base $B^{(i)}$ and they are lexicographically in between the quasi-greedy and the greedy $B^{(i)}$-expansions of $1$.  

Since indices of the words $d'_i$ and $\mathbf{d}_i$ can be taken modulo $p$, we can rewrite \eqref{eq:w_ij} as 
\begin{equation}
\label{eq:w_ij-sigma}
    \mathbf{w}_{i,j}:= d'_id'_{\sigma(i)}\cdots d'_{\sigma^{j-1}(i)}\mathbf{d}_{\sigma^j(i)}.
\end{equation}
Depending on the situation, we will allow ourselves to select the most convenient notation between \eqref{eq:w_ij} or \eqref{eq:w_ij-sigma}. For our further developments, we also define the values $m_{i,j}$ by
\begin{equation}
\label{eq:m_ij}
    \sigma^j(i)=i+k_{i,j}-m_{i,j}p.
\end{equation}

The following example illustrates the four categories of vertices described above.

\begin{example}
\label{ex:running}
Let $U=(U_n)_{n\ge 0}$ be defined by $U_{n+10}=16U_{n+5}-9U_n$ for $n\ge 0$ and the following initial conditions:
\[
\begin{array}{c|ccccc ccccc}
     n & 0 & 1 & 2 & 3 & 4 & 5 & 6 & 7 & 8 & 9 \\
     \hline
     U_n & 1 & 2 & 3 & 6 & 10 & 19 & 29 & 48 & 96 & 151
\end{array}
\]
For $i\in\{0,\ldots,4\}$, the limits
\[
	\beta_i=\lim_{n\to+\infty}\frac{U_{5n-i}}{U_{5n-i-1}}
\]
exist and can be effectively computed:
\[
    \begin{array}{c|c}
    i & \beta_i\\
    \hline
        0 & \frac{11+3\sqrt{55}}{17} \\
        1 & \frac{2+\sqrt{55}}{6}  \\
        2 & 2 \\
        3 & \frac{6+3\sqrt{55}}{17}\\
        4 & \frac{11+3\sqrt{55}}{22} \\
    \end{array}
\]
The product of these bases is $8+\sqrt{55}$, which is a root of $X^2-16X^2+9$ as expected.
Set $B=(\beta_0,\ldots,\beta_4)$. We get the following greedy and quasi-greedy $B^{(i)}$-expansions of $1$:
\begin{center}
\begin{minipage}{3.5cm}
\[
\begin{array}{c|l}
   i &  \mathbf{d}_i\\
   \hline
   0 & 1110^\omega\\
   1 & 11(00010)^\omega\\
   2 & 20^\omega\\
   3 & 110^\omega\\
   4 & 110^\omega
\end{array}
\]
\end{minipage}\qquad\qquad
\begin{minipage}{3.5cm}
\[
\begin{array}{c|l}
   i &  \mathbf{d}_i^*\\
   \hline
   0 & (11010)^\omega\\
   1 & 11(00010)^\omega\\
   2 & 1(10110)^\omega\\
   3 & (10110)^\omega\\
   4 & 1011(00010)^\omega
\end{array}
\]
\end{minipage}
\end{center}
For example, the intermediate $B$-expansions and $B^{(2)}$-expansions of $1$ are given by 
\begin{center}
\begin{minipage}{4.5cm}
\begin{align*}
	\mathbf{w}_{0,1}&=110\cdot 110^\omega\\
	\mathbf{w}_{0,2}&=110\cdot 10 \cdot 1110^\omega\\
	\mathbf{w}_{0,3}&=110\cdot 10 \cdot 110\cdot 110^\omega\\
	&\;\ \vdots
\end{align*}
\end{minipage}\qquad\qquad
\begin{minipage}{4.5cm}
\begin{align*}
	\mathbf{w}_{2,1}&=1\cdot 110^\omega\\
	\mathbf{w}_{2,2}&=1\cdot 10 \cdot 1110^\omega\\
	\mathbf{w}_{2,3}&=1\cdot 10\cdot 110 \cdot 110^\omega\\
	&\;\ \vdots
\end{align*}
\end{minipage}
\end{center}
There are no intermediate $B^{(1)}$-expansions of 1 since $\mathbf{w}_{1,0}=\mathbf{d}_1$ is infinite. The associated graph $G$ is depicted in \cref{fig:graphG}. The four categories outlined above correspond to $\{1\},\{4\},\{0,3\}$ and $\{2\}$, respectively.
\begin{figure}[htb]
\centering
\begin{tikzpicture}
\tikzstyle{every node}=[shape=circle, fill=none, draw=black,minimum size=15pt, inner sep=2pt]
\node(0) at (0,0) {$0$};
\node(1) at (2,0) {$1$} ;
\node(2) at (4,0) {$2$} ;
\node(3) at (6,0) {$3$} ;
\node(4) at (8,0) {$4$} ;
\tikzstyle{every path}=[color=black, line width=0.5 pt]
\tikzstyle{every node}=[shape=circle]
\draw [-Latex] (0) to [bend left=30] node [] {} (3) ;
\draw [-Latex] (3) to [bend left=30] node [] {} (0) ;
\draw [-Latex] (2) to [] node [] {} (3) ;
\draw [-Latex] (4) to [bend left=20] node [] {} (1) ;
\end{tikzpicture}
\caption{The graph $G$ associated with the numeration system $U$ of \cref{ex:running}.}
\label{fig:graphG}
\end{figure}
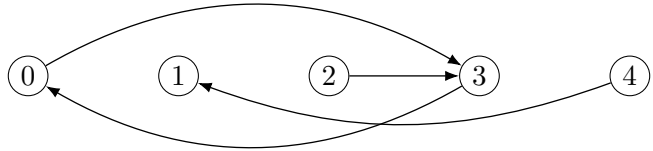
\end{example}

\section{Fourth step: Restriction to Parry alternate bases}
\label{sec:Parry}

The next result is one of the key arguments of our proof. It connects the lexicographically maximal words of each length in the numeration language to the representations of $1$ in the associated alternate base.

We introduce the following notation, which will allow us to include alternate bases $(\beta_0,\ldots,\beta_{p-1})$ with some $\beta_i$ equal to $1$. As previously discussed, such alternate bases can be linked to certain linear numeration systems, and therefore must be taken into account in our study.

\begin{definition}
\label{def:rep-ic}
Let $U$ be a positional numeration system with an associated alternate base $(\beta_0,\ldots,\beta_{p-1})$. Then for integers $i,c,n$ such that $0\le i<p$, $n\ge 1$ and $1\le c\le U_{np-i}$, we define
\[
    \rep_{i,c}(n):=0^\ell\rep_U(U_{np-i}-c),
\]
where $\ell=np-i-|\rep_U(U_{np-i}-c)|$. 
\end{definition}

The idea is that we pad the greedy representation of $U_{np-i}-c$ with leading zeros in order to obtain a representation of length $np-i$. Most of the time, this consists in doing nothing since if $\beta_i>1$, then $\rep_{i,c}(n)=\rep_U(U_{np-i}-c)$ for all large enough $n$. Moreover, we have $\rep_{i,1}(n)=\rep_U(U_{np-i}-1)$ for all $n\ge 1$. However, in the case where $\beta_i=1$ and $c\ge 2$, these few leading zeros will ensure that we keep working with words of the desired length $np-i$.  

\begin{proposition}
\label{prop:Pref-max=Pref-di}    
Let $U$ be a positional numeration system with an associated alternate base $(\beta_0,\ldots,\beta_{p-1})$, let $i\in\{0,\ldots,p-1\}$ and let $c$ be a positive integer. For all $L\ge 0$, there exists $N$ such that for all $n\ge N$, there exists $j\in\{0,\ldots,L\}$ such that $\rep_{i,c}(n)$ and $\mathbf{w}_{i,j}$ share a common prefix of length $L$.
\end{proposition}

\begin{proof}
For all $L,n\ge 0$, if there exists some $j\in\{0,\ldots,L\}$ such that $\rep_{i,c}(n)$ and $\mathbf{w}_{i,j}$ share a common prefix of length $L$, we define $j(L,n)$ to be the minimal such $j$. We show that $j(L,n)$ is well defined for all $L\ge 0$ and all large enough $n$ by induction on $L$. This holds for $L=0$ as $j(0,n)=0$ for all $n\ge 0$. Now, consider a fixed $L\ge 0$ and assume that there exists $N$ such that for all $n\ge N$, the value $j(L,n)$ is well defined. We have to prove the existence of $j(L+1,n)$ for all $n$ large enough.

For each $j$, we write 
\[
    \mathbf{w}_{i,j}=a_{j,0}a_{j,1}\cdots
\]
where we have dropped the dependence on $i$ for the digits since $i$ is fixed in this proof. 

Consider $n\ge N$. From the greedy algorithm and the induction hypothesis, we know that the digit indexed by $L$ in $\rep_{i,c}(n)$ is equal to $\lfloor b_{n,j(L,n)}\rfloor$ where
\[
    b_{n,j}:=\frac{U_{np-i}-c - a_{j,0} U_{np-i-1} - \cdots - a_{j,L-1} U_{np-i-L}}{U_{np-i-L-1}}.
\]
By hypothesis, for any fixed $k$, the quotient $\frac{U_{np-i-k}}{U_{np-i-L-1}}$ tends to the product $\beta_{i+k}\cdots\beta_{i+L}$ as $n$ tends to infinity. Therefore we obtain that, for a fixed $j$, the quantity $b_{n,j}$ converges to
 \begin{equation}
 \label{eq:star2}
    \beta_{i}\cdots\beta_{i+L} \left( 1-\frac{a_{j,0}}{\beta_i} -\cdots - \frac{a_{j,L-1}}{\beta_i\cdots\beta_{i+L-1}}\right)
\end{equation}
as $n$ tends to infinity.

Note that by minimality of $j(L,n)$, we have $k_{i,j(L,n)}\le L$. Indeed, we have $k_{i,0}=0$ and for $j\ge 1$, the words $\mathbf{w}_{i,j}$ and $\mathbf{w}_{i,j-1}$ share the same prefix of length $k_{i,j}-1$. Now, we consider some fixed $j\in\{0,\ldots,L\}$. For all $n\ge N$ such that $j(L,n)=j$, we study three cases that cover all possible situations. 
\begin{itemize}
    \item If $L\geq k_{i,j+1}$, all nonzero digits of $\mathbf{w}_{i,j}$ are contained in the prefix of length $L$. This implies that the quantity \eqref{eq:star2} is equal to $0$. Since $b_{n,j}$ is nonnegative and goes to $0$ as $n$ goes to infinity, its floor must be $0$ if $n$ is large enough, hence it coincides with $a_{j,L}$. 
    \item If $L=k_{i,j+1}-1$, then \eqref{eq:star2} is equal to $a_{j,L}$, which is positive. In this case, for large enough $n$, the floor of $b_{n,j}$ is either $a_{j,L}$ or $a_{j,L}-1$, which are the digits in position $L$ in $\mathbf{w}_{i,j}$ and $\mathbf{w}_{i,j+1}$ respectively. 
    \item Otherwise we have $k_{i,j}\le L<k_{i,j+1}-1$ and the quantity \eqref{eq:star2} belongs to the interval $(a_{j,L},a_{j,L}+1)$, which implies that $\lfloor b_{n,j}\rfloor=a_{j,L}$ for large enough $n$.
\end{itemize}
Since $j(L,n)\le L$, we get that there exists $N'\ge N$ such that for all $n\ge N'$, the value $j(L+1,n)$ is well defined and is equal to either $j(L,n)$ or $j(L,n)+1$.
\end{proof}

Let us consider a few examples in order to illustrate the previous proposition.

\begin{example} \
\label{ex:liste-exemples}
    \begin{enumerate}[label=(\arabic*)]\setlength\itemsep{0.7em}
        \item Let $U=(U_n)_{n\ge 0}=(1,2,3,5,8,\ldots)$ be the Zeckendorf numeration system already seen in \cref{ex:associated-alternate-base}. Since $p=1$, only $i=0$ has to be considered. We obtain $\mathbf{w}_{0,j}=(10)^j110^\omega$ for all $j\ge 0$ and $\rep_U(U_n-1)=\Pref_n((10)^\omega)=\Pref_n(\mathbf{w}_{0,n})$ for all $n\ge 0$, in accordance with \cref{prop:Pref-max=Pref-di}. 
        
        \item Let $U=(U_n)_{n\ge 0}=(1,2,3,5,9, 15, 24, 39, \ldots)$ be defined by $U_n=U_{n-1}+U_{n-3}+U_{n-4}+1$ for $n\ge 4$ and the initial conditions $(1,2,3,5)$. As in the previous example, this system has the dominant root $\varphi=\frac{1+\sqrt{5}}{2}$. However, in this case, the maximal words $\rep_U(U_n-1)$ are given by $\Pref_n(110^\omega)=\Pref_n(\mathbf{w}_{0,0})$ for $n\equiv 0,1\pmod 4$ and by $\Pref_n(10110^\omega)=\Pref_n(\mathbf{w}_{0,1})$ for $n\equiv 2,3\pmod 4$, which is also a behavior predicted by \cref{prop:Pref-max=Pref-di}. In particular, we see that the sequence of maximal words does not admit a limit.

        \item Let $U=(U_n)_{n\ge 0}=(1, 4, 19, 82, 325, \ldots)$ defined by $U_n = n  3^n + 1$. In \cref{tab:slow-convergence}, we have computed the first maximal words in $L_U$.
        \begin{table}
        \centering
        \begin{tabular}{l|l||l|l}
        $U_n-1$ & $\rep_U(U_n-1)$ & $U_n-1$ & $\rep_U(U_n-1)$ \\
        \hline
        0 & $\varepsilon$ & 15309 &  3123333 \\
        3 & 3 & 52488 & 31123332 \\
        18 & 42 & 177147 & 310320333 \\
        81 & 411 & 590490 & 3101123331 \\
        324 & 3402& 1948617 & 30310320330 \\
        1215 & 32400 & 6377292 & 302310320322 \\
        4374 & 320400 & 20726199 & 3022101123321
        \end{tabular}
        \bigskip
        \caption{First maximal words in $L_U$ for $U=(n  3^n + 1)_{n\ge 0}$.}
        \label{tab:slow-convergence}
        \end{table}
        This system is linear and has the dominant root $3$. It satisfies $\lim_{n\to \infty} \rep_U(U_n-1)=\mathbf{w}_{0,0}=30^\omega$ although the convergence speed is extremely slow. For example, we see a second $0$ only from $n= 30$ and a third one only from $n= 85$.
        
        \item Consider the positional numeration system \ref{item:ex-favori} from \cref{ex:associated-alternate-base}. Contrarily to the previous three examples where we had the dominant root condition, we now have $p=2$. We have $\mathbf{w}_{0,j}=20(01)^j10^\omega$ and $\mathbf{w}_{1,j}=(10)^j110^\omega$ for $j\ge 0$. The maximal words are given by $\rep_U(U_{2n}-1)=20(01)^{n-1}=\Pref_{2n}(\mathbf{w}_{0,n-1})$ and $\rep_U(U_{2n-1}-1)=(10)^{n-1}1=\Pref_{2n-1}(\mathbf{w}_{1,n-1})$ for $n\ge 1$.
        
        \item \label{item:liste-exemples-5} In the previous example, the minimal polynomial of $U$ was of the form $f(X^p)$ for some polynomial $f$. Now, consider $U=(U_n)_{n\ge0}=(1,3,8,12, 16, 48, \ldots)$ such that $U_{n+3}=2U_{n+2}-4U_{n+1}+8U_n$. This system has no dominant root but we have
        \begin{align*}
        \beta_0&=\lim_{n\to+\infty}\frac{U_{4n}}{U_{4n-1}}=\frac{4}{3},
        \qquad
        \beta_1=\lim_{n\to+\infty}\frac{U_{4n-1}}{U_{4n-2}}=\frac{3}{2}, \\
        \beta_2&=\lim_{n\to+\infty}\frac{U_{4n-2}}{U_{4n-3}}=\frac{8}{3},
        \qquad
        \beta_3=\lim_{n\to+\infty}\frac{U_{4n-3}}{U_{4n-4}}=3.
        \end{align*}
        Hence we get $p=4$ and we can compute
        \begin{align*}
        \mathbf{w}_{0,j} &=(1010)^j10110^\omega \quad \text{for }j\ge 0 \\
        \mathbf{w}_{1,0}&=1110^\omega \quad \text{and}\quad
        \mathbf{w}_{1,j} =110(1010)^{j-1}10110^\omega \quad \text{for }j\ge 1 \\
        \mathbf{w}_{2,0}&=220^\omega \quad \text{and}\quad
        \mathbf{w}_{2,j} =21(1010)^{j-1}10110^\omega \quad \text{for }j\ge 1 \\
        \mathbf{w}_{3,0}&=30^\omega \quad \text{and}\quad
        \mathbf{w}_{3,j} =2(1010)^{j-1}10110^\omega \quad \text{for }j\ge 1.
        \end{align*}
        The maximal words are given by
        \begin{align*}
        \rep_U(U_{4n}-1) &= (1010)^n = \Pref_{4n}(\mathbf{w}_{0,n}) \\
        \rep_U(U_{4n-1}-1) &= 110(1010)^{n-1} = \Pref_{4n-1}(\mathbf{w}_{1,n}) \\
        \rep_U(U_{4n-2}-1) &= 21(1010)^{n-1} = \Pref_{4n-2}(\mathbf{w}_{2,n}) \\
        \rep_U(U_{4n-3}-1) &= 2(1010)^{n-1} = \Pref_{4n-3}(\mathbf{w}_{3,n})
        \end{align*}
        for all $n\ge 1$.
        
        \item Consider a system $U=(1,3,6,11,15,\ldots)$ given by the recurrence relation $U_{n+4}=2U_{n+2}+3U_{n}$. Equivalently, $U$ is given by the relations $U_{2n}=U_{2n-1}+U_{2n-2}-2(-1)^n$ and $U_{2n+1}=2U_{2n}+(-1)^n$. In this case we have $p=2$, with $\beta_0=3/2$ and $\beta_1=2$. The maximal words are given by
        \begin{align*}
        \rep_U(U_{4n}-1) &= (10)^{2n} = \Pref_{4n}(\mathbf{w}_{0,2n}) \\
        \rep_U(U_{4n-1}-1) &= 1110^{4n-4} = \Pref_{4n-1}(\mathbf{w}_{1,1}) \\
        \rep_U(U_{4n-2}-1) &= 110^{4n-5}1 = \Pref_{4n-3}(\mathbf{w}_{0,0})1 \\
        \rep_U(U_{4n-3}-1) &= 20^{4n-4} = \Pref_{4n-3}(\mathbf{w}_{1,0})
        \end{align*}
        for all $n\ge 2$. We see that the sequence of maximal words does not admit a limit, even when selecting words of a given parity.

        \item Finally, let us illustrate the case where $\rep_{i,c}(n)$ differs from $\rep_U(U_{np-i}-c)$. This situation happens whenever $c\ge 2$ and $\beta_i=1$. Consider for instance the system $U=(U_n)_{n\ge 0}$ defined by $U_0=1$, $U_{2n-1}=2U_{2n-2}$ and $U_{2n}=U_{2n-1}+1$ for $n\ge 1$. The associated alternate base is $(1,2)$. It is easy to see that $\rep_U(U_{2n}-2)=110^{2n-3}$ for all $n\ge 2$ and that $\rep_U(U_{2n-1}-2)=10110^{2n-5}$ for all $n\ge 3$. Observe that both $\rep_U(U_{2n}-2)$ and $\rep_U(U_{2n-1}-2)$ have length $2n-1$. Hence for all $n\ge 3$, the words from \cref{def:rep-ic} are given by $\rep_{0,2}(n)=0110^{2n-3}$ and $\rep_{1,2}(n)=10110^{2n-5}$. For the alternate base $B=(1,2)$, we obtain that 
        \begin{align*}
        & \mathbf{w}_{0,2j}=(01)^j10^\omega, \quad
            \mathbf{w}_{0,2j+1}=(01)^j020^\omega, \\
        & \mathbf{w}_{1,2j}=(10)^j20^\omega, \quad
            \mathbf{w}_{1,2j+1}=(10)^j110^\omega    
        \end{align*}
        for all $j\ge 0$. We thus see that 
        \[
           \rep_{0,2}(n) =\Pref_{2n}(\mathbf{w}_{0,2}) 
           \quad\text{ and }\quad
           \rep_{1,2}(n) =\Pref_{2n}(\mathbf{w}_{1,3})
        \]
        for all $n\ge 3$, in accordance with \cref{prop:Pref-max=Pref-di}. Note that in contrast, $\rep_U(U_{2n}-2)$ starts with $11$ whereas none of the words $\mathbf{w}_{0,j}$ do, justifying the switch to $\rep_{i,c}$ in \cref{prop:Pref-max=Pref-di}.
    \end{enumerate}
\end{example}

Let us extract some useful information from \cref{prop:Pref-max=Pref-di} in the case where $\mathbf{d}_i$ is infinite. In terms of the graph $G$, this means that $i$ is a vertex with no outgoing edge.

\begin{corollary}
\label{cor:Pref-max=Pref-di}
Let $U$ be a positional numeration system with an associated alternate base $(\beta_0,\ldots,\beta_{p-1})$, let $i\in\{0,\ldots,p-1\}$ be such that $\mathbf{d}_i$ is infinite, and let $c$ be a positive integer. We have 
\[
    \lim_{n\to\infty}\rep_{i,c}(n)=\mathbf{d}_i
\]
where the limit is taken with respect to the product topology.
\end{corollary}

\begin{proof}
    Since $\mathbf{d}_i=\mathbf{d}^*_i$, only $j=0$ is possible in \cref{prop:Pref-max=Pref-di}.
\end{proof}

In \cref{sec:alternate-bases}, we have seen that any positional numeration system $U$ with a regular numeration language has an associated alternate real base $B=(\beta_0,\ldots,\beta_{p-1})$. The next result provides the additional information that this associated alternate base $B$ must have the property that $\mathbf{d}_i^*$ is ultimately periodic for every $i\in\{0,\ldots,p-1\}$. Alternate bases with this property are said to be \emph{Parry}, see~\cite{Charlier&Cisternino&Masakova&Pelantova:2023}. Equivalently, one can ask that $\mathbf{d}_i$ is finite or ultimately periodic for every $i\in\{0,\ldots,p-1\}$, see~\cite[Proposition 40]{Charlier&Cisternino:2021}.

\begin{proposition}
\label{prop:regular->Parry}
Let $U$ be a positional numeration system with a regular numeration language. Then its associated alternate base is Parry.
\end{proposition}

\begin{proof}
Let $p$ be the length of the alternate base and consider $i\in\{0,\ldots,p-1\}$ such that $\mathbf{d}_i$ is infinite. We have to show that $\mathbf{d}_i$ is ultimately periodic. By \cref{cor:Pref-max=Pref-di}, the sequence of finite words $(\rep_U(U_{np-i}-1))_{n\ge 1}$ converges to the infinite word $\mathbf{d}_i$. \cref{prop:RegIffMaxReg,lem:xy*z,lem:SplitModp} then imply that $\mathbf{d}_i$ is ultimately periodic.
\end{proof}

\section{Our strategy}\label{sec:strategy}

We consider an arbitrary positional numeration system $U$ and we aim at determining whether the numeration language $L_U$ is regular or not. At this point, we have obtained several necessary conditions for $L_U $ to be regular. We thus put ourselves in the restricted situation where these conditions are satisfied. Namely, we suppose that $U$ is linear and that there is a Parry alternate base $B=(\beta_0,\ldots,\beta_{p-1})$ associated with $U$. In view of the results of \cref{sec:max-regular}, in order to understand when $L_U$ is indeed regular, it suffices to study the regularity of the $p$ languages 
\[
    L_i:=\{\rep_U(U_{np-i}-1):n\ge 1\}.
\]
Either these $p$ languages are all regular, in which case, the numeration language $L_U$ is regular, or one of them is not regular, in which case $L_U$ is not regular.

We will see that the regularities of these $p$ languages are interconnected. Our approach is based on the graph $G$ which is meant to encode the connections between the $p$ shifts of the alternate base when performing the quasi-greedy algorithm. We will consider the four families of vertices $i$ that we already introduced in \cref{sec:graph}. After dealing with these four cases separately, we will combine them in order to describe a decision procedure that will allow us to determine whether or not the numeration language $L_U$ is regular. These results will not be illustrated until \cref{sec:decision}, where two appropriate examples will be presented to exhibit a variety of behaviors. We invite the reader to consult \cref{ex:ProcedureInfini} and \cref{ex:ProcedureBoucle}, for \cref{sec:i-infini,sec:i-vers-infini}, and \cref{sec:i-cycle,sec:i-vers-cycle}, respectively.

Since the regularities of the languages $L_i$ with $i\in\{0,\ldots,p-1\}$ are linked together, we will be lead to not only consider the maximal words of each length, but also words of the form $\rep_U(U_n-c)$ where $c$ is a constant. Note that \cref{prop:Pref-max=Pref-di} was already stated in this spirit. The following definition will be used in the statements of the four main theorems that describe the situation for each category of vertices in the graph $G$.

\begin{definition}
For $i\in\{0,\ldots,p-1\}$ and positive integers $c$, we define
\[
    L_{i,c}:=\{\rep_{i,c}(n):n\ge 1,\ U_{np-i}\ge c\}
\]
where the notation $\rep_{i,c}(n)$ has been defined in \cref{def:rep-ic}.
\end{definition}

Again, the condition that $U_{np-i}\ge c$ will not be a true restriction as for any $N\ge 0$, the language $L_{i,c}$ is regular if and only if so is the language $L_{i,c}\cap \{w\in A_U^* : |w|\ge N\}$.
Note that we always have $L_i=L_{i,1}$ as $\rep_{i,1}(n)=\rep_U(U_{np-i}-1)$ for all $n\ge 1$.

\section{Vertices with no outgoing edge}
\label{sec:i-infini}

In this section, we consider a (fixed) vertex $i$ with no outgoing edge in the graph $G$. This means that the greedy and the quasi-greedy $B^{(i)}$-expansions of 1 coincide, i.e., $\mathbf{d}_i=\mathbf{d}^*_i$. Since we are dealing with a Parry alternate base, this infinite word is ultimately periodic. We let $q_0$ be the minimal preperiod of $\mathbf{d}_i$ and we let $m_0$ be its minimal period that is a multiple of $p$. Thus, we have 
\[
    \mathbf{d}_i
    =t_{i,0}\cdots t_{i,q_0-1}
    (t_{i,q_0}\cdots t_{i,q_0+m_0-1})^\omega.
\]

The following definition generalizes an idea of Hollander \cite{Hollander:1998} to the non-dominant root case. 

\begin{definition}
\label{def:Delta_iqm}
    For $i\in\{0,\ldots,p-1\}$ such that $\mathbf{d}_i$ is infinite, integers $q\ge 0$ and $m\ge 1$, and integers $n$ such that $np-i\ge q+m$,  we define
    \[
        (\Delta_{i,q,m})_n
        =\left(U_{np-i}-\sum_{\ell=1}^{q+m}t_{i,\ell-1}U_{np-i-\ell}\right)
        -\left(U_{np-i-m}-\sum_{\ell=1}^q t_{i,\ell-1}U_{np-i-m-\ell}\right).
    \]
\end{definition}

In what follows, we will be concerned with properties of the values $(\Delta_{i,q,m})_n$ for large $n$ only. Therefore, we will not always pay much attention to verify that the condition $np-i\ge q+m$ is satisfied. Otherwise stated, we will only consider these values provided that they are indeed well defined. Also, we will allow ourselves to talk about the sequence $\Delta_{i,q,m}$ in reference to the sequence of values $(\Delta_{i,q,m})_n$ for large enough $n$.

\begin{lemma}
\label{lem:Delta-q}
Let $i\in\{0,\ldots,p-1\}$ be such that $\mathbf{d}_i$ is infinite.
For all $q\ge q_0$, all $k\ge 1$ and all $n$ large enough so that the following expressions are well defined, we have
\[
    (\Delta_{i,q,km_0})_n
    =(\Delta_{i,q_0,km_0})_n
    =\sum_{\ell=0}^{k-1}(\Delta_{i,q_0,m_0})_{n-\ell\frac{m_0}{p}}.
\]
\end{lemma}

\begin{proof}
Since $q\ge q_0$, we have
\[
    (\Delta_{i,q,km_0})_n-(\Delta_{i,q_0,km_0})_n
    =-\sum_{\ell=q_0+km_0+1}^{q+km_0}t_{i,\ell-1}U_{np-i-\ell}
    +\sum_{\ell=q_0+1}^q t_{i,\ell-1}U_{np-i-km_0-\ell}.
\]
Since $t_{i,\ell-1}=t_{i,\ell-1-km_0}$ for $\ell>q_0+km_0$, the two sums cancel and we obtain the first announced 
equality. 

We now turn to the second equality of the statement.
We proceed by induction on $k$. For $k=1$, the result is immediate. Now, assume that the equality holds for some $k\ge 1$ and let us show it for $k+1$. Then we have
\begin{align*}
    (\Delta_{i,q_0,(k+1)m_0})_n
    &=(\Delta_{i,q_0+m_0,km_0})_n+(\Delta_{i,q_0,m_0})_{n-k\frac{m_0}{p}} \\
    &=(\Delta_{i,q_0,km_0})_n+(\Delta_{i,q_0,m_0})_{n-k\frac{m_0}{p}} \\
    &=\sum_{\ell=0}^{k-1}(\Delta_{i,q_0,m_0})_{n-\ell\frac{m_0}{p}}+(\Delta_{i,q_0,m_0})_{n-k\frac{m_0}{p}} \\
    &=\sum_{\ell=0}^k(\Delta_{i,q_0,m_0})_{n-\ell\frac{m_0}{p}}
\end{align*}
where we have used the first part of the statement for the second step and the recurrence hypothesis for the third one.
\end{proof}

We now present the main result of this section. 

\begin{theorem}
\label{thm:i-infini}
Let $i\in\{0,\ldots,p-1\}$ be such that $\mathbf{d}_i$ is infinite. The following assertions are equivalent.

\smallskip
\begin{enumerate}[label=(\alph*)] \setlength\itemsep{0.7em}
    \item \label{item:i-infini-a}The language $L_i$ is regular. 
    \item \label{item:i-infini-b}For all $c\ge 1$, the languages $L_{i,c}$ are regular.
    \item \label{item:i-infini-c}There exists $k\ge 1$ such that the sequence $\Delta_{i,q_0,km_0}$ is ultimately zero.
    \item \label{item:i-infini-d}There exists $q\ge q_0$ and $k\ge 1$ such that the sequence $\Delta_{i,q,km_0}$ is ultimately zero.
\end{enumerate}
\end{theorem}

\begin{proof}
Clearly, \ref{item:i-infini-b} implies \ref{item:i-infini-a}. By using the first equality of \cref{lem:Delta-q}, we see that \ref{item:i-infini-c} and \ref{item:i-infini-d} are equivalent. We show that \ref{item:i-infini-a} implies \ref{item:i-infini-d}. Thus, we suppose that the language $L_i$ is regular. By \cref{lem:xy*z,cor:Pref-max=Pref-di}, we can split this language into a union of the form
\[
    L_i=F\cup \bigcup_{e=0}^{M-1}xy^*z_e
\]
where $F$ is a finite language, $M$ is a positive integer, $|y|=Mp$ and $\mathbf{d}_i=xy^\omega$. We must have that $|x|\ge q_0$ and $|y|$ is a multiple of $m_0$. Set $q=|x|$. We aim to show that $\Delta_{i,q,Mp}$ is eventually zero.

Let $n$ be large enough so that $\rep_{i,1}(n-M)\notin F$. Then there exist $e\in\{0,\ldots,M-1\}$ and $r\ge 1$ such that $\rep_{i,1}(n)=xy^rz_e$ and $\rep_{i,1}(n-M)=xy^{r-1}z_e$. Thus, we have
\begin{align*}
    U_{np-i}-U_{(n-M)p-i}
    & =(U_{np-i}-1)-(U_{(n-M)p-i}-1) \\
    & =\val_U(xy^rz_e)-\val_U(xy^{r-1}z_e) \\ 
    & =\val_U(xy0^{(r-1)|y|+|z_e|})-\val_U(x0^{(r-1)|y|+|z_e|}) \\
    & =\sum_{\ell=1}^{q+Mp} t_{i,\ell-1}U_{np-i-\ell}-\sum_{\ell=1}^q t_{i,\ell-1}U_{(n-M)p-i-\ell}.
\end{align*}  
We have therefore shown that $\Delta_{i,q,Mp}$ is ultimately zero, as expected.
\medskip

We now show that \ref{item:i-infini-d} implies \ref{item:i-infini-b}. Let $c\ge 1$ be fixed, and assume that there is some $q\ge q_0$ and $m$ multiple of $m_0$ such that $\Delta_{i,q,m}$ is ultimately zero, i.e., there exists $N_1$ such that $(\Delta_{i,q,m})_n=0$ for all $n\ge N_1$. As $m_0$ is a multiple of $p$, we can write $m=Mp$. By \cref{cor:Pref-max=Pref-di}, there exists $N_2$ such that for all $n\ge N_2$, the prefix of length $q+m$ of the word $\rep_{i,c}(n)$ is $t_{i,0}\cdots t_{i,q+m-1}$. Let $N=\max\{N_1,N_2\}$ and consider a fixed $n\ge N$. Set $z$ to be the suffix defined as 
\[
    \rep_{i,c}(n) =t_{i,0}\cdots t_{i,q-1} z.
\] 
We show by induction that
\begin{equation}
\label{eq:a}
    \rep_{i,c}(n+aM)=t_{i,0}\cdots t_{i,q-1}(t_{i,q}\cdots t_{i,q+m-1})^a z.
\end{equation}
 for all $a\ge 0$. The base case $a=0$ is the definition of $z$. Now, fix some $a\ge 0$ and suppose that \eqref{eq:a} holds. Let $s$ be the suffix defined as
\[
    \rep_{i,c}(n+(a+1)M)=t_{i,0}\cdots t_{i,q+m-1} s.
\]
We have to show that $s=(t_{i,q}\cdots t_{i,q+m-1})^a z$. These two words have the same length $np+am-i-q$. Moreover, they both belong to $L_U$ since they are suffixes of words in $L_U$. Therefore, in order to see that  these words are actually equal, it suffices to show that they have the same value. Since $(\Delta_{i,q,m})_{n+(a+1)M}=0$, we know that
\[
    U_{(n+(a+1)M)p-i}-U_{(n+aM)p-i}
    =\sum_{\ell=1}^{q+m}t_{i,\ell-1}U_{(n+(a+1)M)p-i-\ell}
    -\sum_{\ell=1}^qt_{i,\ell-1}U_{(n+aM)p-i-\ell}.
\]
Using the induction hypothesis, the latter equality and the definition of $s$, we successively obtain 
\begin{align*}
    \val_U\big((t_{i,q}\cdots t_{i,q+m-1})^a z\big)
    &= U_{(n+aM)p-i}-c - \sum_{\ell=1}^{q}t_{i,\ell-1}U_{(n+aM)p-i-\ell} \\
    &= U_{(n+(a+1)M)p-i}-c - \sum_{\ell=1}^{q+m}t_{i,\ell-1}U_{(n+(a+1)M)p-i-\ell} \\
    &= \val_U(s),
\end{align*}
as desired.
Thus, we have shown that for all $n\ge N$, there exists a finite word $z_n$ such that
\[
    \{\rep_{i,c}(n+aM): a\ge 0\}
    =t_{i,0}\cdots t_{i,q-1}(t_{i,q}\cdots t_{i,q+m-1})^* z_n.
\]
As in \cref{lem:SplitModp}, this implies that the language $L_{i,c}$ is regular since 
\[
    L_{i,c}\cap \{w\in A_U^* : |w|\ge Np-i\}
    =\bigcup_{n=N}^{N+M-1}\{\rep_{i,c}(n+aM): a\ge 0\}.
\]
\end{proof}

The following result shows that \cref{thm:i-infini} can be used in practice to decide whether the language $L_i$ is regular.

\begin{proposition} 
\label{prop:i-infini-effective}
    Assume that the eigenvalues of $U$ are known and let $i\in\{0,\ldots, p-1\}$ be such that ${\mathbf d}_i$ is ultimately periodic. Then the condition \ref{item:i-infini-c} of \cref{thm:i-infini} is effective. 
\end{proposition}

\begin{proof}
   Since ${\mathbf d}_i$ is known to be ultimately periodic, the values $q_0$ and $m_0$ (as defined in the beginning of \cref{sec:i-infini}) can be found by inspecting the remainders in the greedy algorithm. Note that if $U$ satisfies a recurrence relation of characteristic polynomial $P$, then the same is true for $\Delta_{i, q_0, m_0}$. As a result, the eigenvalues of $\Delta_{i,q_0,m_0}$ are among those of $U$ and can be effectively identified. Now, given \cref{lem:Delta-q}, for any $k\ge 1$, the sequence $\Delta_{i,q_0,km_0}$ is ultimately zero if and only if the sequence $\Delta_{i,q_0,m_0}$ ultimately satisfies the recurrence relation of characteristic polynomial 
    $1+X^{\frac{m_0}{p}}+\ldots+X^{(k-1)\frac{m_0}{p}}$. Since
    \[
        \left(1-X^{\frac{m_0}{p}}\right)\left(1+X^{\frac{m_0}{p}}+\ldots+X^{(k-1)\frac{m_0}{p}}\right)=1-X^{\frac{km_0}{p}},
    \]
    one of the following two cases must happen. 
    If the eigenvalues of $\Delta_{i,q_0,m_0}$ are all zero or roots of unity of order not dividing $\frac{m_0}{p}$, then the least common multiple of those orders is the desired $k$ so that \ref{item:i-infini-c} holds. Otherwise, there exists no $k$ such that \ref{item:i-infini-c} holds. 
\end{proof}

In the dominant root case studied by Hollander, \cref{thm:i-infini} corresponds to the case where $\beta$ is a non-simple Parry number; see \cite[Lemmas 7.3 and 7.4]{Hollander:1998}. In this case, we have $p=1$ and we only have to consider the sequence $\Delta_{q,m}$ defined by
    \[
        (\Delta_{q,m})_n
        =\left(U_n-\sum_{\ell=1}^{q+m}t_{\ell-1} U_{n-\ell}\right)
        -\left(U_{n-m}-\sum_{\ell=1}^q t_{\ell-1} U_{n-m-\ell}\right)
    \]
where the coefficients $t_j$ are the digits of the greedy expansion of the dominant root, which is a Parry number: $d_\beta(1)=t_0\cdots t_{q-1}(t_q\cdots t_{q+m-1})^\omega$. It is easily seen that this sequence is ultimately zero if and only if the base sequence $U$ ultimately satisfies the linear recurrence relation of characteristic polynomial
\[
    P_{q,m}:=\left(X^{q+m}-\sum_{\ell=1}^{q+m}t_{\ell-1} X^{q+m-\ell}\right)
    -\left(X^q-\sum_{\ell=1}^q t_{\ell-1} X^{q-\ell}\right).
\]
Such a polynomial $P_{q,m}$ is called an \emph{extended $\beta$-polynomial} by Hollander. In particular, whenever $q$ and $m$ are chosen to be the minimal preperiod $q_0$ and period $m_0$ of $d_\beta(1)$, we obtain the so-called \emph{Parry polynomial} $P_{q_0,m_0}$, also called \emph{the $\beta$-polynomial}. As noted by Hollander, all extended $\beta$-polynomials can be obtained by multiplying the Parry polynomial $P_{q_0,m_0}$  by some very specific polynomials, as we have
\begin{equation}
\label{eq:extended-beta-polynomials}
    P_{q,km_0}=X^{q-q_0}(1+X^{m_0}+\cdots+X^{(k-1)m_0}) P_{q_0,m_0}
\end{equation}
for any $q\ge q_0$ and $k\ge 1$.

With these comments, we see that the first part of the main result of \cite{Hollander:1998} can be re-obtained as a corollary of \cref{thm:i-infini}. 

\begin{corollary}[\cite{Hollander:1998}]
    Let $U$ be a positional numeration system with a dominant root $\beta>1$ such that $d_\beta(1)$ is ultimately periodic, i.e., $\beta$ is a non-simple Parry number. The numeration language $L_U$ is regular if and only if the base sequence $U$ satisfies a linear recurrence relation whose characteristic polynomial is given by an extended $\beta$-polynomial.
\end{corollary}

The relationships given in \cref{lem:Delta-q} can be viewed as an analogue of \eqref{eq:extended-beta-polynomials}. However, in the general setting considered in this paper, i.e., with possibly $p\ge 2$, there are no such nice polynomials associated with the sequences $\Delta_{i,q,m}$. This makes the situation more complex and forces us to work with the graph $G$ as a whole rather than to independently consider each vertex $i$ of $G$.

\section{Vertices leading to a vertex with no outgoing edge}
\label{sec:i-vers-infini}

In this section, we investigate the regularity of $L_i$ when there is a non-trivial path from $i$ to a vertex $\sigma^s(i)$ with no outgoing edge in the graph $G$. This means that $\mathbf{d}_i$ is finite but, after $s\ge 1$ steps, the quasi-greedy algorithm sees an infinite expansion $\mathbf{d}_{\sigma^s(i)}$ and stops. Note that this situation never occurs whenever $p=1$, that is, whenever the numeration system $U$ has a dominant root. As in the previous section, our proofs will rely on the introduction of an auxiliary sequence $\Delta$. However, the definition of $\Delta$ must be adapted to account for the finiteness of our expansions. This definition will be used in this section as well as the two following ones.

\begin{definition}
\label{def:Delta_i}
For $i\in\{0,\ldots,p-1\}$ such that $\mathbf{d}_i$ is finite and for all integers $n$ such that $np-i\ge \ell_i$, we let
\[
    (\Delta_i)_n
    =U_{np-i}-\sum_{\ell=1}^{\ell_i}t_{i,\ell-1} U_{np-i-\ell}.
\]
\end{definition}

As before, we are only concerned with such values for large $n$, so we usually omit to verify the condition $np-i\ge \ell_i$.    
We will be interested not only in the values of $\Delta_i$ for the current $i$, but also for its successors $\sigma^h(i)$ in the graph $G$, for $h\in\{0,\ldots,s-1\}$. 

Our main result in this section is the following one. In particular, it describes a previously unseen behavior when only considering systems with a dominant root.

\begin{theorem}
\label{thm:i-vers-infini}
Let $i\in\{0,\ldots,p-1\}$ be such that there exists $s\ge 1$ such that $\mathbf{d}_{\sigma^s(i)}$ is infinite but $\mathbf{d}_i,\mathbf{d}_{\sigma(i)},\ldots,\mathbf{d}_{\sigma^{s-1}(i)}$ are all finite, and
assume that the languages $L_{\sigma(i)},\ldots,L_{\sigma^s(i)}$ are regular. Then the following assertions are equivalent.

\smallskip
\begin{enumerate}[label=(\alph*)] \setlength\itemsep{0.7em}
    \item \label{item:i-vers-infini-a} The language $L_i$ is regular.
    \item \label{item:i-vers-infini-b} For all $c\ge 1$, the languages $L_{i,c}$ are regular.
    \item \label{item:i-vers-infini-c} The sequence $\Delta_i$ is ultimately periodic.
\end{enumerate}
\end{theorem}

\begin{proof}
The proof proceeds by induction on $s\ge 1$. We assume that the equivalences hold for $\sigma(i),\ldots,\sigma^{s-1}(i)$ and show them for $i$. The base case $s=1$ and the induction step will be addressed simultaneously. Note that, for the base case, our only assumption is that $L_{\sigma(i)}$ is regular since there are no previous equivalences to check.

Since the languages $L_{\sigma(i)},\ldots,L_{\sigma^{s-1}(i)}$ are assumed to be regular, we know by the recurrence hypothesis that the sequences $\Delta_{\sigma(i)},\ldots,\Delta_{\sigma^{s-1}(i)}$ are ultimately periodic, and also that for all $c\ge 1$, the languages $L_{\sigma(i),c},\ldots,L_{\sigma^{s-1}(i),c}$ are regular. Since $L_{\sigma^s(i)}$ is also assumed to be regular and $\mathbf{d}_{\sigma^s(i)}$ is infinite, we can also make use of \cref{thm:i-infini}.

Clearly, \ref{item:i-vers-infini-b} implies \ref{item:i-vers-infini-a}. We show that \ref{item:i-vers-infini-a} implies \ref{item:i-vers-infini-c}. Thus, we assume that $L_i$ is regular. 
By \cref{lem:xy*z} combined with some elementary arithmetical considerations, it can be written as a disjoint union of the form
\[
    L_i=F\cup \bigcup_{e=0}^{M-1}x_ey_e^*z_e
\]
where $F$ is a finite language, $M\ge 1$ can be chosen to be a multiple of the periods of the sequences $\Delta_{\sigma(i)},\ldots,\Delta_{\sigma^{s-1}(i)}$ and a multiple of $k\frac{m_0}{p}$ where $k$ is such that $(\Delta_{\sigma^s(i),q_0,km_0})$ is ultimately zero (with the notation of \cref{thm:i-infini}), and for each $e$, we have $|y_e|=Mp$, $|x_ez_e|\equiv -i\pmod p$, and moreover, $|x_0z_0|=tMp-i$ for some $t\ge 1$, and $|x_{e+1}z_{e+1}|=|x_ez_e|+p$ if $e<M-1$.

Our aim is to show that 
\begin{equation}
\label{eq:Delta-periodique}
    (\Delta_i)_n=(\Delta_i)_{n-M}
\end{equation}  
for all large $n$. Consider $n\ge (t+1)M$ and let $e= n\bmod M$. We have $\rep_{i,1}(n)=x_e y_e^r z_e$ and $\rep_{i,1}(n-M)=x_ey_e^{r-1}z_e$ for some $r\ge 1$. By \cref{prop:Pref-max=Pref-di}, we know that for all $L$ and all large enough $n$, the prefix of length $L$ of $\rep_{i,1}(n)$ coincides with the prefix of length $L$ of some $\mathbf{w}_{i,j}$. Due to our hypothesis on $i$, the possible $j$ are $0,1,\ldots,s$. We consider the cases $j<s$ and $j=s$ separately.

First, suppose that $j<s$. Then 
\[
    \mathbf{w}_{i,j}=d'_i\cdots d'_{\sigma^{j-1}(i)}d_{\sigma^j(i)}0^\omega.
\]
In this case, we see that $y_e$ can only contain zeros, i.e., we have $\mathbf{w}_{i,j}=x_e0^\omega$ and $y_e=0^{Mp}$. 
Then
\begin{align*}
    U_{np-i}-U_{(n-M)p-i} 
    &=(U_{np-i}-1)-(U_{(n-M)p-i}-1) \\
    &=\val_U(x_e0^{rMp}z_e)-\val_U(x_e0^{(r-1)Mp}z_e) \\
    &= \left(\sum_{h=0}^j\sum_{\ell=1}^{\ell_{\sigma^h(i)}}t_{\sigma^h(i),\ell-1} U_{np-i-k_{i,h}-\ell} 
    -\sum_{h=1}^j  U_{np-i-k_{i,h}} \right) \\
    &\quad -\left(\sum_{h=0}^j\sum_{\ell=1}^{\ell_{\sigma^h(i)}}t_{\sigma^h(i),\ell-1} U_{(n-M)p-i-k_{i,h}-\ell} 
    -\sum_{h=1}^j  U_{(n-M)p-i-k_{i,h}} \right).
\end{align*}
Rearranging the terms, this gives
\begin{align*}
    &\left(U_{np-i}-\sum_{\ell=1}^{\ell_i}t_{i,\ell-1} U_{np-i-\ell} \right)
    -\left(U_{(n-M)p-i} -\sum_{\ell=1}^{\ell_i}t_{i,\ell-1} U_{(n-M)p-i-\ell}\right) \\
    &=\quad -\sum_{h=1}^j \left(  U_{np-i-k_{i,h}}
    -\sum_{\ell=1}^{\ell_{\sigma^h(i)}}t_{\sigma^h(i),\ell-1} U_{np-i-k_{i,h}-\ell} \right) \\
    &\qquad +\sum_{h=1}^j \left( U_{(n-M)p-i-k_{i,h}} 
    -\sum_{\ell=1}^{\ell_{\sigma^h(i)}}t_{\sigma^h(i),\ell-1} U_{(n-M)p-i-k_{i,h}-\ell} \right).
\end{align*}
Using \cref{def:Delta_i} and \eqref{eq:m_ij}, this can be reexpressed as
\begin{equation}
    \label{eq:difference-Delta}
    (\Delta_i)_n-(\Delta_i)_{n-M}
    =-\sum_{h=1}^j \left( (\Delta_{\sigma^h(i)})_{n-m_{i,h}} - (\Delta_{\sigma^h(i)})_{n-m_{i,h}-M} \right)
\end{equation}
Since $M$ is a common multiple of the periods of the sequences $\Delta_{\sigma(i)},\ldots,\Delta_{\sigma^{s-1}(i)}$, every term of the sum on the right-hand side is ultimately zero. We thus get~\eqref{eq:Delta-periodique} for all large $n$ such that the corresponding $j$ is less than $s$. Note that for the base case $s=1$, this sum is in fact empty, thus the right-hand side is $0$, which leads us to the same conclusion.

Second, we consider the case $j=s$. We have 
\[
    \mathbf{w}_{i,s}=d'_i\cdots d'_{\sigma^{s-1}(i)}\mathbf{d}_{\sigma^s(i)}
\]
with $\mathbf{d}_{\sigma^s(i)}$ is ultimately periodic. Since we know that $\mathbf{d}_{\sigma^s(i)}$ is not purely periodic, see \cite[Proposition 38]{Charlier&Cisternino:2021}, we must have
\[
   x_e=d'_i\cdots d'_{\sigma^{s-1}(i)}t_{\sigma^s(i),0}\cdots t_{\sigma^s(i),q-1}
\]
and 
\[
   y_e=t_{\sigma^s(i),q}\cdots t_{\sigma^s(i),q+Mp-1}
\]
for some $q$ larger than the preperiod of $\mathbf{d}_{\sigma^s(i)}$. Similarly to the previous case, we obtain
\begin{align*}
    &U_{np-i} - U_{(n-M)p-i} \\
    &\quad= \left( 
    \sum_{h=0}^{s-1}\sum_{\ell=1}^{\ell_{\sigma^h(i)}}t_{\sigma^h(i),\ell-1} U_{np-i-k_{i,h}-\ell} 
    -\sum_{h=1}^s  U_{np-i-k_{i,h}}
    + \sum_{\ell=1}^{q+Mp} t_{\sigma^s(i),\ell-1} U_{np-i-k_{i,s}-\ell} 
    \right)  \\
    &\qquad -\left(
    \sum_{h=0}^{s-1}\sum_{\ell=1}^{\ell_{\sigma^h(i)}}t_{\sigma^h(i),\ell-1} U_{(n-M)p-i-k_{i,h}-\ell} 
    -\sum_{h=1}^s  U_{(n-M)p-i-k_{i,h}} \right. \\
    &\qquad\qquad\left. + \sum_{\ell=1}^q t_{\sigma^s(i),\ell-1} U_{(n-M)p-i-k_{i,s}-\ell} \right).
\end{align*}
Using \cref{def:Delta_i}, \cref{def:Delta_iqm} and \eqref{eq:m_ij}, we get
\begin{align*}
    (\Delta_i)_n-(\Delta_i)_{n-M} 
    &=-\sum_{h=1}^{s-1} \left( (\Delta_{\sigma^h(i)})_{n-m_{i,h}} - (\Delta_{\sigma^h(i)})_{n-m_{i,h}-M} \right) \\
    &\qquad -\left( U_{np-i-k_{i,s}} 
    - \sum_{\ell=1}^{q+Mp} t_{\sigma^s(i),\ell-1} U_{np-i-k_{i,s}-\ell}
     \right. \\
    &\qquad\qquad \left. -U_{(n-M)p-i-k_{i,s}}
    + \sum_{\ell=1}^q t_{\sigma^s(i),\ell-1} U_{(n-M)p-i-k_{i,s}-\ell} 
    \right) \\
    &=-\sum_{h=1}^{s-1} \left( (\Delta_{\sigma^h(i)})_{n-m_{i,h}} - (\Delta_{\sigma^h(i)})_{n-m_{i,h}-M} \right) + (\Delta_{\sigma^s(i),q,Mp})_{n-m_{i,s}}.
\end{align*}
As in the previous case, every term of the sum over $h$ is ultimately zero. Moreover, the additional term $(\Delta_{\sigma^s(i),q,Mp})_{n-m_{i,s}}$ is also ultimately equal to zero by using \cref{lem:Delta-q} and the link between \cref{thm:i-infini} and the choice of $M$. Therefore \eqref{eq:Delta-periodique} also holds for all large $n$ such that the corresponding $j$ is equal to $s$.

\medskip
We now move on to the proof that \ref{item:i-vers-infini-c} implies \ref{item:i-vers-infini-b}. Assume that the sequence $\Delta_i$ is ultimately periodic, say with period $M$, and let $c\ge 1$ be fixed. By \cref{lem:SplitModp}, in order to show that $L_{i,c}$ is regular, it suffices to prove that there exists $N$ such that the $M$ languages 
$\{\rep_{i,c}(nM+e) : n\ge N\}$
are regular, for $e\in\{0,\ldots,M-1\}$.

Consider some fixed $e\in\{0,\ldots,M-1\}$. By \cref{prop:Pref-max=Pref-di}, we know that for $n$ large enough, the word $\rep_{i,c}(nM+e)$ starts with $t_{i,0}\cdots t_{i,\ell_i-2}$ and the next digit is either $t_{i,\ell_i-1}$ or $t_{i,\ell_i-1}-1$. This next digit is 
\begin{align*}
    \left\lfloor\frac{U_{(nM+e)p-i}-c-\sum_{\ell=1}^{\ell_i-1}t_{i,\ell-1}U_{(nM+e)p-i-\ell}}{U_{(nM+e)p-i-\ell_i}} \right\rfloor
   & = \left\lfloor\frac{(\Delta_i)_{nM+e}-c+t_{i,\ell_i-1}U_{(nM+e)p-i-\ell_i}}{U_{(nM+e)p-i-\ell_i}} \right\rfloor \\
   & = t_{i,\ell_i-1}+\left\lfloor\frac{(\Delta_i)_{nM+e}-c}{U_{(nM+e)p-i-\ell_i}} \right\rfloor.
\end{align*}   
By hypothesis, $(\Delta_i)_{nM+e}$ is ultimately a constant. We denote this constant value by $C_{i,e}$. We separate our work based on whether $C_{i,e}\ge c$ or $C_{i,e}< c$.

First, suppose that  $C_{i,e}\ge c$. For $n$ large, the next digit of $\rep_{i,c}(nM+e)$ is $t_{i,\ell_i-1}$ and
\[
    \rep_{i,c}(nM+e)\in d_i0^*\rep_U(C_{i,e}-c).
\]
Therefore, there exists $N$ and $r$ such that
\[
    \{\rep_{i,c}(nM+e):n\ge N\}
    =d_i 0^r(0^{Mp})^*\rep_U(C_{i,e}-c),
\]
which shows that this language is regular.

Second, suppose that $C_{i,e}<c$. In this case, for large $n$, the next digit of the word $\rep_{i,c}(nM+e)$ is $t_{i,\ell_i-1}-1$ and,  using the notation~\eqref{eq:m_ij}, we obtain that
\[
   \rep_{i,c}(nM+e)= d'_i\, \rep_{\sigma(i),c-C_{i,e}}(nM+e-m_{i,1}).
\]
As a result, there exists $N$ such that
\[
    \{\rep_{i,c}(nM+e):n\ge N\}
    =d'_i L_{\sigma(i),c-C_{i,e}}
    \cap A^r(A^{Mp})^*,
\]
where $r=(nM+e)p-i$.
By the induction hypothesis if $s\ge 2$, and by \cref{thm:i-infini} for the base case $s=1$, the equivalence between \ref{item:i-vers-infini-a} and \ref{item:i-vers-infini-b} holds for $\sigma(i)$. Since the language $L_{\sigma(i)}$ is regular by hypothesis, we get  that the language $L_{\sigma(i),c-C_{i,e}}$ is regular. This allows us to conclude that the language 
$\{\rep_{i,c}(nM+e):n\ge N\}$ is regular in this case as well.
\end{proof}

\begin{remark}
\label{rem:i-vers-infini-effective}
    As in \cref{prop:i-infini-effective}, we argue that the condition \ref{item:i-vers-infini-c} of \cref{thm:i-vers-infini} can be used effectively to decide the regularity of a given $L_i$, assuming the eigenvalues of $U$ are known. Since the sequence $\Delta_i$ satisfies the same recurrence relations as $U$, its eigenvalues form a subset of those of $U$ and can be computed effectively. It then suffices to check whether these eigenvalues are all zero or roots of unity, which is equivalent to $\Delta_i$ being ultimately periodic.
\end{remark}

It should be noted that when $L_{\sigma(i)}$ (or a further successor) is not regular, $L_i$ may or may not be regular. Let us illustrate this comment with two examples.

\begin{example}
\label{ex:SuccesseurNonReg}
On the one hand, one can consider the system $U$ given by the relations $U_{2n}=3U_{2n-1}+1$ and $U_{2n-1}=2U_{2n-2}+2U_{2n-3}-U_{2n-4}-1$ and the initial conditions $(U_0,\ldots,U_5)=(1,2,7,16,49,122)$. This sequence satisfies the linear recurrence relation $U_{n+6}=9U_{n+4}-11U_{n+2}+3U_n$ and its associated alternate base is $(3,\frac{4+\sqrt{13}}{3})$. We have $\mathbf{d}_0=30^\omega$ and $\mathbf{d}_1=21^\omega$. Thus, $\sigma(0)=1$. The corresponding graph $G$ is depicted in \cref{fig:ex-L_i-L_sigma(i)}. Recall that in \cref{sec:i-infini}, we expect the period $m_0$ of the infinite expansion to be a multiple of $p$. So, with $i=1$ and $p=2$ here, we have $q_0=1$ and $m_0=2$. We find that $(\Delta_{1,1,2})_n$ is ultimately equal to $-1$. From \cref{lem:Delta-q}, we then find that $(\Delta_{1,1,2k})_n$ is ultimately equal to $-k$ for all $k\ge 1$. Therefore, the criterion \ref{item:i-infini-c} of \cref{thm:i-infini} tells us that the language $L_1$ is not regular. However, the language $L_0$ is given by $ 30(00)^*$ and thus is regular.

\smallskip 
On the other hand, if the system $U$ is defined by the relations $U_{2n}=3U_{2n-1}$ and $U_{2n-1}=2U_{2n-2}+2U_{2n-3}-U_{2n-4}-1$ and the initial conditions $(1,2,6,14,42,105)$, it satisfies the same recurrence relation, is associated with the same alternate base, and we again find that $(\Delta_{1,1,2})_n$ is ultimately equal to $-1$. Thus, $L_1$ is again not regular. However, this time we have $L_0=2\cdot L_1$, thus this language is not regular either.
    \begin{figure}[htb]
    \centering
    \begin{tikzpicture}
    \tikzstyle{every node}=[shape=circle, fill=none, draw=black,minimum size=15pt, inner sep=2pt]
    \node(0) at (0,0) {$0$};
    \node(1) at (2,0) {$1$} ;
    \tikzstyle{every path}=[color=black, line width=0.5 pt]
    \tikzstyle{every node}=[shape=circle]
    \draw [-Latex] (0) to [] node [] {} (1) ;
    \end{tikzpicture}
    \caption{The graph $G$ associated with the alternate base $B=(3,\frac{4+\sqrt{13}}{3})$.}
    \label{fig:ex-L_i-L_sigma(i)}
    \end{figure}
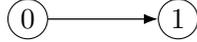
\end{example}

\section{Vertices in a cycle}
\label{sec:i-cycle}

In this section, we examine the regularity of $L_i$ when $i$ is part of a cycle in the graph $G$. In this case, there exists some $r$ such that $\sigma^r(i)=i$. Consequently, $\mathbf{d}_i,\ldots,\mathbf{d}_{\sigma^{r-1}(i)}$ are all finite and $\mathbf{d}_i^*=(d'_i\cdots d'_{\sigma^{r-1}(i)})^\omega$. We know that, up to taking large enough $n$, the words $\rep_U(U_{np-i}-1)$ share prefixes of arbitrary length with one of the $\mathbf{w}_{i,j}$'s. However, these choices for all vertices of the cycle are not independent of one another, which brings additional information on top of \cref{prop:Pref-max=Pref-di}. This observation will be made clear shortly as it will be one of the main arguments in what follows.

The following result gives a necessary condition for the regularity of all languages "in the cycle" that will allow us to focus on ultimately periodic sequences.

\begin{proposition}
\label{prop:cycle-per}
Let $i\in\{0,\ldots,p-1\}$ be such that there exists $r\ge 1$ such that $\sigma^r(i)=i$. If the languages $L_i,L_{\sigma(i)},\ldots, L_{\sigma^{r-1}(i)}$ are all regular, then the sequences $\Delta_i,\Delta_{\sigma(i)},\ldots, \Delta_{\sigma^{r-1}(i)}$ are all ultimately periodic. 
\end{proposition}

\begin{proof}
Suppose that the languages $L_i,L_{\sigma(i)},\ldots, L_{\sigma^{r-1}(i)}$ are all regular. Then by ~\cref{lem:xy*z} combined with arithmetic considerations, they can be decomposed in disjoint unions as follows. For every $h\in\{0,\ldots,r-1\}$, we have
\[
    L_{\sigma^h(i)}
    = F_h \cup \bigcup_{e=0}^{M-1} x_{h,e}y_{h,e}^*z_{h,e}
\]
where $F_h$ is a finite language, $M\ge 1$ with $|y_{h,e}|=Mp$, $|x_{h,e}z_{h,e}|\equiv -\sigma^h(i)\pmod p$, $|x_{h,0}z_{h,0}|=tMp-\sigma^h(i)$ for some $t$, and $|x_{h,e+1}z_{h,e+1}|=|x_{h,e}z_{h,e}|+p$ for each $e$. Without loss of generality, we can ask that $Mp$ is a multiple of $k_{i,r}$, which is the sum of the lengths of the finite expansions $\mathbf{d}_{\sigma^h(i)}$ for $h\in\{0,\ldots,r-1\}$. We let $M'$ be this multiple, so that we have
\[
    Mp=M'k_{i,r}=k_{i,M'r}.
\]

Given the symmetry of the situation, in order to show the result it is enough to prove that $\Delta_i$ is ultimately periodic. To this end, we will show that 
\begin{equation}
\label{eq:Delta-i-periodique}
    (\Delta_i)_n=(\Delta_i)_{n-M}
\end{equation}
for large $n$.

Consider $e\in\{0,\ldots,M-1\}$ and $n\ge tM$ (or larger if needed) such that $n\equiv e\pmod M$. Thus, we have $\rep_U(U_{np-i}-1)\in x_{0,e} y_{0,e}^* z_{0,e}$.
By \cref{prop:Pref-max=Pref-di}, we know that for all $L$ and all large enough $n$, the prefix of length $L$ of $\rep_U(U_{np-i}-1)$ coincides with the prefix of length $L$ of some $\mathbf{w}_{i,j}$. This implies that either $x_{0,e} y_{0,e}^\omega=\mathbf{w}_{i,j_0}$ for some $j_0\ge 0$ or $x_{0,e} y_{0,e}^\omega=\mathbf{d}^*_i$. We consider these two cases separately.

\medskip
First, we suppose that $x_{0,e} y_{0,e}^\omega=\mathbf{w}_{i,j_0}$ for some $j_0\ge 0$. Then $y_{0,e}=0^{Mp}$. Now for $h\in\{1,\ldots,j_0\}$ and $e'=(e-m_{i,h})\bmod M$, we also have that $x_{h,e'}y_{h,e'}^\omega$ is either $\mathbf{w}_{\sigma^h(i),j_h}$ for some $j_h\ge 0$ or $\mathbf{d}_{\sigma^h(i)}^*$. However, in our case only the former is possible, with $j_h\leq j_0-h$. Indeed, otherwise the suffix of length $np-i-k_{i,h}$ of $\rep_U(U_{np-i}-1)$ would be lexicographically greater than $\rep_U(U_{np-i-k_{i,h}}-1)$ for large $n$, in contradiction with \cref{lem:Suffixes}. Similarly to \eqref{eq:difference-Delta} and using that $m_{i,h+j}=m_{i,h}+m_{\sigma^h(i),j}$, we obtain
\begin{align}
        \label{eq:difference-Delta-cycle}
    (\Delta_{\sigma^h(i)})_{n-m_{i,h}}
    &-(\Delta_{\sigma^h(i)})_{n-m_{i,h}-M} \\ \nonumber
    & =-\sum_{j=1}^{j_h} \left( (\Delta_{\sigma^{h+j}(i)})_{n-m_{i,h+j}} - (\Delta_{\sigma^{h+j}(i)})_{n-m_{i,h+j}-M} \right).
\end{align}
We can now prove by descending induction on $h\in\{0,\ldots,j_0\}$ that 
\[
    (\Delta_{\sigma^h(i)})_{n-m_{i,h}}-(\Delta_{\sigma^h(i)})_{n-m_{i,h}-M}
\]
is ultimately equal to $0$. If $h=j_0$ then $j_h=0$ and the right-hand side of~\eqref{eq:difference-Delta-cycle} is trivially zero. Now, let $h\in\{0,\ldots,j_0-1\}$ and suppose that the claim holds for $h+1,\ldots,j_0$. Then each term of the sum of the right-hand side of~\eqref{eq:difference-Delta-cycle} is ultimately zero by induction hypothesis since $h+j\in\{h+1\ldots, j_0\}$ as we have $h+j_h\le j_0$.

\medskip
Second, we suppose that $x_{0,e} y_{0,e}^\omega=\mathbf{d}^*_i$. Let us argue that, in this case, for every $h\in\{1,\ldots, M'r-1\}$, the word $\rep_U(U_{np-i-k_{i,h}}-1)$ has long common prefixes either with some $\mathbf{w}_{\sigma^h(i),j_h}$ with $j_h<M'r-h$ or with $\mathbf{d}^*_{\sigma^h(i)}$. Indeed, we know that the word $\rep_U(U_{np-i-Mp}-1)$ also shares long prefixes with $x_{0,e} y_{0,e}^\omega$, which we are assuming is $\mathbf{d}^*_i$. Since $k_{i,M'r}=Mp$, we see that $\rep_U(U_{np-i-k_{i,h}}-1)$ cannot share long common prefixes with some $\mathbf{w}_{\sigma^h(i),j_h}$ with $j_h\ge M'r-h$ for otherwise, its suffix of length $np-i-k_{i,M'r}$ would be lexicographically greater than $\rep_U(U_{np-i-Mp}-1)$, in contradiction with \cref{lem:Suffixes}.

We define a finite sequence $(h_q)_{1\le q\le Q}$ as follows. Set $h_1=1$. Then, if $h_q$ is defined and is less than $M'r$, and if there is some $j_{h_q}<M'r-h_q$ such that $\rep_U(U_{np-i-k_{i,h_q}}-1)$ has long common prefixes with $\mathbf{w}_{\sigma^{h_q}(i),j_{h_q}}$, then we set $h_{q+1}=h_q+j_{h_q}+1$. Since this sequence is increasing and bounded above by $M'r$, it must end with some $h_Q$. From the previous paragraph, the word $\rep_U(U_{np-i-k_{i,h_Q}}-1)$ has long common prefixes with $\mathbf{d}^*_{\sigma^{h_Q}(i)}$.

We thus have the following factorizations: first, 
\[
    \rep_U(U_{np-i}-1)= d'_id'_{\sigma(i)}\cdots d'_{\sigma^{h_Q-1}(i)}wz_{0,e}
\]
where $w$ is a prefix of $\mathbf{d}^*_{\sigma^{h_Q}(i)}$; then, for every $q\in\{1,\ldots,Q-1\}$, there exists $e_q\in\{0,\ldots,M-1\}$ such that
\[
    \rep_U(U_{np-i-k_{i,h_q}}-1)\in d'_{\sigma^{h_q}(i)}d'_{\sigma^{h_q+1}(i)}\cdots d'_{\sigma^{h_{q+1}-2}(i)}d_{\sigma^{h_{q+1}-1}(i)}0^*z_{\sigma^{h_q}(i),e_q}
\] 
and finally, there exists $e_Q\in\{0,\ldots,M-1\}$ such that
\[
    \rep_U(U_{np-i-k_{i,h_Q}}-1)= w'z_{\sigma^{h_Q}(i),e_Q}
\]
where $w'$ is a prefix of $\mathbf{d}^*_{\sigma^{h_Q}(i)}$.

Using these expansions, we can write
\begin{align*}
    0&=U_{np-i}-1-\sum_{h=0}^{h_Q-1}\sum_{\ell=1}^{\ell_{\sigma^h(i)}} t_{\sigma^h(i),\ell-1} U_{np-i-k_{i,h}-\ell}
    +\sum_{h=1}^{h_Q} U_{np-i-k_{i,h}} 
    - \val_U(wz_{0,e}) \\
    &-\sum_{q=1}^{Q-1}\left(U_{np-i-k_{i,h_q}}-1 
    - \sum_{h=h_q}^{h_{q+1}-1}\sum_{\ell=1}^{\ell_{\sigma^h(i)}} t_{\sigma^h(i),\ell-1} U_{np-i-k_{i,h}-\ell} \right. \\
    & \qquad \qquad \left. +\sum_{h=h_{q}+1}^{h_{q+1}-1} U_{np-i-k_{i,h}} 
    - \val_U\big(z_{\sigma^{h_q}(i),e_q}\big) \right) \\
    & - \left(U_{np-i-k_{i,h_Q}} -1 - \val_U\big(w'z_{\sigma^{h_Q}(i),e_Q}\big) \right).
\end{align*}
Multiple cancellations yield 
\[
    (\Delta_i)_n= \val_U(wz_{0,e}) -\val_U\big(w'z_{\sigma^{h_Q}(i),e_Q}\big) 
                - \sum_{q=1}^{Q-1} \big(\val_U\big(z_{\sigma^{h_q}(i),e_q}\big)+1\big).
\]
In the right-hand side of this equality, only $w$ and $w'$ depend on $n$, since the sequence $(h_q)_{1\le q\le Q}$ and the suffixes $z_{0,e}$ and $z_{\sigma^{h_q}(i),e_q}$ only depend on $e$.
Since $w$ and $w'$ are both prefixes of the same infinite word, one must be a prefix of the other. Since moreover 
\[
    |wz_{0,e}|=|w'z_{\sigma^{h_Q}(i),e_Q}|=np-i-k_{i,h_Q},
\]
we get that
\[
    \val_U(wz_{0,e})-\val_U\big(w'z_{\sigma^{h_Q}(i),e_Q}\big)
    =\begin{cases}
        \val_U(z_{0,e})-\val_U\big(vz_{\sigma^{h_Q}(i),e_Q}\big), & \text{if } w'=wv; \\
        \val_U(vz_{0,e})-\val_U\big(z_{\sigma^{h_Q}(i),e_Q}\big), & \text{if } w=w'v.
    \end{cases}
\]
Because $|v|$ is either $|z_{0,e}|-|z_{\sigma^{h_Q}(i),e_Q}|$ or $|z_{\sigma^{h_Q}(i),e_Q}|-|z_{0,e}|$, this value only depends on $e$.
This proves that~\eqref{eq:Delta-i-periodique} hold for all large $n$, as desired.
\end{proof}

In order to state the main result of this section, we need a new definition. This definition and the following lemma are stated in the more general case of a vertex that is either in a cycle or leading to a cycle in the graph $G$, as they will be used in both the current and the next sections.

\begin{definition}
\label{def:sum-Delta}
    Let $i\in\{0,\ldots,p-1\}$ be such that there exist $s\ge 0$ and $r\ge 1$ such that $\sigma^{s+r}(i)=\sigma^s(i)$. For all integers $j\ge 0$ and all integers $n$ such that $np-i\ge k_{i,j}$, we let
    \[
        (\Delta_i^{(j)})_n = \sum_{h=0}^{j-1} (\Delta_{\sigma^h(i)})_{n-m_{i,h}}.
    \]
\end{definition}

That is, $(\Delta_i^{(j)})_n$ is the cumulative sum of $j$ values of the sequences $\Delta_{\sigma^h(i)}$, for $h\in\{0,\ldots,j\}$. These values are taken along relevant positions with respect to the execution of the greedy algorithm when representing $U_{np-i}-1$. The following result is what motivates the previous definition. 

\begin{lemma}
\label{lem:tech}
    Let $i\in\{0,\ldots,p-1\}$ be such that there exist $s\ge 0$ and $r\ge 1$ such that $\sigma^{s+r}(i)=\sigma^s(i)$. For all integers $c\ge 1$ and $j\ge 0$, there exists $N$ such that for all $n\ge N$,
    
    \smallskip
    \begin{itemize}\setlength\itemsep{0.7em}
    \item if $(\Delta_i^{(h)})_n<c$ for all $h\in\{1,\ldots,j\}$, then
    \[
        \rep_{i,c}(n) = d'_id'_{\sigma(i)}\cdots d'_{\sigma^{j-1}(i)}
        \rep_{\sigma^j(i),c_n}(n-m_{i,j}),
    \]
    where $c_n=c-(\Delta_i^{(j)})_n$;
    \item if $(\Delta_i^{(h)})_n<c$ for all $h\in\{1,\ldots,j-1\}$ and $(\Delta_i^{(j)})_n\ge c$, then
    \[
        \rep_{i,c}(n) 
        \in d'_id'_{\sigma(i)}\cdots d'_{\sigma^{j-2}(i)} 
        d_{\sigma^{j-1}(i)} 0^* \rep_U\big((\Delta_i^{(j)})_n-c\big).
    \]
    \end{itemize}
\end{lemma}

\begin{proof}
    We proceed by induction on $j$. The result is trivial for $j=0$. Let thus $j\ge 1$ be such that the result holds for $j-1$, and let $c\ge 1$. By induction hypothesis, there exists $N_1$ such that for all $n\ge N_1$ such that 
    \begin{equation}
        \label{eq:hypo-j-1}
        (\Delta_i^{(h)})_n<c\quad \text{ for }h\in\{1,\ldots,j-1\},
    \end{equation}
    we have
    \[
        \rep_{i,c}(n) 
        = d'_id'_{\sigma(i)}d'_{\sigma^{j-2}(i)} \rep_{\sigma^{j-1}(i),b_n}(n-m_{i,j-1})
    \]
    where $b_n=c-(\Delta_i^{(j-1)})_n$.
    Therefore, and by using \cref{prop:Pref-max=Pref-di}, there exists $N_2\ge N_1$ such that for all $n\ge N_2$ such that the inequalities \eqref{eq:hypo-j-1} hold, the prefix of length $k_{i,j}$ of $\rep_{i,c}(n)$   coincides with that of either $\mathbf{w}_{i,j-1}$ or $\mathbf{w}_{i,j}$.
    Thus, for such $n$, the prefix of length $\ell_{\sigma^{j-1}(i)}$ of $\rep_{\sigma^{j-1}(i),b_n}(n-m_{i,j-1})$ is either $d'_{\sigma^{j-1}(i)}$ or $d_{\sigma^{j-1}(i)}$. The last digit of this prefix is given by
    \[
        \left\lfloor 
        \frac{U_{np-i-k_{i,j-1}}-c+(\Delta_i^{(j-1)})_n - \sum_{\ell=1}^{\ell_{\sigma^{j-1}(i)}-1} t_{\sigma^{j-1}(i),\ell-1} U_{np-i-k_{i,j-1}-\ell} }{U_{np-i-k_{i,j}}}
        \right\rfloor. 
    \]
    Using \cref{def:Delta_i,def:sum-Delta}, this digit can be rewritten as
    \begin{align*}
        &\left\lfloor 
        \frac{(\Delta_i^{(j-1)})_n + (\Delta_{\sigma^{j-1}(i)})_{n-m_{i,j-1}}-c}{U_{np-i-k_{i,j}}}
        \right\rfloor + t_{\sigma^{j-1}(i),\ell_{\sigma^{j-1}(i)}-1} \\
        & \quad = \left\lfloor 
        \frac{(\Delta_i^{(j)})_n-c}{U_{np-i-k_{i,j}}} 
        \right\rfloor + t_{\sigma^{j-1}(i),\ell_{\sigma^{j-1}(i)}-1}.
    \end{align*}
    If $(\Delta_i^{(j)})_n\ge c$, then this digit is $t_{\sigma^{j-1}(i),\ell_{\sigma^{j-1}(i)}-1}$ and
    \[
       \rep_{i,c}(n)
        \in d'_id'_{\sigma(i)}\cdots d'_{\sigma^{j-2}(i)} 
        d_{\sigma^{j-1}(i)} 0^* \rep_U\big((\Delta_i^{(j)})_n-c\big)
    \]
    as expected. Otherwise, if $(\Delta_i^{(j)})_n< c$, then this digit is $t_{\sigma^{j-1}(i),\ell_{\sigma^{j-1}(i)}-1}-1$ and 
    \[
       \rep_{i,c}(n)
       = d'_id'_{\sigma(i)}\cdots d'_{\sigma^{j-1}(i)}
       \rep_{\sigma^j(i),c_n}(n-m_{i,j})
    \]
    where $c_n=c-(\Delta_i^{(j)})_n$,
    as expected.
\end{proof}

We will also need the following lemma on circular sums.

\begin{lemma}
\label{lem:SequenceRotation}
    Consider a finite sequence $\delta_0,\ldots,\delta_{M-1}$ such that $\sum_{m=0}^{M-1} \delta_m<0$, and set $\delta_m=\delta_{m\bmod M}$ for $m\ge M$. Then there exists $j\in\{0,\ldots,M-1\}$ such that for every $t\in\{0,\ldots,M-1\}$, we have $\sum_{m=j}^{j+t} \delta_m<0$.
\end{lemma}

\begin{proof}
Let $j\in \{0,\ldots,M-1\}$ be such that $\sum_{m=0}^{j-1} \delta_m$ is maximal, and choose $j$ maximal among all possible such values. If $t\in\{0,\ldots,M-1-j\}$, then 
\[
	\sum_{m=j}^{j+t} \delta_m
	= \sum_{m=0}^{j+t} \delta_m - \sum_{m=0}^{j-1} \delta_m
	<0
\]
by choice of $j$. If $t\in\{M-j,\ldots,M-1\}$, then we get that
\[
	\sum_{m=j}^{j+t} \delta_m  
	=\sum_{m=j}^{M-1} \delta_m + \sum_{m=M}^{j+t} \delta_m
    =\sum_{m=0}^{M-1} \delta_m - \sum_{m=0}^{j-1} \delta_m + \sum_{m=0}^{j+t-M} \delta_m
 	<0,
\]
since the first sum is negative by assumption, and the third is less or equal than the second by choice of $j$. 
\end{proof}

We introduce one last definition.

\begin{definition}
\label{def:constant-Delta}
Let $i\in\{0,\ldots,p-1\}$ be such that there exists $r\ge 1$ such that $\sigma^r(i)=i$ and such that the sequences $\Delta_i,\Delta_{\sigma(i)},\ldots,\Delta_{\sigma^{r-1}(i)}$ are all ultimately periodic with a common period $M$. For all integers $j\ge 0$ and all $e\in\{0,\ldots,M-1\}$, we let $\Delta_{i,e,j}$ denote the ultimate constant value of $(\Delta_{\sigma^j(i)})_{nM+e-m_{i,j}}$ and we let $\Delta^{(j)}_{i,e}$ denote the ultimate constant value of $(\Delta^{(j)}_i)_{nM+e}$, so that we have
\[
    \Delta^{(j)}_{i,e}=\sum_{h=0}^{j-1} \Delta_{i,e,h}.
\]
\end{definition}

We are ready to state the main result of this section.

\begin{theorem}
\label{thm:i-cycle}
Let $i\in\{0,\ldots,p-1\}$ be such that there exists $r\ge 1$ such that $\sigma^r(i)=i$, in which case $\mathbf{d}_i,\ldots,\mathbf{d}_{\sigma^{r-1}(i)}$ are finite and $\mathbf{d}^*_i=(d'_i\cdots d'_{\sigma^{r-1}(i)})^\omega$. We assume that the sequences $\Delta_i,\Delta_{\sigma(i)},\ldots,\Delta_{\sigma^{r-1}(i)}$ are all ultimately periodic with a common period $M$ such that $Mp=M'k_{i,r}$ with $M'\ge 1$. 

\smallskip
The following assertions are equivalent.

\smallskip
\begin{enumerate}[label=(\alph*)] \setlength\itemsep{0.7em}
    \item \label{item:i-cycle-a} The languages $L_i,L_{\sigma(i)}\ldots,L_{\sigma^{r-1}(i)}$ are all regular.
    \item \label{item:i-cycle-b} For all $c\ge 1$, the languages $L_{i,c},L_{\sigma(i),c}\ldots,L_{\sigma^{r-1}(i),c}$ are all regular.
    \item \label{item:i-cycle-c} For all $e\in \{0,\ldots,\frac{k_{i,r}}{p}-1\}$ and $j\in\{0,\ldots,r-1\}$, we have $\Delta^{(M'r)}_{\sigma^j(i),e}\ge 0$.
    \item \label{item:i-cycle-d} For all $e\in \{0,\ldots,\frac{k_{i,r}}{p}-1\}$, we have $\Delta^{(M'r)}_{i,e}\ge 0$.
\end{enumerate}
\end{theorem}

\begin{proof}
    We start by showing a few properties which hold in the setting of this theorem. 
    For all $j,h\ge 0$, we have $k_{i,j+h}=k_{i,j}+k_{\sigma^j(i),h}$ and $m_{i,j+h}=m_{i,j}+m_{\sigma^j(i),h}$. Therefore, for all $e\in\{0,\ldots,M-1\}$, we have 
    \begin{equation}
    \label{eq:Delta_e'}
        \Delta_{i,e,j+h}=\Delta_{\sigma^j(i),e',h}
    \end{equation}
    where $e'=(e-m_{i,j})\bmod M$. In particular, note that $e'$ does not depend on $h$. From our hypotheses, we have that $k_{\sigma^j(i),M'r}=k_{i,M'r}=M'k_{i,r}=Mp$. We obtain that  $k_{i,j+M'r}=k_{i,j}+Mp$ and $m_{i,j+M'r}= m_{i,j}+M$. Hence
    \begin{equation}
    \label{eq:Delta_ieh-periodique}
        \Delta_{i,e,j+M'r}=\Delta_{i,e,j}.
    \end{equation}
    We get that, on the one hand, 
    \begin{equation}
    \label{eq:sum-Delta-M'r-ij}
        \Delta^{(M'r)}_{i,e}=\Delta^{(M'r)}_{\sigma^j(i),e'}
    \end{equation}
    and, on the other hand, 
    \begin{equation}
    \label{eq:sum-Delta-cM'r}
        \Delta^{(c M'r)}_{i,e}=c\Delta^{(M'r)}_{i,e}.
    \end{equation}
    for all $c\ge 1$. In particular, since $\sigma^r(i)=i$, note that $k_{i,r}=m_{i,r}p$ and that the relation \eqref{eq:sum-Delta-M'r-ij} with $j=r$ becomes $\Delta^{(M'r)}_{i,e}=\Delta^{(M'r)}_{i,e'}$ where $e'=(e-\frac{k_{i,r}}{p})\bmod M$. Therefore, assuming that $\Delta^{(M'r)}_{\sigma^j(i),e}\ge 0$ for all $e\in \{0,\ldots,\frac{k_{i,r}}{p}-1\}$ implies that $\Delta^{(M'r)}_{\sigma^j(i),e}\ge 0$ for all $e\in \{0,\ldots,M-1\}$.
    
    We obtain from the previous paragraph that \ref{item:i-cycle-d} implies \ref{item:i-cycle-c}. The fact that \ref{item:i-cycle-b} implies \ref{item:i-cycle-a} is obvious. 
    
    In order to show that \ref{item:i-cycle-c} implies \ref{item:i-cycle-b}, we will show the stronger fact that if for all $e\in \{0,\ldots,\frac{k_{i,r}}{p}-1\}$, we have $\Delta^{(M'r)}_{i,e}\ge 0$, then for all $c\ge 1$, the language $L_{i,c}$ is regular. Thus, we consider a fixed $c\ge 1$, and we suppose that $\Delta^{(M'r)}_{i,e}\ge 0$ for all $e\in \{0,\ldots,\frac{k_{i,r}}{p}-1\}$.
    From~\cref{lem:SplitModp}, to get the regularity of the language $L_{i,c}$, it is enough to prove that the languages
    \[
        L_{i,c,e}=\{\rep_{i,c}(n) : n\ge 1,\ n\equiv e\pmod M,\ U_{np-i}\ge c\}.
    \]
    are regular for all $e\in\{0,\ldots,M-1\}$. We fix such an $e$ and consider two cases.
    
    First, suppose that there exists $q\in\{1,\ldots, cM'r\}$ such that $\Delta^{(q)}_{i,e}\ge c$.
    We choose a minimal such $q$. Thus we have $\Delta^{(q')}_{i,e}< c$ for $q'<q$. By \cref{lem:tech}, there exists $N$ such that for all $n\ge N$, we have
    \[
        \rep_{i,c}(nM+e)
        \in d'_id'_{\sigma(i)}\cdots d'_{\sigma^{q-2}(i)} d_{\sigma^{q-1}(i)}
        0^* \rep_U\big(\Delta^{(q)}_{i,e}-c\big).
    \]
    We get that $L_{i,c,e}$ is regular as 
    \[
        L_{i,c,e}= F\cup \left(d'_id'_{\sigma(i)}\cdots d'_{\sigma^{q-2}(i)} d_{\sigma^{q-1}(i)}
        0^* \rep_U\big(\Delta^{(q)}_{i,e}-c\big)\cap \big(A_U^{Mp}\big)^*A_U^{(NM+e)p-i} \right)
    \]
    where $F$ is a finite language. 
    
    Second, suppose that $\Delta^{(q)}_{i,e}< c$ for all $q\in\{1,\ldots, cM'r\}$. Combining our assumption that $\Delta^{(M'r)}_{i,e}\ge 0$ with \eqref{eq:sum-Delta-cM'r}, we get that in fact $\Delta^{(M'r)}_{i,e}=0$. By \cref{lem:tech} and by using that $m_{i,cM'r}=cM$, there exists $N$ such that for all $n\ge N$, we have
    \begin{align*}
        \rep_{i,c}(nM+e)
        &= d'_id'_{\sigma(i)}\cdots d'_{\sigma^{cM'r-1}(i)} \rep_{i,c}(nM+e-m_{i,cM'r}) \\
        &= d'_id'_{\sigma(i)}\cdots d'_{\sigma^{cM'r-1}(i)} \rep_{i,c}((n-c)M+e).
    \end{align*}
    We iterate this argument $\ell$ times until $n-\ell c<N$. We obtain again that $L_{i,c,e}$ is regular as 
    \[
        L_{i,c,e}= F\cup \left(d'_id'_{\sigma(i)}\cdots d'_{\sigma^{cM'r-1}(i)}\right)^* G
    \]
    where $F$ and $G$ are finite languages, namely, 
    \begin{align*}
        F&=\{\rep_{i,c}(nM+e) : 1\le n<N-c,\ U_{(nM+e)p-i}\ge c\} \\
        G&=\{\rep_{i,c}((N-t) M+e): 1\le t\le c\}.
    \end{align*}
    
    \medskip
    We turn our attention to show that \ref{item:i-cycle-a} implies \ref{item:i-cycle-d}. We proceed by contraposition. Thus, we assume that there is some $e\in\{0,\ldots,M-1\}$ with $\Delta^{(M'r)}_{i,e}<0$.
    We have to show that one of the $r$ languages $L_i,L_{\sigma(i)}\ldots,L_{\sigma^{r-1}(i)}$ is not regular. In view of \eqref{eq:Delta_ieh-periodique} and by \cref{lem:SequenceRotation}, there exists $j\in\{0,\ldots,M'r-1\}$ such that for all $t\in\{1,\ldots,M'r\}$, we have
    \[
        \sum_{h=j}^{j+t-1} \Delta_{i,e,h}<0.
    \]
    We are going to show that the language $L_{\sigma^j(i)}$ is not regular.
    Using \eqref{eq:Delta_e'}, the latter sum can be reexpressed as
    \[
        \sum_{h=0}^{t-1} \Delta_{i,e,j+h}
        =\sum_{h=0}^{t-1} \Delta_{\sigma^j(i),e',h}
        =\Delta^{(t)}_{\sigma^j(i),e'}
    \]
    where $e'=(e - m_{i,j}) \bmod M$. In view of \eqref{eq:sum-Delta-M'r-ij} and \eqref{eq:sum-Delta-cM'r}, for all $t\in\{1,\ldots,M'r\}$ and all $c\ge 1$, we have
    \[
        \Delta^{(t+cM'r)}_{\sigma^j(i),e'}
        =\Delta^{(t)}_{\sigma^j(i),e'}+c\Delta^{(M'r)}_{\sigma^{j+t}(i),e''}
        =\Delta^{(t)}_{\sigma^j(i),e'}+c\Delta^{(M'r)}_{i,e}
        <0
    \]
     where $e''=(e-m_{i,j+t}) \bmod M$. We obtain that 
     \[
        \Delta^{(t)}_{\sigma^j(i),e'}<0
     \]
     for all $t\ge 1$.

    Now, let us fix some $C\ge 1$. \cref{lem:tech} ensures that there exists $n$ such that for all $c\in\{0,\ldots, C-1\}$, we have
    \[
        \rep_{\sigma^j(i),1}((n+c)M+e')
        = d'_{\sigma^j(i)}d'_{\sigma^{j+1}(i)}\cdots d'_{\sigma^{j+cM'r-1}(i)} 
        \rep_{\sigma^j(i),c'}(nM+e')
    \]
    where we have set $c'=1-c\Delta^{(M'r)}_{i,e}$, and where we have used~\eqref{eq:sum-Delta-M'r-ij} and \eqref{eq:sum-Delta-cM'r}, as well as the equality $m_{\sigma^j(i),cM'r}=cM$. In particular, the $C$ suffixes 
    \[
        \rep_{\sigma^j(i),c'}(nM+e'),\quad \text{ for }c'\in\{1-c\Delta^{(M'r)}_{i,e}: 0\le c<C\},
    \]
    are distinct and all of the same length $(nM+e')p-\sigma^j(i)$. 
    This implies that $L_{\sigma^j(i)}$ is not regular as slender regular languages have a uniformly bounded number of suffixes of the same length by \cref{lem:xy*z}.
\end{proof}

\begin{remark}
\label{rem:i-cycle-effective}
    As in the previous sections, \cref{thm:i-cycle} can be used effectively in order to decide the regularity of all the languages $L_i$ from a cycle. If the eigenvalues of $U$ are known, then those of the sequences $\Delta_{\sigma^j(i)}$ can be computed, allowing us to test whether they are all ultimately periodic. If this is indeed the case, then the values of $M,M'$ and $\Delta_{i,e}^{(M'r)}$ can be computed, and the condition \ref{item:i-cycle-d} of \cref{thm:i-cycle} can be tested.
\end{remark}

In the dominant root case, this section concerns the case where the dominant root is a simple Parry number, i.e., $d_\beta(1)=t_0\ldots t_{\ell-1}0^\omega$. In this situation, the graph $G$ contains a single vertex with a loop, and \cref{prop:cycle-per,thm:i-cycle} reduce to the following corollary, where the sequence $\Delta=(\Delta_n)_{n\ge \ell}$ is defined by
\[
    \Delta_n=U_n-\sum_{k=1}^\ell t_{k-1}U_{n-k}.
\]

\begin{corollary}
\label{cor:simple-Parry}
    Let $U$ be a positional numeration system with a dominant root $\beta \ge  1$ such that $d_\beta(1)$ is finite of length $\ell$, i.e., $\beta$ is a simple Parry number or $\beta=1$.

    \smallskip
    \begin{itemize} \setlength\itemsep{0.7em}
        \item If the numeration language $L_U$ is regular, then the sequence $\Delta$ is ultimately periodic.
        \item Assume that the sequence $\Delta$ is ultimately periodic with a preperiod $N\ge \ell$ and a period $M=M'\ell$ with $M'\ge 1$. Then the numeration language $L_U$ is regular if and only if 
        \[
            \sum_{j=0}^{M'-1}\Delta_{n-j\ell}\ge 0
        \]
        for all $n\in\{N+M-\ell,\ldots,N+M-1\}$. 
    \end{itemize} 
\end{corollary}

\begin{proof}
    This follows from \cref{prop:cycle-per,thm:i-cycle} where the notation boils down to $p=1$, $i=0$, $r=1$, $k_{i,r}=\ell$.
\end{proof}

In comparison with \cite{Hollander:1998}, this result is new. It provides a necessary and sufficient condition for the regularity of $L_U$ in the simple Parry dominant root case, which is precisely the case that Hollander did not solve entirely. In particular, Hollander exhibited an example showing that the regularity of $L_U$ can depend on the initial conditions, and not only on the characteristic polynomial of a linear recurrence satisfied by $U$. The example is the following. Suppose that $U$ satisfies the linear recurrence of characteristic polynomial $(X-1)(X-3)$. This exactly means that the sequence $\Delta$ is constant as this sequence is given by $\Delta_n=U_n-3U_{n-1}$, thus it satisfies the linear recurrence of characteristic polynomial $X-1$. As shown by Hollander, by choosing the initial conditions $U_0=1$, $U_1=4$, we obtain a regular numeration language, whereas by choosing the initial conditions $U_0=1$, $U_1=2$, we obtain a non-regular numeration language. This situation was not handled by Hollander's result but is covered by \cref{cor:simple-Parry}. The notation drastically reduces since the dominant root of the system is $3$ and $\Delta$ is constant, hence we get $\ell=N=M=M'=1$. For the initial conditions $U_0=1$, $U_1=4$, we get that $U_n=3U_{n-1}+1$ and $\Delta_n=1$ for all $n\ge 1$. Therefore, \cref{cor:simple-Parry} tells us that $L_U$ is regular. However, for the initial conditions $U_0=1$, $U_1=2$, we get that $U_n=\frac{3^n+1}{2}$ and $\Delta_n=-1$ for all $n\ge 1$. In this case, \cref{cor:simple-Parry} yields that $L_U$ is not regular. 

The results of this section also include the case where $U$ has $1$ as a dominant root. We note that this case was not considered at all in \cite{Hollander:1998}. We discuss this particular case in the next corollary.

\begin{corollary}
\label{cor:beta=1}
     Let $U$ be a positional numeration system with the dominant root $1$.
    Then the numeration language $L_U$ is regular if and only if the sequence $(U_{n+1}-U_n)_{n\ge 0}$ is ultimately periodic. 
    
    In particular, if the sequence $U$ is ultimately a polynomial, i.e., there exists a polynomial $P\in\C[X]$ and an integer $N\ge 0$ such that $U_n=P(n)$ for all $n\ge N$, then the numeration language $L_U$ is regular if and only if this polynomial has integer coefficients and degree $1$, i.e., there exists $a,b\in\Z$ with $a>0$ such that $U_n=an+b$ for all $n\ge N$.
\end{corollary}

\begin{proof}
     As for \cref{cor:simple-Parry}, under the dominant root hypothesis, the notation of this section reduces to $p=1$, $i=0$, $r=1$, $k_{i,r}=\ell$. Moreover, since the dominant root is $1$, we also have $\ell=1$ and the sequence $\Delta$ is given by $\Delta_n=U_n-U_{n-1}$ for $n\ge 1$. By \cref{cor:simple-Parry}, if $L_U$ is regular then $\Delta$ is ultimately periodic. Let us argue that the converse also holds. Suppose that $\Delta$ is ultimately periodic with preperiod $N\ge 1$ and period $M\ge 1$. With the notation of \cref{cor:simple-Parry}, we have $M'=M$ and in order to obtain that $L_U$ is regular, it suffices to show that 
     \[
        \sum_{j=0}^{M-1}\Delta_{N+j}=\sum_{j=0}^{M-1}(U_{N+j}-U_{N+j-1})\ge 0.
     \]
     This is clearly the case as $U$ is an increasing sequence of integers.

     The particular case is straightforward.
\end{proof}

\begin{example}
The sequence $U$ given by the initial conditions $U_0=1,U_1=2$ and the relation $U_n=U_{n-2}+3$ for $n\ge 2$ yields a regular numeration language as $(U_{n+1}-U_n)_{n\ge 1}=(1,2,1,2,\ldots)$. Indeed, it easily checked that $L_U=10^*\cup 1(00)^*1\cup \{\varepsilon\}$. On the other hand, Shallit proved in \cite{Shallit:1994} that the numeration system $U$ given by $U_n=(n+1)^2$ for all $n\ge 0$ has a non-regular numeration language. This result can be recovered by the characterization given in \cref{cor:beta=1}.
\end{example}

Finally, similarly to the discussion ending \cref{sec:i-infini}, let us explain how the second part of Hollander's main result from \cite{Hollander:1998} can be re-obtained as a consequence of \cref{cor:simple-Parry}.

\begin{corollary}[\cite{Hollander:1998}]
    Let $U$ be a numeration system with a dominant root $\beta>1$ such that $d_\beta(1)$ is finite of length $\ell$, i.e., $\beta$ is a simple Parry number. 

    \smallskip
    \begin{itemize} \setlength\itemsep{0.7em}
        \item If the numeration language $L_U$ is regular, then the base sequence $U$ satisfies a linear recurrence relation whose characteristic polynomial is given by an extended $\beta$-polynomial multiplied by $(X^\ell-1)$.
        \item If the base sequence $U$ satisfies a linear recurrence relation whose characteristic polynomial is given by an extended $\beta$-polynomial, then the numeration language $L_U$ is regular.
    \end{itemize}
\end{corollary}

\begin{proof}
    Suppose that $d_\beta(1)=t_0\ldots t_{\ell-1}0^\omega$ with $t_{\ell-1}>0$. Then $d^*_\beta(1)=(t_0\ldots t_{\ell-2}(t_{\ell-1}-1))^\omega$. The Parry polynomial was introduced at the end of \cref{sec:i-infini} for non-simple Parry numbers. Mimicking the definition by using the digits of the quasi-greedy expansion $d_{\beta}^*(1)=(t_0\ldots t_{\ell-2}(t_{\ell-1}-1))^\omega$, the Parry polynomial is defined as
    \[
        P_{0,\ell}
        :=\left(X^\ell-\sum_{k=1}^{\ell-1} t_{k-1} X^{\ell-k}-(t_{\ell-1}-1) \right)-X^0
        =X^\ell- \sum_{k=1}^\ell t_{k-1}X^{\ell-k}
    \]
    and the extended $\beta$-polynomials as $P_{N,M\ell}=X^N(1+X^\ell+\ldots+X^{(M-1)\ell})P_{0,\ell}$ for $N\ge 0$ and $M\ge 1$. Thus, the first part of \cref{cor:simple-Parry} amounts to saying that if the numeration language $L_U$ is regular, then the base sequence $U$ satisfies a linear recurrence relation whose characteristic polynomial is of the form $X^N(X^M-1)P_{0,\ell}$ for some $N\ge 0$ and $M\ge 1$. Since $X^N(X^M-1)P_{0,\ell}=(X^\ell-1)P_{N,M\ell}$, this proves the first item of the result.

    For the second item, suppose that the base sequence $U$ satisfies a linear recurrence relation whose characteristic polynomial is given by an extended $\beta$-polynomial $P_{N,M\ell}$. This exactly means that the sequence $\Delta$ satisfies a linear recurrence relation whose characteristic polynomial is given by $X^N(1+X^\ell+\cdots+X^{(M-1)\ell})$. This in turn means that
    \[
        \sum_{j=0}^{M-1}\Delta_{n-j\ell}=0
    \]
    for all $n\ge N+M\ell$. Since $X^{M\ell}-1=(X^\ell-1)(1+X^\ell+\cdots+X^{(M-1)\ell})$, we get that $\Delta$ is ultimately periodic with period $M\ell$. Therefore, the second part of \cref{cor:simple-Parry} tells us that the numeration language $L_U$ is regular.
\end{proof}

\section{Vertices leading to a cycle}
\label{sec:i-vers-cycle}

This is the final case in our discussion. We are now focusing on vertices $i$ in the graph $G$ that have a non-trivial path leading to a cycle in $G$.  In this case, there exists $s\ge 1$ and $r\ge 1$ such that $\sigma^{s+r}(i)=\sigma^s(i)$. The expansions $\mathbf{d}_i,\ldots,\mathbf{d}_{\sigma^{s+r-1}(i)}$ are all finite and $\mathbf{d}_i^*=d_i'\cdots d_{\sigma^{s-1}(i)}'(d_{\sigma^{s}(i)}'\cdots d_{\sigma^{s+r-1}(i)}')^\omega$. In the remainder of this section, we assume that $s\ge 1$ is minimal such that there exists $r\geq 1$ with $\sigma^{s+r}(i)=\sigma^s(i)$, to avoid duplicating the previous case. As for \cref{sec:i-vers-infini}, we note that this situation does not occur in the dominant root case of \cite{Hollander:1998}, i.e., whenever $p=1$. We begin with a lemma.

\begin{lemma}
\label{lem:vp2x-2cond}
For any sequence $(x_n)_{n\ge 0}$ of complex numbers and any positive integer $M$, the following conditions are equivalent.

\smallskip
\begin{enumerate}[label=(\alph*)] \setlength\itemsep{0.7em}
    \item \label{item:vp2x-a} The sequence $(x_n)_{n\ge 0}$ is a linear recurrence sequence and all its nonzero eigenvalues are $M$-th roots of unity and have multiplicity at most $2$.
    \item \label{item:vp2x-b} The sequence $(x_{n+M}-x_n)_{n\ge 0}$ is ultimately periodic with period $M$.
\end{enumerate}
\end{lemma}

\begin{proof}
We start by proving that \ref{item:vp2x-a} implies \ref{item:vp2x-b}. We suppose that \ref{item:vp2x-a} holds and we let $\lambda_1,\ldots,\lambda_d$ be the nonzero eigenvalues of $(x_n)_{n\ge 0}$. Then there exist $N\ge 0$ and $a_1,\ldots,a_d$, $b_1,\ldots,b_d\in\C$  such that
\[
    x_n=\sum_{i=1}^d (a_in+b_i)\lambda_i^n
\]
for all $n\ge N$. We get
\[
    x_{n+M}-x_n = \sum_{i=1}^m \left((a_i(n+M)+b_i)\lambda_i^{n+M}
    - (a_in+b_i)\lambda_i^n\right) 
                = \sum_{i=1}^m a_iM\lambda_i^n
\]
and
\[
    x_{n+2M}-x_{n+M} = \sum_{i=1}^d \left((a_i(n+2M)+b_i)\lambda_i^{n+2M}
                    - (a_i(n+M)+b_i)\lambda_i^{n+M}\right) 
                = \sum_{i=1}^d a_iM\lambda_i^n
\]
for all $n\ge N$. Thus, we see that the sequence $(x_{n+M}-x_n)_{n\ge 0}$ is ultimately periodic with period $M$.

We now prove the converse implication. Suppose that there exists $N\ge 0$ such that for all $n\ge N$, we have $x_{n+2M}-x_{n+M}=x_{n+M}-x_n$, or equivalently, $x_{n+2M}=2x_{n+M}-x_n$. Then the sequence $(x_n)_{n\ge 0}$ satisfies the recurrence relation of characteristic polynomial $X^N(X^{2M}-2X^M+1)$, which is $X^N(X^M-1)^2$. Therefore, the eigenvalues of $(x_n)_{n\ge 0}$ are either $0$ or $M$-th roots of unity with multiplicity at most $2$, as expected.
\end{proof}

\begin{proposition}
\label{prop:vers-cycle-vp2x}
    Let $i\in\{0,\ldots,p-1\}$ be such that there exist $s\ge 1$ and $r\ge 1$ with $\sigma^{s+r}=\sigma^s(i)$, with $s$ minimal. If the languages $L_i,L_{\sigma(i)},\ldots,L_{\sigma^{s+r-1}(i)}$ are all regular, then the sequence $\Delta_i$ satisfies the two conditions of \cref{lem:vp2x-2cond} for some $M\ge 1$.
\end{proposition}

\begin{proof}
    Suppose that the languages $L_i,L_{\sigma(i)},\ldots,L_{\sigma^{s+r-1}(i)}$ are all regular. By~\cref{lem:xy*z} combined with a few arithmetic considerations, they can be decomposed in disjoint unions as follows. For every $h\in\{0,\ldots,s+r-1\}$, we have
    \[
        L_{\sigma^h(i)}
        = F_h \cup \bigcup_{e=0}^{M-1} x_{h,e}y_{h,e}^*z_{h,e}
    \]
    where $F_h$ is a finite language, $M\ge 1$ with $|y_{h,e}|=Mp$, $|x_{h,e}z_{h,e}|\equiv-\sigma^h(i)\pmod p$, $|x_{h,0}z_{h,0}|=tMp-\sigma^h(i)$ for some $t$, and $|x_{h,e+1}z_{h,e+1}|=|x_{h,e}z_{h,e}|+p$ for each $e$. Without loss of generality, we can ask that $Mp$ is a multiple of $k_{\sigma^s(i),r}$, which is the sum of the lengths of the finite expansions $\mathbf{d}_{\sigma^(i)}$ for $h\in\{s,\ldots,s+r-1\}$. We let $M'$ be this multiple, so that we have
    \[
        Mp=M'k_{\sigma^s(i),r}=k_{\sigma^s(i),M'r}.
    \]
    From the proof of \cref{prop:cycle-per}, we know that 
    \begin{equation}
    \label{eq:difference-Delta-vers-cycle}
        (\Delta_{\sigma^h(i)})_n=(\Delta_{\sigma^h(i)})_{n-M}
    \end{equation}
    for all large $n$ and all $h\in\{s,\ldots,s+r-1\}$. 

    Let us show that for each $e\in\{0,\ldots,M-1\}$, there exists a constant $\Gamma_{i,e}$ such that
    \[
        (\Delta_i)_n-(\Delta_i)_{n-M}=\Gamma_{i,e}
    \]
    holds for all large $n$ with $n\equiv e\pmod M$. We proceed by induction on $s\ge 1$. The base case $s=1$ and the induction step will be addressed simultaneously. Fix $e\in\{0,\ldots,M-1\}$ and let $n\ge (t+1)M$ such that $n\equiv e\pmod M$. Then $\rep_U(U_{np-i}-1)\in x_{0,e}y_{0,e}^*z_{0,e}$ and similarly as in the proof of \cref{prop:cycle-per}, either $x_{0,e}y_{0,e}^\omega=\mathbf{w}_{i,j}$ for some $j\ge 0$ or $x_{0,e}y_{0,e}^\omega=\mathbf{d}^*_i$, and we consider these two cases separately.

    \medskip
    First, we suppose that $x_{0,e}y_{0,e}^\omega=\mathbf{w}_{i,j}$ for some $j\ge 0$. Then $y_{0,e}=0^{Mp}$. As for \eqref{eq:difference-Delta} in the proof of \cref{thm:i-vers-infini}, we find 
    \[
       (\Delta_i)_n-(\Delta_i)_{n-M}
        =-\sum_{h=1}^j \left( (\Delta_{\sigma^h(i)})_{n-m_{i,h}} - (\Delta_{\sigma^h(i)})_{n-m_{i,h}-M} \right).
    \] 
    Using the induction hypothesis combined with \eqref{eq:difference-Delta-vers-cycle} for $h\in\{s,\ldots,s+r-1\}$, for large enough $n$, each term of the sum in the right-hand side is equal to a constant depending only on the residue class $e$ modulo $M$. (In particular, this constant is $0$ for $h\ge s$.)
    
    \medskip
    Second, we suppose that $x_{0,e}y_{0,e}^\omega=\mathbf{d}^*_i$. In this case, we find
    \[
       (\Delta_i)_n-(\Delta_i)_{n-M}
        =
        -\left(\sum_{h=1}^{s-1} \left( (\Delta_{\sigma^h(i)})_{n-m_{i,h}} - (\Delta_{\sigma^h(i)})_{n-m_{i,h}-M} \right)\right)
        -\left(\Delta_{\sigma^s(i)}^{(M'r)}\right)_{n-m_{i,s}}.
    \] 
    As for the first case, for large enough $n$, each term of the sum in the right-hand side is equal to a constant. Moreover, from the proof of \cref{prop:cycle-per}, we know that the sequences $\Delta_{\sigma^s(i)},\Delta_{\sigma^{s+1}(i)}\ldots,\Delta_{\sigma^{s+r-1}(i)}$ are all ultimately periodic with period $M$. This implies that for all large $n$, the additional term has the constant value $\Delta_{\sigma^s(i),e'}^{(M'r)}$, where $e'=(e-m_{i,s})\bmod M$.
\end{proof}

There are indeed cases where $i$ is on a path leading to a cycle in $G$, is such that the corresponding language $L_i$ is regular and the sequence $\Delta_i$ satisfies the conditions of \cref{lem:vp2x-2cond} without being periodic. This is illustrated by the following example. 

\begin{example}
\label{ex:racines-doubles}
    Consider the sequence $U$ given by $U_0=1$ and the relations
    \[    
        U_n=
        \begin{cases}
            U_{n-1}+U_{n-2}+1,      & \text{if } n\equiv 0\pmod 2; \\
            3U_{n-1}-\frac{n-1}{2}, & \text{if } n\equiv 1\pmod {10}; \\
            3U_{n-1}, & \text{otherwise.}
        \end{cases}
    \]
    This sequence is indeed a linear recurrence sequence, satisfying the linear recurrence relation $U_{n+22}=4U_{n+20}+2U_{n+12}-8U_{n+10}-U_{n+2}+4U_n$ for $n\ge 0$. It is associated with the alternate base $\big(\frac{4}{3},3\big)$. This base yields $\mathbf{d}_0=110^\omega$, $\mathbf{d}_1=30^\omega$, $\mathbf{d}_0^*=(10)^\omega$ and $\mathbf{d}_1^*=2(10)^\omega$. The associated graph $G$ is depicted in \cref{fig:graphG-ex-racines-doubles}. 
    \begin{figure}[htb]
\centering
\begin{tikzpicture}
\tikzstyle{every node}=[shape=circle, fill=none, draw=black,minimum size=15pt, inner sep=2pt]
\node(0) at (0,0) {$0$};
\node(1) at (2,0) {$1$} ;
\tikzstyle{every path}=[color=black, line width=0.5 pt]
\tikzstyle{every node}=[shape=circle]
\draw [-Latex] (0) to [loop above] node [] {} (0);
\draw [-Latex] (1) to [] node [] {} (0) ;
\end{tikzpicture}
\caption{The graph $G$ associated with the alternate base numeration system $\big(\frac{4}{3},3\big)$.}
\label{fig:graphG-ex-racines-doubles}
\end{figure}
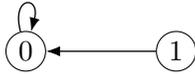 
    One can see by inspection that the language $L_U$ is regular, as we have
    \[  
        \rep_U(U_n-1)=
        \begin{cases}
            110^{n-2},          & \text{if } n\equiv 0\pmod 2; \\
            2(10)^{(n-1)/2},    & \text{if } n\equiv 1\pmod {10}; \\
            2110^{n-3},         & \text{otherwise.}
        \end{cases}
    \]
    The sequence $\Delta_1$ is given by
    \begin{equation}
        \label{eq:racines-doubles}
        (\Delta_1)_n=U_{2n-1}-3U_{2n-2}=
        \begin{cases}
            -n+1,   & \text{if } n\equiv 1\pmod 5; \\
            0,      & \text{otherwise}.
        \end{cases}
    \end{equation}  
    It is a linear recurrence sequence with minimal polynomial $X^{10}-2X^5+1=(X^5-1)^2$. Thus, we see all fifth roots of unity as eigenvalues of this sequence.
    Of course, the choice of $M=5$ was arbitrary here, and this example shows that any root of unity can occur as a double eigenvalue of $\Delta_i$ in this case.
\end{example}

Now that we have reduced our study to linear sequences with eigenvalues that are roots of unity with multiplicity at most $2$, we investigate precisely which of these sequences lead to regular languages. This is the point of our next and final main result. We will see that although eigenvalues of multiplicity $2$ may occur, these occurrences are heavily constrained. In fact, only one new type of behavior arises. 

In order to state this result, we introduce one last definition.

\begin{definition}
\label{def:Gamma-iej}
    Let $i\in\{0,\ldots,p-1\}$ be such that there exist $s\ge 1$ and $r\ge 1$ such that $\sigma^{s+r}(i)=\sigma^s(i)$, with $s$ minimal,
    and such that the sequences $\Delta_i,\Delta_i,\ldots,\Delta_{\sigma^{s-1}}$ all satisfy the condition \ref{item:vp2x-b} of \cref{lem:vp2x-2cond} with a common $M$. For all $j\in\{0,\ldots,s-1\}$ and all $e\in\{0,\ldots,M-1\}$, we let $\Gamma_{i,e,j}$ the ultimate constant value of $(\Delta_{\sigma^j(i)})_{nM+e-m_{i,j}}-(\Delta_{\sigma^j(i)})_{(n-1)M+e-m_{i,j}}$.
\end{definition}

\begin{theorem}
\label{thm:i-vers-cycle}
    Let $i\in\{0,\ldots,p-1\}$ be such that there exist $s\ge 1$ and $r\ge 1$ such that $\sigma^{s+r}(i)=\sigma^s(i)$, with $s$ minimal. In this case the greedy expansions $\mathbf{d}_i,\ldots,\mathbf{d}_{\sigma^{s+r-1}(i)}$ are finite and $\mathbf{d}^*_i=d'_i\cdots d'_{\sigma^{s-1}(i)}(d'_{\sigma^s(i)}\cdots d'_{\sigma^{s+r-1}(i)})^\omega$. We assume that the languages $L_{\sigma(i)},\ldots,L_{\sigma^{s+r-1}(i)}$ are all regular, that the sequences $\Delta_i,\Delta_{\sigma(i)},\ldots,\Delta_{\sigma^{s-1}(i)}$ all satisfy the condition \ref{item:vp2x-b} of \cref{lem:vp2x-2cond} with a common $M$ such that $Mp=M'k_{\sigma^s(i),r}$ with $M'\ge 1$, and that the sequences $\Delta_{\sigma^s(i)},\ldots,\Delta_{\sigma^{s+r-1}(i)}$ are all ultimately periodic with period $M$. 
    
    \smallskip
    The following assertions are equivalent.
    
    \smallskip
    \begin{enumerate}[label=(\alph*)] \setlength\itemsep{0.7em}
        \item \label{item:i-vers-cycle-a} The language $L_i$ is regular.
        \item \label{item:i-vers-cycle-b} For all $c\ge 1$, the language $L_{i,c}$ is regular.
        \item \label{item:i-vers-cycle-c} For all $e\in\{0,\ldots,M-1\}$, either $\Gamma_{i,e,0}=0$, or $\Gamma_{i,e,0}=-\Delta_{\sigma^s(i),e'}^{(M'r)}<0$ where $e'=(e-m_{i,s})\bmod M$ and $\Gamma_{i,e,j}=0$ for all $j\in\{1,\ldots,s-1\}$.
    \end{enumerate}
\end{theorem}

\begin{proof}
    We prove this proposition by induction on $s$. We assume that the equivalences hold for $\sigma(i),\ldots,\sigma^{s-1}(i)$ and show them for $i$. The base case $s=1$ and the induction step will be addressed simultaneously. It is clear that \ref{item:i-vers-cycle-b} implies \ref{item:i-vers-cycle-a}. 

    \medskip
    Let us show that \ref{item:i-vers-cycle-c} implies \ref{item:i-vers-cycle-b}. Consider some fixed $c\ge 1$.
    From~\cref{lem:SplitModp}, it suffices to show that the language
    \[
        L_{i,c,e}=\{\rep_{i,c}(n) : n\ge 1,\ n\equiv e\pmod M,\ U_{np-i}\ge c\}.
    \]
    is regular for each $e\in\{0,\ldots,M-1\}$. We fix such an $e$ and consider the two cases given in \ref{item:i-vers-cycle-c} separately. In what follows, we let $n$ be an integer congruent to $e$ modulo $M$ and since we are interested in ultimate properties only, we always ask $n$ to be sufficiently large so that all the future claims hold.

    \medskip
    First, consider the case where $\Gamma_{i,e,0}=0$. This means that the value $(\Delta_i)_n$ 
    is eventually equal to a constant that we  name $\Delta_{i,e,0}$ by analogy with \cref{def:constant-Delta}.
    If $\Delta_{i,e,0}\ge c$ then we get from \cref{lem:tech} that 
    \[
        \rep_{i,c}(n)\in d_i0^*\rep_U(\Delta_{i,e,0}-c),
    \]
    which implies that $L_{i,c,e}$ is indeed regular. Now, suppose that $\Delta_{i,e,0}<c$. \cref{lem:tech} yields
    \[
        \rep_{i,c}(n)=d'_i\rep_{\sigma(i),c'}(n-m_{i,1}).
    \]
    where $c'=c-\Delta_{i,e,0}$. Letting $e'=(e-m_{i,1})\bmod M$, we get that the languages $L_{i,c,e}$ and $d'_iL_{\sigma(i),c',e'}$ coincide on all long enough words. Since we have assumed $L_{\sigma(i)}$ to be a regular language, the language $L_{\sigma(i),c',e'}$ is regular by induction hypothesis if $s\ge 2$ and by \cref{thm:i-cycle} if $s=1$. This proves that $L_{i,c,e}$ is regular in this case as well.

    Second, suppose that $\Gamma_{i,e,0}=-\Delta_{\sigma^s(i),e'}^{(M'r)}<0$ where $e'=(e-m_{i,s})\bmod M$ and that 
    \begin{equation}
    \label{eq:Gamma=0}
        \Gamma_{i,e,j}=0 \quad \text{ for }j\in\{1,\ldots,s-1\}.    
    \end{equation}
    Using the ultimate periodicity of the sequences $\Delta_{\sigma^s(i)},\ldots,\Delta_{\sigma^{s+r-1}(i)}$ combined with \eqref{eq:Gamma=0}, we have that
    \begin{equation}
        \label{eq:tM}
        \big(\Delta_i^{(j)}\big)_{n+tM} 
        = \big(\Delta_i^{(j)}\big)_n + t\Gamma_{i,e,0}
    \end{equation}
    for all $j,t\ge 0$.
    Since $\Gamma_{i,e,0}<0$, we obtain that $\big(\Delta_i^{(j)}\big)_n<c$ for all $j\in\{1,\ldots,s+M'r\}$ and all large enough $n$.
    As a result, \cref{lem:tech} ensures that
    \[
        \rep_{i,c}(n)=d'_i\cdots d'_{\sigma^{s+M'r-1}(i)} \rep_{\sigma^s(i),c_n}(n-M-m_{i,s})
    \]
    where $c_n=c-\big(\Delta_i^{(s)}\big)_n+\Gamma_{i,e,0}$. We have used that $m_{i,s+M'r}=m_{i,s}+M$ and that
    \[
        \big(\Delta_i^{(s+M'r)}\big)_n=(\Delta_i^{(s)})_n+(\Delta_{\sigma^s(i)}^{(M'r)})_{n-m_{i,s}}=(\Delta_i^{(s)})_n-\Gamma_{i,e,0}.
    \]
    Similarly, we obtain
    \[
        \rep_{i,c}(n-M) = d'_i\cdots d'_{\sigma^{s-1}(i)} \rep_{\sigma^s(i),c_n}(n-M-m_{i,s})
    \]
    where we have used \eqref{eq:tM} in order to get that $c-\big(\Delta_i^{(s)}\big)_{n-M}=c-\big(\Delta_i^{(s)}\big)_n+\Gamma_{i,e,0}=c_n$.
    We see that $\rep_{i,c}(n)$ and $\rep_{i,c}(n-M)$ share the same suffix of length $(n-M)p-i-k_{i,s}$. From there, we conclude that there exists $N\ge 1$ and a suffix $w$ such that for $n$ larger than $N$ and congruent to $e$ modulo $M$, we have
    \[
        \rep_{i,c}(n) \in d'_i\cdots d'_{\sigma^{s-1}(i)} 
            \left( d'_{\sigma^s(i)} \cdots d'_{\sigma^{s+M'r-1}(i)} \right)^* w.
    \]
    Thus $L_{i,c,e}$ is regular, as expected. 

	\medskip
    We now prove that \ref{item:i-vers-cycle-a} implies \ref{item:i-vers-cycle-c}. We will proceed by contraposition and show that if the condition \ref{item:i-vers-cycle-c} is not met, then we can construct $C$ different suffixes of identical length of words in $L_i$ for arbitrary $C$, which contradicts the regularity of this language by \cref{lem:xy*z}.

    We suppose that there exists $e\in\{0,\ldots,M-1\}$ such that $\Gamma_{i,e,0}\ne 0$ and, either $\Gamma_{i,e,0}\ne -\Delta_{\sigma^s(i),e'}^{(M'r)}$ where $e'=(e-m_{i,s})\bmod M$, or there exists $j\in\{1,\ldots,s-1\}$ such that $\Gamma_{i,e,j}\ne 0$. We pick such an $e$ and we fix $C\ge 1$. As above, we consider large $n$ with $n\equiv e\pmod M$.
    
    First, we consider the case where $\Gamma_{i,e,0}>0$. Since $(\Delta_i)_{n+tM}=(\Delta_i)_n+t\Gamma_{i,e,0}$ for all $t\ge 0$, one has $(\Delta_i)_n>0$ if $n$ is large enough. \cref{lem:tech} gives us
    \[
        \rep_{i,1}(n)\in d_i0^*\rep_U((\Delta_i)_n-1).
    \]
    Note that, in the present case, at least one of the bases $\beta_j$ of the alternate base $B=(\beta_0,\ldots,\beta_{p-1})$ is greater than $1$, which implies that $U_{np-j}-U_{np-j-1}$ tends to infinity as $n$ does. Therefore, up to letting $n$ grow, we can build $C$ distinct words 
    \[
        \rep_U((\Delta_i)_n-1),\rep_U((\Delta_i)_{n+M}-1),\ldots,\rep_U((\Delta_i)_{n+(C-1)M}-1)
    \]
    of identical length which are suffixes of words in $L_i$.
	
	Now, we suppose that $\Gamma_{i,e,0}<0$. Using the ultimate periodicity of the $r$ sequences $\Delta_{\sigma^s(i)},\ldots,\Delta_{\sigma^{s+r-1}(i)}$, we have 
    \begin{equation}
        \label{eq:Delta_n+tM}
        \big(\Delta_i^{(j)}\big)_{n+tM} 
        =  \big(\Delta_i^{(j)}\big)_n + t\sum_{h=0}^{s-1}\Gamma_{i,e,h}
    \end{equation}
    for all $j\ge s$ and $t\ge 0$. By induction hypothesis, since the languages $L_{\sigma(i)},\ldots,L_{\sigma^{s-1}(i)}$ are assumed to be regular, the condition \ref{item:i-vers-cycle-c} is met for $\sigma(i),\ldots,\sigma^{s-1}(i)$. Thus, for each $h\in\{1,\ldots,s-1\}$, since $\Gamma_{i,e,h}=\Gamma_{\sigma^h(i),f,0}$ with $f=(e - m_{i,h})\bmod M$, we obtain that either $\Gamma_{i,e,h}=0$, or $\Gamma_{i,e,h}=-\Delta^{(M'r)}_{\sigma^s(i),e'}< 0$ with $e'=(e-m_{i,s})\bmod M$ and $\Gamma_{i,e,h'}=0$ for every $h'\in\{h+1,\ldots,s-1\}$, where we have used that $m_{\sigma^h(i),s-h}\equiv (m_{i,s}-m_{i,h})\pmod M$. In particular, we have $\Gamma_{i,e,h}\le 0$ for every $h\in\{1,\ldots,s-1\}$. Therefore we have 
    \[
        \big(\Delta_i^{(j)}\big)_{n+tM}
        \le \big(\Delta_i^{(j)}\big)_n+t\Gamma_{i,e,0}
    \]
    for all $j\ge s$ and $t\ge 0$. Since $\Gamma_{i,e,0}<0$, we obtain that $\big(\Delta_i^{(j)}\big)_n\le 0$ for all $j\ge s$ and all large enough $n$ (depending on $j$). \cref{lem:tech} then gives us that for all $c\in\{0,\ldots,C-1\}$, we have
	\[
	    \rep_{i,1}(n+cM)  
        = d'_i\cdots d'_{\sigma^{s-1}(i)} 
        \big( d'_{\sigma^s(i)}\cdots d'_{\sigma^{s+M'r-1}(i)}\big)^c 
        \rep_{i,c'}(n-m_{i,s})
	\]
    where $c'=1-\big(\Delta_i^{(s+cM'r)}\big)_{n+cM}$. Yet, using \eqref{eq:Delta_n+tM} and \eqref{eq:sum-Delta-cM'r}, we get
    \begin{align*}
        \big(\Delta_i^{(s+cM'r)}\big)_{n+cM}
        &=\big(\Delta_i^{(s)}\big)_{n+cM}
        +\big(\Delta_{\sigma^s(i)}^{(cM'r)}\big)_{n+cM-m_{i,s}} \\
        &=\big(\Delta_i^{(s)}\big)_n+c\sum_{j=0}^{s-1}\Gamma_{i,e,j}+c\Delta^{(M'r)}_{\sigma^s(i),e'}
    \end{align*}
    where $e'=(e-m_{i,s})\bmod M$. Let us argue that the constant 
    \[
        Q:=\sum_{j=0}^{s-1}\Gamma_{i,e,j}+\Delta^{(M'r)}_{\sigma^s(i),e'}
    \]
    is different from $0$. If $\Gamma_{i,e,j}=0$ for all $j\in\{1,\ldots,s-1\}$, then by choice of $e$, we must have $\Gamma_{i,e,0}\ne -\Delta^{(M'r)}_{\sigma^s(i),e'}$, hence $Q\ne 0$. Now, suppose that $\Gamma_{i,e,j}\ne 0$ for some $j\in\{1,\ldots,s-1\}$. Then, by the induction hypothesis, as explained at the beginning of this case, we must have $\Gamma_{i,e,j}=-\Delta^{(M'r)}_{\sigma^s(i),e'}<0$ and $\Gamma_{i,e,j'}=0$ for all $j'\in\{1,\ldots,s-1\}$ that are distinct from $j$. Since $\Gamma_{i,e,0}<0$, we also get $Q\ne 0$. Thus, up to letting $n$ grow, once again, we see that we can build $C$ distinct words 
    \[
         \rep_{i,c'}(n-m_{i,s}),\ \text{ for }c'\in\{1-\big(\Delta_i^{(s)}\big)_n-cQ : 0\le c<C\},
    \]
    of identical length which are suffixes of words in $L_i$.
\end{proof}

\begin{remark} 
\label{rem:i-vers-cycle-effective}
Let us argue that the condition \ref{item:i-vers-cycle-c} of \cref{thm:i-vers-cycle} can be used effectively. Once the eigenvalues of $U$ are known, we can find the eigenvalues of the sequences $\Delta_{\sigma^j(i)}$, for $j\in\{0,\ldots,s-1\}$, and check whether these sequences all satisfy the condition \ref{item:vp2x-a} of \cref{lem:vp2x-2cond}. If this is the case, and assuming that we already know that the sequences $\Delta_{\sigma^j(i)}$ with $j\in\{s,\ldots,s+r-1\}$ are ultimately periodic, a common $M$ (as in the statement of \cref{thm:i-vers-cycle}) can be computed. From there, the values $\Gamma_{i,e,j}$ and $\Delta_{\sigma^s(i),e'}^{(M'r)}$ can be computed, and \cref{thm:i-vers-cycle} can be used to decide the regularity of $L_i$.
\end{remark}

\begin{example}
     We resume \cref{ex:racines-doubles} to illustrate the notation of \cref{thm:i-vers-cycle}. We have $p=2$, $i=1$, $s=r=1$, $k_{0,1}=2$, and $M=M'=5$. We also have $(\Delta_0)_n=1$ for all $n$ while the values $(\Delta_1)_n$ were given in \eqref{eq:racines-doubles}  We get $\Gamma_{1,e,0}=0$ for $e=0,2,3,4$. Now, consider the case where $e=1$. For any $n$, we have $\Gamma_{1,1,0}=(\Delta_1)_{5n+1}-(\Delta_1)_{5n-4}=-5n-(-5n+5)=-5$. Finally, we get $m_{1,1}=1$, $e'=(e-m_{1,1})\bmod 5=0$ and $\Delta^{(5)}_{0,0}=\sum_{h=0}^4 \Delta_{0,0,h}=\sum_{h=0}^4 (\Delta_0)_{5n-h}=5$, in accordance with \cref{thm:i-vers-cycle}.
\end{example}

\section{Decision procedure}
\label{sec:decision}

The results of this work can be used in practice to decide if the language $L_U$ associated with any given positional numeration system $U$ is regular. In this section, we describe the semi-decision procedure induced by our results, namely \cref{prop:regular->Parry,thm:i-infini,thm:i-vers-infini,prop:cycle-per,thm:i-cycle,prop:vers-cycle-vp2x,thm:i-vers-cycle}. Below, we will argue that in the situation where the obtained alternate base is Parry, we indeed get a decision procedure.

\begin{enumerate}[label=(\arabic*)]
    \item Using \cref{prop:reg->linear}, identify the values of $p$ and $\beta_0,\ldots,\beta_{p-1}$.
    \item Compute the infinite words $\mathbf{d}_0,\ldots,\mathbf{d}_{p-1}$ and the graph $G$. By \cref{prop:regular->Parry}, if one of the infinite words $\mathbf{d}_0,\ldots,\mathbf{d}_{p-1}$ is aperiodic, then $L_U$ is not regular. In this case, the algorithm stops with "no".
    \item \label{item:decision-3} Otherwise, study the regularity of the languages $L_0,\ldots,L_{p-1}$.
    \begin{enumerate}
        \item For each vertex $i$ with no outgoing edge, use the item \ref{item:i-infini-c} of \cref{thm:i-infini}
        to check whether the corresponding languages $L_i$ is regular or not. As soon as one can find a non-regular such language, the algorithm stops with "no". 
        
        \item Proceed with the vertices $i$ leading to a vertex with no outgoing edge, starting with the vertices at distance $1$, then distance $2$, etc. For each such vertex $i$, use the item \ref{item:i-vers-infini-c} of \cref{thm:i-vers-infini} to decide whether the corresponding language $L_i$ is regular. As soon as one can find a non-regular such language, the algorithm stops with "no". 
        
        \item For each cycle in $G$, first check whether every sequence $\Delta_i$ with $i$ in the cycle is ultimately periodic. If at least one of them is not, then the algorithm stops with "no", following \cref{prop:cycle-per}. Otherwise, find a common period $M$ of these sequences such that $Mp$ is a multiple of $k_{i,r}$ where $r$ is the length of the cycle and $i$ is any vertex of the cycle. Then use the item \ref{item:i-cycle-d} of \cref{thm:i-cycle} to decide whether all languages $L_i$, with $i$ in the cycle, are regular. If at least one of them is not regular, the algorithm stops with "no". 
        
        \item Finally, consider the vertices $i$ leading to a cycle, starting with the vertices at distance $1$, then distance $2$, etc. For each such vertex $i$, use \cref{prop:vers-cycle-vp2x} and the item \ref{item:i-vers-cycle-c} of \cref{thm:i-vers-cycle} to decide whether the corresponding language $L_i$ is regular.
        As soon as one can find a non-regular such language, the algorithm stops with "no". 
    \end{enumerate}
    \item If the languages $L_i$ for all $i$ in the graph have been checked to be regular, then $L_U$ itself is regular. The algorithm stops with "yes". 
\end{enumerate}

Note that, at step \ref{item:decision-3}, we can run the tests in parallel, for each path to either a vertex with no outgoing edge or a cycle. First, we can consider all the vertices with no outgoing edge and the cycles, and then proceed iteratively with the vertices at distance $d\ge 1$ to such vertices or cycles, for increasing values of $d$ until we have considered all vertices of the graph.
    
The greedy algorithm provides a semi-algorithm for testing if a given expansion of $1$ is ultimately periodic: we may simply generate digits $t_i$ and memorize the remainders $r_i$ that appear in the algorithm (to recall, $t_i=\lfloor \beta_i r_i\rfloor$ and $r_{i+1}=\beta_i r_i -t_i$). If two remainders are equal, then the algorithm loops and the expansion of $1$ is ultimately periodic. In particular, we have a semi-algorithm that tests whether a given alternate base is Parry. However, testing when an alternate base is not Parry is harder. Techniques such as those of \cref{sec:MemePolGraphesDiffs} can sometimes be applied, but there is no known general algorithm.

If we assume that the eigenvalues of $U$ are known and that the associated alternate base is known to be Parry, \cref{prop:i-infini-effective,rem:i-vers-infini-effective,rem:i-cycle-effective,rem:i-vers-cycle-effective} ensure that the above procedure can be carried out effectively.

We now illustrate the above decision procedure with two examples, chosen to present a variety of behaviors.

\begin{example}
\label{ex:ProcedureBoucle}
Let us consider the system $U$ generated  by the recurrence relation
\[ 
    U_{n+18} = 23U_{n+15} + 5U_{n+12} -46 U_{n+9} -7U_{n+6} + 23 U_{n+3} +3 U_n
\]
and the initial conditions 
\begin{align*}
    (U_0,\ldots, U_{17}) =
        & (1,\, 4,\, 9,\, 20,\, 70,\, 175,\, 489,\, 1641,\, 4015,\, 11294,\, 37898,\, 92748,\, 261291,\, \\
        & \ 876620,\, 2145176,\, 6043562,\, 20275863,\, 49617086).
\end{align*}
% \todoS[inline]{Built from the relations 
% \begin{align*}
% U_{6n}  &= 2U_{6n-1}+2U_{6n-2}-2n, \\
% U_{6n-1}&= 2U_{6n-2}+U_{6n-3}+U_{6n-4}+U_{6n-5}+2, \\
% U_{6n-2}&= 3U_{6n-3}+U_{6n-4}+1, \\
% U_{6n-3}&= 2U_{6n-4}+2U_{6n-5}-1, \\
% U_{6n-4}&= 2U_{6n-5}+U_{6n-6}+U_{6n-7}+U_{6n-8}-1, \\
% U_{6n-5}&= 3U_{6n-6}+U_{6n-7}
% \end{align*}
% and the initial conditions $(1,4,9)$.}
The minimal polynomial of $U$ is $P(X^3)$ where
\begin{align*}
    P(X)   &=X^6-23X^5-5X^4+46X^3+7X^2-23X-3 \\
        &=(X+1)^2(X-1)^2\left(X-\frac{23-\sqrt{541}}{2}\right)\left(X-\frac{23+\sqrt{541}}{2}\right).
\end{align*}  
As in the proof of \cref{prop:U->beta}, we find $p=3$ and we obtain closed formulas for $U_{3n}, U_{3n-1}$ and $U_{3n-2}$ from which we find
\[
    \beta_0=\frac{19+\sqrt{541}}{15},\quad 
    \beta_1=\frac{11+\sqrt{541}}{14} 
    \quad \text{ and } \quad
    \beta_2=\frac{17+\sqrt{541}}{12}.
\]
Using the semi-algorithm described above, we find that the alternate base $B$ is Parry as we obtain 
\[
    \mathbf{d}_0=220^\omega,\quad 
    \mathbf{d}_1=21110^\omega
    \quad \text{ and } \quad 
    \mathbf{d}_2=310^\omega.
\]
The graph $G$ is depicted on \cref{fig:graphGProcBoucle}.
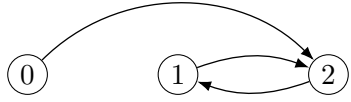
\begin{figure}[htb]
\centering
\begin{tikzpicture}
\tikzstyle{every node}=[shape=circle, fill=none, draw=black,minimum size=15pt, inner sep=2pt]
\node(0) at (0,0) {$0$};
\node(1) at (2,0) {$1$} ;
\node(2) at (4,0) {$2$} ;
\tikzstyle{every path}=[color=black, line width=0.5 pt]
\tikzstyle{every node}=[shape=circle]
\draw [-Latex] (0) to [bend left=45] node [] {} (2) ;
\draw [-Latex] (1) to [bend left=20] node [] {} (2) ;
\draw [-Latex] (2) to [bend left=20] node [] {} (1) ;
\end{tikzpicture}
\caption{The graph $G$ associated with the alternate base $B$ of \cref{ex:ProcedureBoucle}.}
\label{fig:graphGProcBoucle}
\end{figure}

Now we may start studying the regularity of the languages $L_i$ individually. We start with $L_1$ and $L_2$. Using \cref{def:Delta_i}, for $n\ge 2$, we find that
\[
    (\Delta_1)_n=
    \begin{cases}
        2,     & \text{if } n\text{ is even}; \\
        -1,    & \text{if } n\text{ is odd}
    \end{cases}
    \qquad\text{ and }\qquad
    (\Delta_2)_n=
    \begin{cases}
        1,     & \text{if } n\text{ is even}; \\
        -1,    & \text{if } n\text{ is odd}.
    \end{cases}
\]
 This can be proved by noting that $\Delta_1$ and $\Delta_2$ ultimately satisfy the recurrence relation given by $P$. We may thus proceed past \cref{prop:cycle-per} and verify the criterion \ref{item:i-cycle-d} in \cref{thm:i-cycle} for $i=1$. Note that this criterion only needs to be checked for one vertex in the cycle. Here we have $M=2$, $r=2$, $k_{1,r}=6$, $M'=1$, $m_{1,0}=0$ and $m_{1,1}=1$. Using \cref{def:constant-Delta}, we can compute the following values, where $n$ is taken large enough to enter the ultimate constant part:
\begin{align*}
    \Delta_{1,0}^{(2)}&=\Delta_{1,0,0}+\Delta_{1,0,1}= (\Delta_1)_{2n+0-0}+(\Delta_2)_{2n+0-1}=2-1=1\\
    \Delta_{1,1}^{(2)}&=\Delta_{1,1,0}+\Delta_{1,1,1}=(\Delta_1)_{2n+1-0}+(\Delta_2)_{2n+1-1}=-1+1=0.
\end{align*}
Since $\Delta_{1,e}^{(2)}\ge 0$ for $e\in\{0,1\}$, the languages $L_1$ and $L_2$ are both regular. 

It remains to check the regularity of the language $L_0$. Here we find that for $n\ge 1$, we have 
\[
    (\Delta_0)_n=
    \begin{cases}
        -1,     & \text{if } n\text{ is even}; \\
        -6n,    & \text{if } n\text{ is odd}.
    \end{cases}
\]
 As in the previous paragraph, this can be proved by noting that $\Delta_0$ ultimately satisfies the recurrence relation given by $P$. So $(\Delta_0)_{n+2}-(\Delta_0)_n$ is ultimately periodic with period $2$, and we may proceed past \cref{prop:vers-cycle-vp2x} and verify the criterion \ref{item:i-vers-cycle-c} in \cref{thm:i-vers-cycle}. Here $M=2$ satisfies the assumptions of the statement. Using \cref{def:Gamma-iej} and choosing large enough $n$ to be in the constant part, we have 
 \begin{align*}
     \Gamma_{0,0,0} &=(\Delta_0)_{2n}-(\Delta_0)_{2n-2}=0\\
     \Gamma_{0,1,0} &=(\Delta_0)_{2n+1}-(\Delta_0)_{2n-1}=-12.
 \end{align*} 
 Since $\Gamma_{0,1,0}$ is not zero, it must be equal to $-\Delta_{2,1}^{(2)}$. But this fails to be the case as $m_{2,0}=0$, $m_{2,1}=1$ and
 \[
    \Delta_{2,1}^{(2)}=\Delta_{2,1,0}+\Delta_{2,1,1}=(\Delta_2)_{2n+1-0}+(\Delta_1)_{2n+1-1}=-1+2=1.
 \]
 Therefore, the language $L_0$ is not regular, hence the numeration language $L_U$ is not regular either. 
\end{example}

\begin{example}
\label{ex:ProcedureInfini}
Consider the system $U$ generated by the linear recurrence relation
\begin{align*}
    U_{n+13}
    &=-U_{n+12}-U_{n+11}+22U_{n+10}+22U_{n+9}+22U_{n+8}+13U_{n+7} \\
    &\qquad +13U_{n+6}+13U_{n+5}-10U_{n+4}-10U_{n+3}-10U_{n+2}
\end{align*}
and the initial conditions
\[
(U_0,\ldots,U_{12}) = (1,\,4,\,8,\,22,\,71,\,185,\,476,\,1614,\,4179,\,10740,\,36396,\,94271,\,242238).
\]
The minimal polynomial of $U$ factors as $X^2 (X+1) (X^2-X+1) (X^2+X+1) (X^6 - 23 X^3 + 10)$. Multiplying this polynomial by $X-1$, we see  $U$ also satisfies the linear recurrence relation of characteristic polynomial $X^2P(X^3)$, where $P=(X^2-1)(X^2-23X+10)$. 
% \todoS[inline]{Built from the relations 
% \begin{align*}
% U_{6n}  &= 2U_{6n-1}+U_{6n-2}+2U_{6n-3}-U_{6n-4}-1, \\
% U_{6n-1}&= 2U_{6n-2}+2U_{6n-3}-1, \\
% U_{6n-2}&= 3U_{6n-3}+U_{6n-4}-3, \\
% U_{6n-3}&= 2U_{6n-1}+U_{6n-2}+2U_{6n-3}-U_{6n-4}+1, \\
% U_{6n-4}&= 2U_{6n-5}+2U_{6n-6}-1, \\
% U_{6n-5}&= 3U_{6n-6}+U_{6n-7}+1
% \end{align*}
% and the initial conditions $(1,4,8,22)$.
% }
In this example, we find $p=3$ and we obtain closed formulas for $U_{3n},U_{3n-1}$ and $U_{3n-2}$, from which the associated alternate base $B=(\beta_0,\beta_1,\beta_2)$ can be computed. We find 
\[
    \beta_0=\frac{19+\sqrt{489}}{16},\quad
    \beta_1=\frac{53+\sqrt{489}}{29}
    \quad \text{ and } \quad
    \beta_2=\frac{5+\sqrt{489}}{8}.
\]
This alternate base is Parry since the corresponding greedy expansions of $1$ are 
\[
    \mathbf{d}_0=21^\omega,\quad 
    \mathbf{d}_1=220^\omega
    \quad \text{ and } \quad
    \mathbf{d}_2=310^\omega. 
\]
The associated graph is depicted in \cref{fig:graphGProcInfini}. 
\begin{figure}[htb]
\centering
\begin{tikzpicture}
\tikzstyle{every node}=[shape=circle, fill=none, draw=black,minimum size=15pt, inner sep=2pt]
\node(0) at (0,0) {$0$};
\node(1) at (2,0) {$1$} ;
\node(2) at (4,0) {$2$} ;
\tikzstyle{every path}=[color=black, line width=0.5 pt]
\tikzstyle{every node}=[shape=circle]
\draw [-Latex] (1) to [] node [] {} (0) ;
\draw [-Latex] (2) to [] node [] {} (1) ;
\end{tikzpicture}
\caption{The graph $G$ associated with the alternate base of \cref{ex:ProcedureInfini}.}
\label{fig:graphGProcInfini}
\end{figure}
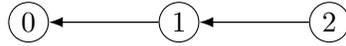

We must first decide the regularity of $L_0$, then $L_1$, then $L_2$. 
For $L_0$, the expansion $\mathbf{d}_0$ has minimal preperiod $q_0=1$ and the minimal period that is a multiple of $p$ is $m_0=3$. Using \cref{def:Delta_iqm}, we obtain 
\[
    (\Delta_{0,1,3})_n=
    \begin{cases}
        -1,     & \text{if } n\text{ is even}; \\
        1,    & \text{if } n\text{ is odd}
    \end{cases}
\]
for all $n\ge 2$. This can be seen from the initial values of this sequence and the fact that it ultimately follows the linear recurrence relation given by $P$. \cref{lem:Delta-q} then gives us $(\Delta_{0,1,6})_n=(\Delta_{0,1,3})_n+(\Delta_{0,1,3})_{n-1}=0$. The language $L_0$ is then regular by using the criterion \ref{item:i-infini-c} of \cref{thm:i-infini} with $k=2$. 

For $L_1$ and subsequently $L_2$, using \cref{def:Delta_i}, we have   
\[
    (\Delta_1)_n=-1
    \qquad\text{ and }\qquad
    (\Delta_2)_n=
    \begin{cases}
        -3,     & \text{if } n\text{ is even}; \\
        -1,    & \text{if } n\text{ is odd}.
    \end{cases}
\]
for all $n\ge 2$. Therefore, we may use the criterion \ref{item:i-vers-infini-c} of \cref{thm:i-vers-infini} to deduce that $L_1$, then $L_2$ are regular. In the end, the numeration language $L_U$ is regular.
\end{example}

\section{Comments on Hollander's original conjecture}
\label{sec:comments-Hollander}

In this section, we go back to Hollander's original conjecture, which was the starting point for this research. We restate it with our notation.
\begin{conjecture}[\cite{Hollander:1998}, Section 8.2]
    If $L$ is regular, there exists $p$ such that the limit
    \[
    \lim_{n\to\infty} \frac{U_{np+i}}{U_{(n-1)p+i}} 
    \]
    exists and is independent of $i$. Furthermore, the minimal polynomial $P$ of the recurrence relation satisfied by $U$ is of the form $P(X)=Q(X^p)$ where $Q$ is the minimal polynomial for a recurrence which gives rise to a regular language.
\end{conjecture}
We should note that in \cite{Hollander:1998} this conjecture is given in the form of a comment rather than a formal statement. Therefore, some parts of the statement are difficult to interpret unambiguously. In particular, it is not clear in Hollander's statement whether $P$ must be the \emph{minimal} polynomial of $U$, rather than any polynomial associated with a recurrence relation satisfied by $U$.

The first part of the statement was proven in \cref{prop:U->beta}. For the second part of the statement, we give two examples that refute the conjecture as written above and provide a reason to disbelieve the interpretation where $P$ need not be minimal.

\begin{example}
Consider again item \ref{item:liste-exemples-5} of \cref{ex:liste-exemples}, i.e., the system $U$ generated by the recurrence relation $U_{n+3}=2U_{n+2}-4U_{n+1}+8U_{n}$ and the initial conditions $(U_0,U_1,U_2)=(1,3,8)$. 
It can be seen that $L_U$ is regular in this case, and that $\lim_{n\to\infty}\frac{U_{n}}{U_{n-4}}=16$. Nevertheless, the minimal polynomial of $U$ is $X^3-2X^2+4X-8$, which is not of the form $Q(X^4)$.
\end{example}

The next example illustrates a more fundamental objection to the conjecture. There exists an alternate base which is Parry but whose product of elements is not Parry (in the non-alternate sense). For such a base, even if its minimal polynomial is of the form $Q(X^p)$, the polynomial $Q$ cannot be the minimal polynomial of a sequence giving rise to a regular language.

\begin{example}
    Consider the system $U$ generated by the recurrence relation $U_{n+6}=9U_{n+3}-9U_n$ and the initial conditions $(U_0,U_1,U_2,U_3,U_4,U_5)=(1,2,3,9,15,24)$. In this case, the numeration language $L_U$ is again regular. However, the associated alternate base is $(3,\varphi,\varphi)$ where $\varphi$ is the golden ratio, and $3\varphi^2$ is not a Parry number. Therefore, $X^2-9X+9$ is not the minimal polynomial of a recurrence which gives rise to a regular language, independently of the choice of initial conditions. Note that taking multiples of $p$ cannot fix this as no power of $3\varphi^2$ is a Parry number. Indeed, no power of $3\varphi^2$ is a Pisot number and it is known that Parry and Pisot numbers of degree $2$ coincide \cite{Bassino:2002}. 
\end{example}

The latter example does not completely refute Hollander's conjecture in the setting where $P$ is not assumed to be minimal, but for the conjecture to hold there would need to exist a Parry number such that $X^2-9X+9$ is a factor of its Parry polynomial (and similarly for any Parry alternate base whose product of elements is not a Parry number).

\section{Why obtaining criteria for regularity relying only on recurrence relations satisfied by $U$ is not achievable}
\label{sec:MemePolGraphesDiffs}

The results in Hollander~\cite{Hollander:1998} are of a different nature compared to ours. Rather than extracting new sequences from $U$ and deciding the regularity of $L_U$ based on them, Hollander aims to link directly the regularity of $L_U$ with polynomials giving recurrence relations satisfied by $U$, which is a more tractable criterion. The aim of this section is to present evidence that we cannot replicate this in our setting.

Indeed, in Hollander's case with $p=1$, the knowledge of the minimal polynomial $P$ of $U$ gives us the value of the dominant root $\beta$, from where Hollander's study can take place. In our case however, knowing the minimal polynomial of $U$ only tells us the value of the product of all bases, $\delta=\beta_0\cdots\beta_{p-1}$, but does not inform us on the values of $\beta_0,\ldots,\beta_{p-1}$ themselves. As a result, the behavior of $U$ rarely depends on just the minimal polynomial. We illustrate this by exhibiting a polynomial $P$ and various sets of initial conditions that lead to differing behaviors for $U$.

\begin{example}
\label{ex:MemePolGraphesDiffs}
Consider the polynomial $P=X^8-2X^6-2X^4-2$ and the associated recurrence relation $U_{n+8}=2U_{n+6}+2U_{n+4}+2U_n$. The polynomial $P$ has two roots of maximal modulus, which are $-\sqrt{\delta}$ and $\sqrt{\delta}$ where $\delta\simeq 2.80$ is the dominant root of the polynomial $X^4-2X^3-2X^2-2$. Consequently, a generic increasing sequence that satisfies this recurrence relation is associated with an alternate base of length $2$.

First, consider the initial conditions $(U_0,\ldots,U_7)= (1,2,4,6,12,17,34,47)$. We can obtain closed  formulas for $U_{2n}$ and $U_{2n-1}$. From these, we find the alternate base associated with $U$, which is $(\beta_0,\beta_1) = (2,\delta/2)$. We find $\mathbf{d}_0=20^\omega$ and $\mathbf{d}_1=10100010^\omega$. In this case, the alternate base is Parry and the associated graph is a cycle of length $2$. We find that the numeration language is regular in this case. In fact, the maximal words are  exactly the prefixes of the quasi-greedy expansions $\mathbf{d}_0^*=(11010000)^\omega$ and $\mathbf{d}_1^*=(10100001)^\omega$, depending on the parity of their length.

Now, consider the initial conditions $(U_0,\ldots,U_7)= (1,2,3,5,8,13,21,34)$. Similarly, we may obtain values for $\beta_0$ and $\beta_1$, which are $\beta_0=\frac{81}{755}+\frac{371}{755}\delta-\frac{28}{755}\delta^2+\frac{18}{755}\delta^3$ and $\beta_1=\frac{\delta}{\beta_0}=\frac{226}{119}+\frac{59}{119}\delta+\frac{45}{119}\delta^2-\frac{25}{119}\delta^3$.
We prove that this base is not Parry, as both greedy expansions of $1$ are infinite and aperiodic. We prove this by showing that the greedy algorithm started on $1$ does not reach the same remainder twice, using an idea of Schmidt \cite{Schmidt:1980,Charlier&Cisternino&Kreczman:2024}. We know that $1,\delta,\delta^2$ and $\delta^3$ form a base of $\Q(\delta)$ as a $\Q$-vector space. Considering components in  this base, multiplication by any element $\gamma$ in $\Q(\delta)$ can be represented by a matrix $M_\gamma$ in $\Q^{4\times 4}$. 
In particular, we have
\[
    M_{\delta}
    =\begin{pmatrix}
        0 & 0 & 0 & 2\\
        1 & 0 & 0 & 0\\
        0 & 1 & 0 & 2\\
        0 & 0 & 1 & 2
    \end{pmatrix}.
\]

We now consider these matrices $M_\gamma$ as elements of $\C^{4\times 4}$, so that we can diagonalize them. Let $\delta_1=\delta,\delta_2,\delta_3,\delta_4$ be the Galois conjugates of $\delta$, with $|\delta_2|>1$. It is easily seen that for each $k\in\{1,2,3,4\}$, the matrix $M_\delta$ admits the eigenvector $v_k=(-2\delta_k - 2\delta_k^2 + \delta_k^3, -2 -2\delta_k + \delta_k^2, -2+\delta_k, 1)^T$ with eigenvalue $\delta_k$. For any $\gamma\in\Q(\delta)$, if we decompose $\gamma=a+b\delta+c\delta^2+d\delta^3$, then $M_\gamma=aI+bM_\delta+cM_\delta^2+dM_\delta^3$. Therefore, the same vectors $v_k$ are eigenvectors of $M_\gamma$, with corresponding eigenvalues $a+b\delta_k+c\delta_k^2+d\delta_k^3$. Therefore, all matrices $M_\gamma$ are simultaneously diagonalizable by the matrix $S=\begin{pmatrix} v_1 & v_2 & v_3 & v_4\end{pmatrix}$, hence in particular the matrices $M_{\beta_0}$ and $M_{\beta_1}$.

If we now express the remainders in the greedy algorithm when applied to $1$ in the base of the $\C$-vector space $\C^4$ corresponding to the eigenvectors found above, one component corresponds to a Galois conjugate $\delta_2$ of $\delta$ which is approximately $-1.13$. In this component, multiplications by $\beta_0$ in the greedy algorithm correspond to a multiplication by $\frac{81}{755}+\frac{371}{755}\delta_2-\frac{28}{755}\delta_2^2+\frac{18}{755}\delta_2^3$, which is approximately $-0.53$,  multiplications by $\beta_1$ in the greedy algorithm correspond to a multiplication by $\frac{226}{119}+\frac{59}{119}\delta_2+\frac{45}{119}\delta_2^2-\frac{25}{119}\delta_2^3$, which is approximately $2.12$, and the subtraction of $1$ that is sometimes performed between two such multiplications corresponds to adding approximately $0.11$. From this, it can be seen that if an absolute value of $3$ or more is reached on this component when performing the greedy algorithm, then the value of this component tends to infinity as $n$ does. We can numerically verify that the value of this component after $50$ steps of the greedy algorithm for $\mathbf{d}_0$ is about $5.32$ and the one for $\mathbf{d}_1$ is about $-4.63$. As a result, the process as seen in $\C^4$ and in the base of eigenvectors never reaches a periodic point, and neither do the process as seen in the canonical base or the original greedy algorithm. Therefore, both expansions are aperiodic.

Other initial conditions of note are listed in \cref{tab:conditions-initiales}. We have not been able to characterize for which initial conditions the associated alternate base is Parry, which is necessary for regularity.
\begin{table}
    \centering
    \begin{tabular}{c|l}
     Initial conditions & Behavior \\
     \hline
     $(1, 2, 3, 5, 9, 15, 25, 40)$ & Both expansions are ultimately periodic.\\
     $(1, 2, 3, 5, 8, 13, 21, 36)$ & One expansion is ultimately periodic, the other is aperiodic.\\
     $(1, 2, 3, 5, 8, 13, 21, 39)$ & One expansion is finite, the other is ultimately periodic.
\end{tabular}
\medskip

    \caption{Choices of initial conditions giving rise to differing behaviors in \cref{ex:MemePolGraphesDiffs}.}
    \label{tab:conditions-initiales}
\end{table}
\end{example}

\section*{Acknowledgments}

We thank Célia Cisternino for many valuable discussions in the early days of this project.
The second author is funded by the Fonds National de la Recherche Scientifique (ASP grant number 1.A.789.23F).

\bibliographystyle{abbrv}
\bibliography{Biblio-CK25.bib}

@incollection {Bassino:2002,
    AUTHOR = {Bassino, Fr\'{e}d\'{e}rique},
     TITLE = {Beta-expansions for cubic {P}isot numbers},
 BOOKTITLE = {L{ATIN} 2002: {T}heoretical informatics ({C}ancun)},
    SERIES = {Lecture Notes in Comput. Sci.},
    VOLUME = {2286},
     PAGES = {141--152},
 PUBLISHER = {Springer},
   ADDRESS = {Berlin},
      YEAR = {2002},
       DOI = {10.1007/3-540-45995-2\_17}}

@article {Berend&Kumar:2022,
    AUTHOR = {Berend, Daniel and Kumar, Rishi},
     TITLE = {Consecutive ratios in second-order linear recurrence
              sequences},
   JOURNAL = {Unif. Distrib. Theory},
  FJOURNAL = {Uniform Distribution Theory},
    VOLUME = {17},
      YEAR = {2022},
    NUMBER = {2},
     PAGES = {51--76},
       DOI = {10.2478/udt-2022-0012}}

@article {Berend&Kumar:2025,
    AUTHOR = {Berend, Daniel and Kumar, Rishi},
     TITLE = {Kepler sets of linear recurrence sequences},
   JOURNAL = {Acta Math. Hungar.},
  FJOURNAL = {Acta Mathematica Hungarica},
    VOLUME = {175},
      YEAR = {2025},
    NUMBER = {1},
     PAGES = {54--95},
       DOI = {10.1007/s10474-025-01506-6}}

@article {Berstel:1971,
    AUTHOR = {Berstel, Jean},
     TITLE = {Sur les p\^{o}les et le quotient de {H}adamard de s\'{e}ries
              {N}-rationnelles},
   JOURNAL = {C. R. Acad. Sci. Paris S\'{e}r. A-B},
  FJOURNAL = {Comptes Rendus Hebdomadaires des S\'{e}ances de l'Acad\'{e}mie des
              Sciences. S\'{e}ries A et B},
    VOLUME = {272},
      YEAR = {1971}}

@book {Berstel&Reutenauer:2011,
    AUTHOR = {Berstel, Jean and Reutenauer, Christophe},
     TITLE = {Noncommutative rational series with applications},
    SERIES = {Encyclopedia Math. Appl.},
    VOLUME = {137},
 PUBLISHER = {Cambridge University Press},
   ADDRESS = {Cambridge},
      YEAR = {2011},
     PAGES = {xiv+248},
      ISBN = {978-0-521-19022-0},
}

@article{Bertrand-Mathis:1986,
	   AUTHOR = {Bertrand-Mathis, Anne},
     TITLE = {D\'{e}veloppement en base {$\theta$}; r\'{e}partition modulo un de la
              suite {$(x\theta^n)_{n\geq 0}$}; langages cod\'{e}s et
              {$\theta$}-shift},
   JOURNAL = {Bull. Soc. Math. France},
  FJOURNAL = {Bulletin de la Soci\'{e}t\'{e} Math\'{e}matique de France},
    VOLUME = {114},
      YEAR = {1986},
    NUMBER = {3},
     PAGES = {271--323}}

@article {Bertrand-Mathis:1989,
    AUTHOR = {Bertrand-Mathis, Anne},
     TITLE = {Comment \'ecrire les nombres entiers dans une base qui n'est
              pas enti\`ere},
   JOURNAL = {Acta Math. Hungar.},
  FJOURNAL = {Acta Mathematica Hungarica},
    VOLUME = {54},
      YEAR = {1989},
    NUMBER = {3-4},
     PAGES = {237--241},
      ISSN = {0236-5294,1588-2632},
       DOI = {10.1007/BF01952053},
}

@incollection {Bruyere&Hansel:1997,
    AUTHOR = {Bruy\`ere, V\'{e}ronique and Hansel, Georges},
     TITLE = {Bertrand numeration systems and recognizability},
      NOTE = {Latin American Theoretical INformatics (Valpara\'{\i}so, 1995)},
   JOURNAL = {Theoret. Comput. Sci.},
  FJOURNAL = {Theoretical Computer Science},
    VOLUME = {181},
      YEAR = {1997},
    NUMBER = {1},
     PAGES = {17--43},
       DOI = {10.1016/S0304-3975(96)00260-5},}

@incollection{Buchi:1990,
	author = {Büchi, J. Richard},
	address = {New York, NY},
	title = {On a Decision Method in Restricted Second Order Arithmetic},
	booktitle = {The {Collected} {Works} of {J}. {Richard} {Büchi}},
	publisher = {Springer},
	editor = {Mac Lane, Saunders and Siefkes, Dirk},
	year = {1990},
	pages = {425--435}}

@article{Caalim&Demegillo:2020,
	author = {Caalim, Jonathan and Demegillo, Shiela},
	title = {Beta {Cantor} series expansion and admissible sequences},
	journal = {Acta Polytechnica},
	volume = {60},
	number = {3},
	year = {2020},
	pages = {214--224},
    doi = {10.14311/AP.2020.60.0214}}

@incollection {Charlier:2023,
    AUTHOR = {Charlier, {\'E}milie},
     TITLE = {Alternate base numeration systems},
 BOOKTITLE = {Combinatorics on words},
    SERIES = {Lecture Notes in Comput. Sci.},
    VOLUME = {13899},
     PAGES = {14--34},
 PUBLISHER = {Springer},
   ADDRESS = {Cham},
      YEAR = {2023},
       DOI = {10.1007/978-3-031-33180-0\_2}}

@article {Charlier&Cisternino:2021,
    AUTHOR = {Charlier, {\'E}milie and Cisternino, C\'elia},
     TITLE = {Expansions in {C}antor real bases},
   JOURNAL = {Monatsh. Math.},
  FJOURNAL = {Monatshefte f\"ur Mathematik},
    VOLUME = {195},
      YEAR = {2021},
    NUMBER = {4},
     PAGES = {585--610},
       DOI = {10.1007/s00605-021-01598-6},
}

@article {Charlier&Cisternino&Masakova&Pelantova:2023,
    AUTHOR = {Charlier, {\'E}milie and Cisternino, C\'elia and Mas\'akov\'a,
              Zuzana and Pelantov\'a, Edita},
     TITLE = {Spectrum, algebraicity and normalization in alternate bases},
   JOURNAL = {J. Number Theory},
  FJOURNAL = {Journal of Number Theory},
    VOLUME = {249},
      YEAR = {2023},
     PAGES = {470--499},
      ISSN = {0022-314X,1096-1658},
       DOI = {10.1016/j.jnt.2023.02.012},
}

@article {Charlier&Cisternino&Kreczman:2024,
    AUTHOR = {Charlier, {\'E}milie and Cisternino, C\'elia and Kreczman,
              Savinien},
     TITLE = {On periodic alternate base expansions},
   JOURNAL = {J. Number Theory},
  FJOURNAL = {Journal of Number Theory},
    VOLUME = {254},
      YEAR = {2024},
     PAGES = {184--198},
      ISSN = {0022-314X,1096-1658},
       DOI = {10.1016/j.jnt.2023.07.008},
}

@incollection {Charlier&Cisternino&Stipulanti:2022,
    AUTHOR = {Charlier, {\'E}milie and Cisternino, C\'{e}lia and Stipulanti, Manon},
     TITLE = {A full characterization of {B}ertrand numeration systems},
 BOOKTITLE = {Developments in language theory},
    SERIES = {Lecture Notes in Comput. Sci.},
    VOLUME = {13257},
     PAGES = {102--114},
 PUBLISHER = {Springer},
   ADDRESS = {Cham},
      YEAR = {2022},
       DOI = {10.1007/978-3-031-05578-2\_8}}

@article {Dajani&DeVries&Komornik&Loreti:2011,
    AUTHOR = {Dajani, Karma and de Vries, Martijn and Komornik, Vilmos and
              Loreti, Paola},
     TITLE = {Optimal expansions in non-integer bases},
   JOURNAL = {Proc. Amer. Math. Soc.},
  FJOURNAL = {Proceedings of the American Mathematical Society},
    VOLUME = {140},
      YEAR = {2012},
    NUMBER = {2},
     PAGES = {437--447},
      ISSN = {0002-9939,1088-6826},
       DOI = {10.1090/S0002-9939-2011-11226-7},
}

@article {Dumont&Thomas:1989,
    AUTHOR = {Dumont, Jean-Marie and Thomas, Alain},
     TITLE = {Syst\`{e}mes de num\'{e}ration et fonctions fractales relatifs aux
              substitutions},
   JOURNAL = {Theoret. Comput. Sci.},
  FJOURNAL = {Theoretical Computer Science},
    VOLUME = {65},
      YEAR = {1989},
    NUMBER = {2},
     PAGES = {153--169},
       DOI = {10.1016/0304-3975(89)90041-8}}

@article {Fiorenza&Vincenzi:2011,
    AUTHOR = {Fiorenza, Alberto and Vincenzi, Giovanni},
     TITLE = {Limit of ratio of consecutive terms for general order-{$k$}
              linear homogeneous recurrences with constant coefficients},
   JOURNAL = {Chaos Solitons Fractals},
  FJOURNAL = {Chaos, Solitons \& Fractals},
    VOLUME = {44},
      YEAR = {2011},
    NUMBER = {1-3},
     PAGES = {145--152},
       DOI = {10.1016/j.chaos.2011.01.003}}

@article {Fraenkel:1985,
    AUTHOR = {Fraenkel, Aviezri S.},
     TITLE = {Systems of numeration},
   JOURNAL = {Amer. Math. Monthly},
  FJOURNAL = {American Mathematical Monthly},
    VOLUME = {92},
      YEAR = {1985},
    NUMBER = {2},
     PAGES = {105--114},
      ISSN = {0002-9890,1930-0972},
       DOI = {10.2307/2322638},
}

@article {Frougny:1992,
    AUTHOR = {Frougny, Christiane},
     TITLE = {Representations of numbers and finite automata},
   JOURNAL = {Math. Systems Theory},
  FJOURNAL = {Mathematical Systems Theory. An International Journal on
              Mathematical Computing Theory},
    VOLUME = {25},
      YEAR = {1992},
    NUMBER = {1},
     PAGES = {37--60},
      ISSN = {0025-5661},
       DOI = {10.1007/BF01368783},
}

@incollection {Frougny&Sakarovitch:2010,
    AUTHOR = {Frougny, Christiane and Sakarovitch, Jacques},
     TITLE = {Number representation and finite automata},
 BOOKTITLE = {Combinatorics, automata and number theory},
    SERIES = {Encyclopedia Math. Appl.},
    VOLUME = {135},
     PAGES = {34--107},
 PUBLISHER = {Cambridge Univ. Press},
   ADDRESS = {Cambridge},
      YEAR = {2010},
      ISBN = {978-0-521-51597-9},
}

@incollection {Frougny&Solomyak:1996,
    AUTHOR = {Frougny, Christiane and Solomyak, Boris},
     TITLE = {On representation of integers in linear numeration systems},
 BOOKTITLE = {Ergodic theory of {${\bf Z}^d$} actions ({W}arwick,
              1993--1994)},
    SERIES = {London Math. Soc. Lecture Note Ser.},
    VOLUME = {228},
     PAGES = {345--368},
 PUBLISHER = {Cambridge Univ. Press},
   ADDRESS = {Cambridge},
      YEAR = {1996},
      ISBN = {0-521-57688-1},
       DOI = {10.1017/CBO9780511662812.014},
}

@article {Hollander:1998,
    AUTHOR = {Hollander, Michael},
     TITLE = {Greedy numeration systems and regularity},
   JOURNAL = {Theory Comput. Syst.},
  FJOURNAL = {Theory of Computing Systems},
    VOLUME = {31},
      YEAR = {1998},
    NUMBER = {2},
     PAGES = {111--133},
      ISSN = {1432-4350,1433-0490},
       DOI = {10.1007/s002240000082},
}

@incollection {Kreczman&Labbe&Stipulanti:2025,
    AUTHOR = {Kreczman, Savinien and Labb\'{e}, S\'{e}bastien and Stipulanti, Manon},
     TITLE = {A succinct study of positionality for {D}umont-{T}homas
              numeration systems},
 BOOKTITLE = {Combinatorics on words},
    SERIES = {Lecture Notes in Comput. Sci.},
    VOLUME = {15729},
     PAGES = {179--191},
 PUBLISHER = {Springer},
   ADDRESS = {Cham},
      YEAR = {2025},
       DOI = {10.1007/978-3-031-97548-6\_16}}

@article {Lecomte&Rigo:2001,
    AUTHOR = {Lecomte, Pierre and Rigo, Michel},
     TITLE = {Numeration systems on a regular language},
   JOURNAL = {Theory Comput. Syst.},
  FJOURNAL = {Theory of Computing Systems},
    VOLUME = {34},
      YEAR = {2001},
    NUMBER = {1},
     PAGES = {27--44},
       DOI = {10.1007/s002240010014}}

@incollection {Lecomte&Rigo:2010,
    AUTHOR = {Lecomte, Pierre and Rigo, Michel},
     TITLE = {Abstract numeration systems},
 BOOKTITLE = {Combinatorics, automata and number theory},
    SERIES = {Encyclopedia Math. Appl.},
    VOLUME = {135},
     PAGES = {108--162},
 PUBLISHER = {Cambridge Univ. Press},
   ADDRESS = {Cambridge},
      YEAR = {2010}}

@article {Loraud:1995,
    AUTHOR = {Loraud, Nathalie},
     TITLE = {{$\beta$}-shift, syst\`emes de num\'{e}ration et automates},
   JOURNAL = {J. Th\'{e}or. Nombres {B}ordeaux},
  FJOURNAL = {Journal de Th\'{e}orie des Nombres de {B}ordeaux},
    VOLUME = {7},
      YEAR = {1995},
    NUMBER = {2},
     PAGES = {473--498},
       DOI = {10.5802/jtnb.153},}

@book {Lothaire:2002,
    AUTHOR = {Lothaire, M.},
     TITLE = {Algebraic combinatorics on words},
    SERIES = {Encyclopedia Math. Appl.},
    VOLUME = {90},
 PUBLISHER = {Cambridge University Press},
   ADDRESS = {Cambridge},
      YEAR = {2002},
     PAGES = {xiv+504},
      ISBN = {0-521-81220-8},
       DOI = {10.1017/CBO9781107326019},
}

@misc{Mousavi:2021,
	title = {Automatic Theorem Proving in {Walnut}},
	doi = {10.48550/arXiv.1603.06017},
	publisher = {arXiv},
	author = {Mousavi, Hamoon},
	month = may,
	year = {2021},
	note = {arXiv:1603.06017}}

@article {Parry:1960,
    AUTHOR = {Parry, W.},
     TITLE = {On the {$\beta $}-expansions of real numbers},
   JOURNAL = {Acta Math. Acad. Sci. Hungar.},
  FJOURNAL = {Acta Mathematica. Academiae Scientiarum Hungaricae},
    VOLUME = {11},
      YEAR = {1960},
     PAGES = {401--416},
       DOI = {10.1007/BF02020954}}

@article {Paun&Salomaa:1995,
    AUTHOR = {P\u{a}un, Gheorghe and Salomaa, Arto},
     TITLE = {Thin and slender languages},
   JOURNAL = {Discrete Appl. Math.},
  FJOURNAL = {Discrete Applied Mathematics. The Journal of Combinatorial
              Algorithms, Informatics and Computational Sciences},
    VOLUME = {61},
      YEAR = {1995},
    NUMBER = {3},
     PAGES = {257--270},
      ISSN = {0166-218X,1872-6771},
       DOI = {10.1016/0166-218X(94)00014-5},
}

@article {Renyi:1957,
    AUTHOR = {R\'enyi, A.},
     TITLE = {Representations for real numbers and their ergodic properties},
   JOURNAL = {Acta Math. Acad. Sci. Hungar.},
  FJOURNAL = {Acta Mathematica. Academiae Scientiarum Hungaricae},
    VOLUME = {8},
      YEAR = {1957},
     PAGES = {477--493},
      ISSN = {0001-5954,1588-2632},
       DOI = {10.1007/BF02020331},
}

@book {Salomaa&Soittola:1978,
    AUTHOR = {Salomaa, Arto and Soittola, Matti},
     TITLE = {Automata-theoretic aspects of formal power series},
    SERIES = {Texts and Monographs in Computer Science},
 PUBLISHER = {Springer-Verlag},
   ADDRESS = {New York},
      YEAR = {1978},
     PAGES = {x+171},
      ISBN = {0-387-90282-1},
}

@article {Schmidt:1980,
    AUTHOR = {Schmidt, Klaus},
     TITLE = {On periodic expansions of {P}isot numbers and {S}alem numbers},
   JOURNAL = {Bull. London Math. Soc.},
  FJOURNAL = {The Bulletin of the London Mathematical Society},
    VOLUME = {12},
      YEAR = {1980},
    NUMBER = {4},
     PAGES = {269--278},
      ISSN = {0024-6093,1469-2120},
       DOI = {10.1112/blms/12.4.269},
}

@article {Shallit:1994,
    AUTHOR = {Shallit, Jeffrey},
     TITLE = {Numeration systems, linear recurrences, and regular sets},
   JOURNAL = {Inform. and Comput.},
  FJOURNAL = {Information and Computation},
    VOLUME = {113},
      YEAR = {1994},
    NUMBER = {2},
     PAGES = {331--347},
      ISSN = {0890-5401,1090-2651},
       DOI = {10.1006/inco.1994.1076},
}

@book {Shallit:2022,
    AUTHOR = {Shallit, Jeffrey},
     TITLE = {The logical approach to automatic sequences---exploring
              combinatorics on words with {\tt {W}alnut}},
    SERIES = {London Math. Soc. Lecture Note Ser.},
    VOLUME = {482},
 PUBLISHER = {Cambridge University Press},
   ADDRESS = {Cambridge},
      YEAR = {2022},
     PAGES = {xv+358},
      ISBN = {978-1-108-74524-6},
       doi = {10.1017/9781108775267},
}

@article {Zeckendorf:1972,
    AUTHOR = {Zeckendorf, E.},
     TITLE = {Repr\'{e}sentation des nombres naturels par une somme de nombres
              de {F}ibonacci ou de nombres de {L}ucas},
   JOURNAL = {Bull. Soc. Roy. Sci. Li\`ege},
  FJOURNAL = {Bulletin de la Soci\'{e}t\'{e} Royale des Sciences de Li\`ege},
    VOLUME = {41},
      YEAR = {1972},
     PAGES = {179--182}}

\end{document}